\documentclass[a4paper]{amsart}
\usepackage{url}
\usepackage{graphicx}
\usepackage[all]{xy}
\usepackage[mathscr]{eucal}
\usepackage{amsmath,amssymb,amsfonts}
\newtheorem{theorem}{Theorem}[section]
\newtheorem{lemma}[theorem]{Lemma}
\newtheorem{corollary}[theorem]{Corollary}
\newtheorem{example}[theorem]{Example}
\newtheorem{examples}[theorem]{Examples}
\newtheorem{proposition}[theorem]{Proposition}
\newtheorem{remark}[theorem]{Remark}
\newtheorem{conjecture}[theorem]{Conjecture}

\newtheorem{notation}[theorem]{Notation}
\usepackage{stackengine}
\usepackage{xcolor}
\usepackage[all]{xy}
\usepackage{tikz}
\usetikzlibrary{decorations.markings}
\tikzset{->-/.style={decoration={
  markings,
  mark=at position .5 with {\arrow{stealth}}},postaction={decorate}}}

\newcommand{\degree}{degree}

\title[Generalizations of the Thompson groups]{Higher dimensional generalizations of the Thompson groups via higher rank graphs}

\author{Mark V. Lawson}
\address{Mark V. Lawson, Department of Mathematics
and the
Maxwell Institute for Mathematical Sciences,
Heriot-Watt University,
Riccarton,
tight
Edinburgh EH14 4AS,
UNITED KINGDOM.}
\email{m.v.lawson@hw.ac.uk}

\author{Aidan Sims}
\address{Aidan Sims,
School of Mathematics and Applied Statistics,
University of Wollongong,
NSW 2522, 
AUSTRALIA.}
\email{asims@uow.edu.au}

\author{Alina Vdovina}
\address{Alina Vdovina,
Department of Mathematics,
The City College of New York,
160 Convent Avenue,
New York, NY 10031,
USA.}
\email{avdovina@ccny.cuny.edu}

\begin{document}

\begin{abstract}
We construct a family of groups from suitable higher rank graphs which are higher dimensional generalizations of the Thompson groups. 
We introduce group invariants, inspired by the $K$-theory of $C^{*}$-algebras, and show that many of our groups are not isomorphic to the Brin-Thompson groups $nV$, when $n \geq 2$.
\end{abstract}
\maketitle

\section{Introduction}

Groups are the abstract versions of groups of bijections.
But how can bijections themselves be constructed?
We can start with something simpler: namely, partial bijections.
However, the union of two partial bijections need not itself be a partial bijection;
if it is, we say that the partial bijections are {\em compatible}.
Thus, we could try to make bijections from suitable sets of compatible partial bijections.
This raises the question of where we could find the partial bijections.
There are many possible sources, but one is to use cancellative monoids.
If $S$ is a cancellative monoid then multiplication on the left by an element $a$ of $S$
is a partial bijection.
More generally, we can replace the cancellative monoid $S$ by a cancellative category $C$, regarded as a partial algebraic structure,
and left multiplication by an arrow $a$ now has a partial domain of definition
being all those arrows that can be composed on the right with $a$.
The abstract theory of partial bijections is inverse semigroup theory (there is a zero which represents the empty partial bijection),
so the above line of reasoning suggests that we might construct  groups from inverse semigroups, and construct inverse semigroups in turn from cancellative categories.
This, in a nutshell, is the idea of this paper.
We now flesh out the details.

In this paper, we shall show how to construct a family of groups (which, it will transpire, are topological full groups) from suitable higher rank graphs.
The groups can be considered to be infinite analogues of finite direct products of finite symmetric groups.
Despite their name, higher rank graphs are in fact a class of cancellative categories;
they generalize free categories,
and arise naturally in constructing important examples of $C^{\ast}$-algebras \cite{KP}.
Thus, we shall be constructing groups from a family of cancellative categories.
Specifically, in this paper, we generalize \cite{LV2019b} from the monoid case to the category case.
Our approach is closely related to the one adopted in \cite{DM}.

\begin{remark}{\em In this paper, we treat categories as being algebraic structures generalizing monoids.
Thus, our categories are always assumed to be sets except when we talk about categories of structures. 
In particular, category theory {\em per se} plays almost no role in this paper.
For us, a category is a `monoid with many identities.' For this reason, we have used algebraic terminology throughout.}
\end{remark}

Our application of cancellative categories will make essential use of the pioneering work of Jack Spielberg \cite{JS2014, JS2018}.
The background from inverse semigroup theory needed to read this paper is outlined in Section 2;
the construction of groups from cancellative categories is described in Section 3;
and the application of the results from Section~3 to that of constructing groups from higher rank graphs
is the subject of Section 4.

We now want to understand in more detail the structure of the groups we have constructed from higher rank graphs.
To understand a group properly usually requires the group act on a suitable geometric structure.
In the case of the classical groups, they arise as groups of units of matrix rings
and therefore act on geometric structures constructed from vector spaces.
In this paper, we show, in particular, that the groups we construct arise as groups of units of Boolean inverse monoids.
Such inverse monoids have ring-theoretic characteristics, without themselves being rings \cite{Wehrung}.
In addition, their sets of idempotents form Boolean algebras on which the group acts.
Showing that a group is a group of units of a Boolean inverse monoid therefore brings with it geometric information.
This is the subject of Section 5 and Section 6, with Section 7 describing the structure of the Boolean inverse monoid
and relating it to the structure of higher rank graphs.

The groups we have constructed also arise as the topological full groups of certain kinds of topological groupoids.
This is proved in Section 8.
Specifically, we apply the theory of noncommuative Stone duality, 
a summary of which can be found in \cite{Lawson2022},
to the Boolean inverse monoids we have constructed.
Under non-commutative Stone duality, Boolean inverse monoids are related to a class of \'etale groupoids
and the groups of units of the Boolean inverse monoid are the topological full groups of the associated \'etale groupoids.
In the countably infinite case, the \'etale groupoids associated to our Boolean inverse monoids are
Hausdorff \'etale topological groupoids which are effective, minimal and have a space of identities homeomorphic to the Cantor space. 
This implies, in particular, that in this case our group is isomorphic to a subgroup of the group of self-homeomorphisms of the Cantor space.

By \cite[Theorem~3.10]{Matui2015}, our group is therefore a complete invariant
for the corresponding \'etale groupoid. This suggests that both the homology of the groupoid, and the $K$-groups
of its $C^{\ast}$-algebra, should be interesting group-theoretically. 
This theme is developed in Section~9 and leads onto Section 10, where we construct some examples.
Specifically, using the invariants of Section 9, we show
that the corresponding groups are non-isomorphic as well. 
In particular, they are not isomorphic
to the known examples of groups $nV$ for $n\geq 2$.

\begin{center}
{\bf Some notation}
\end{center}

\begin{itemize}

\item $I(X,\tau)$ the inverse semigroup of all partial homeomorphisms between the open subsets of the topological space $(X,\tau)$.
When $X$ is endowed with the discrete topology, we simply denote this inverse semigroup by $I(X)$, the symmetric inverse monoid on the set $X$. Section~2.1.

\item $S^{e}$ the inverse semigroup of all elements $s \in S$ where both $s^{-1}s$ and $ss^{-1}$ are essential idempotents.
Section~2.1.

\item $\mathsf{RI}(C)$ the inverse semigroup of all bijective morphisms between the right ideals of the category $C$. Section~3.1.

\item $\mathsf{R}(C)$ the inverse semigroup of all bijective morphisms between the finitely generated right ideals of the finitely aligned category $C$. Section~3.1.

\item $\mathscr{G}(C) = \mathsf{R}(C)^{e}/\sigma$ the {\em group} constructed from the category $C$ which is finitely aligned and has only a finite number of identities. Section~3.1.

\item $\Sigma (C)$ the inverse hull of the cancellative category $C$. Section~3.2.

\item $\mathsf{P}(C)$ the inverse monoid of all bijective morphisms between finitely generated right ideals generated by codes
when $C$ is a strongly finitely aligned conical cancellative category with a finite number of identities which satisfies the condition (MC). Section~3.3.

\item $\mathsf{B}(C) = \mathsf{R}(C)/{\equiv}$, a Boolean inverse monoid with group of units isomorphic to $\mathscr{G}(C)$
when $C$ is a  strongly finitely aligned higher rank graph with a finite number of identities with no sources and is row finite. Section~5 and Section 6.

\item $\mathcal{G}(C)$ the {\em groupoid} associated with the higher rank graph $C$. Section~8.

\end{itemize}

\section{Background definitions}

This section is included for reference only.

\subsection{Inverse semigroups}

Our semigroups will usually have a zero and we write $e \perp f$ if $e$ and $f$ are idempotents such that $ef = 0$.
We say that $e$ and $f$ are {\em orthogonal}.
We refer the reader to \cite{Lawson1998} for the detailed theory of inverse semigroups but we recall some important definitions here.
Remember: inverse semigroup theory is the abstract theory of partial bijections;
the composition of partial bijections is always defined since there is a partial bijection defined on the empty set.
An {\em inverse semigroup} is a semigroup in which for each element $a$ there is a unique element, denoted by $a^{-1}$,
such that $a = aa^{-1}a$ and $a^{-1} = a^{-1}aa^{-1}$.
The elements $aa^{-1}$ and $a^{-1}a$ are always idempotents and we denote then by $\mathbf{d}(a) = a^{-1}a$
and $\mathbf{r}(a) = aa^{-1}$.
We write $a \perp b$, and say that $a$ and $b$ are {\em orthogonal}, if $\mathbf{d}(a) \mathbf{d}(b) = 0$ and $\mathbf{r}(a) \mathbf{r}(b) = 0$.
The set of idempotents in $S$ is denoted by $\mathsf{E}(S)$.
It is called the {\em semilattice of idempotents} of $S$.\footnote{This is really a meet-semilattice when we define $e \wedge f = ef$.}
If $X \subseteq S$, define $\mathsf{E}(X) = X \cap \mathsf{E}(S)$.
Observe that if $e$ is any idempotent and $a$ is any element then $aea^{-1}$ is an idempotent.
Thus the `conjugate of an idempotent is an idempotent'.
If $S$ is an inverse {\em monoid} its group of units is denoted by $\mathsf{U}(S)$.
A non-zero element $a$ of an inverse semigroup with zero is called an {\em infinitesimal} if $a^{2} = 0$.
Observe that $a$ is an infinitesimal if and only if $\mathbf{d}(a) \perp \mathbf{r}(a)$ by \cite[Lemma~2.5]{Lawson2017}.
An inverse semigroup $S$ is said to be {\em fundamental} if the only elements of $S$ that commute
with all the idempotents of $S$ are themselves idempotents.
The following is an easy consequence of \cite[Theorem~5.2.9]{Lawson1998}.

\begin{lemma}\label{lem:finite-fundamental}
If a fundamental inverse semigroup has a finite number of idempotents then it is itself finite.
\end{lemma}

Define the {\em natural partial order} on $S$ by $a \leq b$ if and only if $a = ba^{-1}a$.
It can be proved that with respect to this order, an inverse semigroup is partially ordered
but observe that $a \leq b$ implies that $a^{-1} \leq b^{-1}$.
An inverse semigroup is called {\em $E$-unitary} if $e \leq a$, where $e$ is an idempotent, implies that $a$ is an idempotent.

Define the {\em compatibility relation} $a \sim b$ precisely when $a^{-1}b$ and $ab^{-1}$ are both idempotents;
this relation is reflexive and symmetric, but not transitive in general.
If $a \sim b$ we say that $a$ and $b$ are {\em compatible}.
Observe that orthogonal elemenets are compatible.
A non-empty subset $X$ of an inverse semigroup is said to be {\em compatible}
if each pair of elements of $X$ is compatible.
Observe that if $a,b \leq c$ then $a \sim b$.
It follows that $a \sim b$ is a necessary condition for $a$ and $b$ to have a join $a \vee b$
with respect to the natural partial order.
If $a \perp b$ and $a \vee b$ exists we speak of an {\em orthogonal join}.
The compatibility relation plays an important role in this paper.
The following is
\cite[Lemma~1.4.11]{Lawson1998}
and
\cite[Lemma~1.4.12]{Lawson1998}.

\begin{lemma}\label{lem:compatibility-meets} In an inverse semigroup, we have the following:
\begin{enumerate}
\item  $a \sim b$
if and only if all of the following hold: $a \wedge b$ exists, $\mathbf{d}(a \wedge b) = \mathbf{d}(a)\mathbf{d}(b)$, and
 $\mathbf{r}(a \wedge b) = \mathbf{r}(a)\mathbf{r}(b)$.
 \item If $a \sim b$ then $a \wedge b = ab^{-1}b = bb^{-1}a = ba^{-1}a = aa^{-1}b$.
 \end{enumerate}
 \end{lemma}

Observe that the meet $a \wedge b$ may exist without $a$ and $b$ being compatible.
An inverse semigroup is called a {\em $\wedge$-semigroup} if each pair of elements has a meet with respect
to the natural partial order.
Inverse $\wedge$-semigroups were first studied in \cite{Leech}
and will play an important role in this paper.

\begin{remark}
{\em If $\theta \colon S \rightarrow T$ is a homomorphism between
inverse semigroups then it is not true in general that $\theta (a \wedge b) = \theta (a) \wedge \theta (b)$.
However, if $a \sim b$ then by Lemma~\ref{lem:compatibility-meets}, we have that $a \wedge b = ab^{-1}b$;
of course, in this case, we do have that  $\theta (a \wedge b) = \theta (a) \wedge \theta (b)$ because $a \wedge b$
can be described purely algebraically.}
\end{remark}

Let $\rho$ be a congruence on a semigroup $S$;
it is said to be {\em idempotent-pure} if $a \, \rho \, e$, where $e$ is an idempotent, implies that $a$ is an idempotent;
it is said to be {\em $0$-restricted} if $a \, \rho \, 0$ implies that $a = 0$.
Let $\theta \colon S \rightarrow T$ be a homomorphism between semigroups;
it is said to be {\em $0$-restricted} if $\theta (a)$ is zero if and only if $a$ is zero;
it is said to be {\em idempotent-pure} if $\theta (a)$ is an idempotent if and only if $a$ is an idempotent.
The proofs of the following are straightforward from the definitions.

\begin{lemma}\label{lem:idpt-pure-characterizations}
Let $S$ be an inverse semigroup:
\begin{enumerate}
\item The congruence $\rho$ is idempotent-pure if and only if $\rho \, \subseteq \, \sim$.
\item The homomorphism $\theta \colon S \rightarrow T$ is idempotent-pure if and only if
$\theta (a) \sim \theta (b)$ implies that $a \sim b$.
\end{enumerate}
\end{lemma}

One way, of course, to construct a group from an inverse monoid is to consider its group of units 
but there is an alternative approach that will be important to this paper.
Let $S$ be an inverse semigroup.
There is a congruence $\sigma$ defined on $S$ such that $S/\sigma$ is a group
and if $\rho$ is any congruence on $S$ such that $S/\rho$ is a group then $\sigma \subseteq \rho$.
Thus $\sigma$ is the {\em minimum group congruence}.
In fact, $s \, \sigma \, t$ if and only if there exists $z \leq s,t$.
See \cite[Section 2.4]{Lawson1998}.
The following was proved as \cite[Theorem 2.4.6]{Lawson1998}.

\begin{proposition}\label{prop:four} Let $S$ be an inverse semigroup.
Then $S$ is $E$-unitary if and only if $\sigma \, = \, \sim$.
\end{proposition}

The following is well-known and easy to check.

\begin{lemma}\label{lem:five} Let $S$ be an $E$-unitary inverse semigroup.
Then for $a,b \in S$ we have $a \sim b$ if and only if $ab^{-1}b = ba^{-1}a$.
\end{lemma}

Intuitively, the above lemma says that the partial bijections $a$ and $b$ are identified
precisely when they agree on the intersection of their domains of definition.

\begin{example}{\em This example illustrates the idea of this paper. 
Groups often arise as groups of symmetries but, sometimes, how they arise is more elusive.
For example, the {\em abstract commensurator} of a group $G$ is the set of all isomorphisms between
subgroups of finite index factored out by the equivalence that identifies two such isomorphisms if they agree on
a subgroup of finite index.
This forms a group $\mbox{Comm}(G)$, called the {\em abstract commensurator} of $G$ \cite{Nek2002}.
In fact, this group is best understood using inverse semigroup theory.
The set, $\Omega (G)$, of all isomorphisms between subgroups of finite index is an inverse semigroup.
The group $\mbox{Comm}(G)$ is then  $\Omega (G)/\sigma$ where $\sigma$ is the minimum group congruence on  $\Omega (G)$.
The elements of $\mbox{Comm}(G)$ are `hidden symmetries' to use the terminology of Farb and Weinberger \cite{FW2004}.
The elements of $\Omega (G)$ are, in some sense, `large'.
}
\end{example}

We now describe an analogous procedure to the one described above for constructing a group
from an inverse semigroup (of partial bijections).
A special case was used in \cite{Lawson2007} for constructing the classical Thompson groups from free monoids.

Let $S$ be an inverse semigroup (of partial isomorphisms, for example).
Let $S' \subseteq S$ be an inverse subsemigroup whose elements are, in some sense, large;
whatever this might mean, we require that $S'$ does not contain a zero.
Then we obtain a group $S'/\sigma$.
We regard the elements of $S'/\sigma$ as hidden symmetries of the structure that gives rise to $S$.
{\em We now define what `large' means in the context of this paper.}
A non-zero idempotent $e$ of an inverse semigroup $S$ is said to be {\em essential} if $ef \neq 0$ for all non-zero
idempotents $f$ of $S$.
An element $s$ is said to be {\em essential} if both $s^{-1}s$ and $ss^{-1}$ are essential.
Denote by $S^{e}$ the set of all essential elements of $S$.
It follows by \cite[Lemma~4.2]{Lawson2007}, that $S^{e}$ is an inverse semigroup (without zero).
We therefore expect the group $S^{e}/\sigma$ to be interesting.
This will be the basis of our construction of a group from an inverse semigroup:
\begin{itemize}
\item We begin with an inverse semigroup $S$.
\item We describe its inverse semigroup of essential elements $S^{e}$.
\item We then construct the group $S^{e}/\sigma$.
\end{itemize}

\begin{remark}
{\em Let $S$ be an inverse semigroup with zero.
We try to axiomatize our notion of `large'.
A subset $\mathsf{L} \subseteq \mathsf{E}(S)$ is said to consist of {\em large elements} if the following properties hold:
\begin{enumerate} 
\item $0 \notin \mathsf{L}$.
\item If $e,f \in \mathsf{L}$ then $e \wedge f \in \mathsf{L}$.
\item If $a$ is an element such that $\mathbf{d}(a), \mathbf{r}(a) \in \mathsf{L}$
and $e \leq \mathbf{d}(a)$ where $e \in \mathsf{L}$ then $\mathbf{r}(ae) \in \mathsf{L}$.
\end{enumerate}
Define $S^{\mathsf{L}}$ to consist of all elements $a \in S$ such that $\mathbf{d}(a), \mathbf{r}(a) \in \mathsf{L}$.
Then $S^{\mathsf{L}}$ is an inverse subsemigroup of $S$.\footnote{Our thanks to Alex Martin for some discussions on this subject.}}
\end{remark}

Inverse semigroups arise naturally from topological spaces.\footnote{In fact, this is the origin of inverse semigroups as the abstract manifestations of pseudogroups of tranformations.}
If $(X,\tau)$ is a topological space, then the set $I(X,\tau)$ of all
homeomorphisms between the open subsets of $X$ is an inverse monoid.
The elements of $I(X,\tau)$ are called {\em partial homeomorphisms}.\footnote{The partial homeomorphism defined on the empty set is the zero for this semigroup,
so the multiplication is everywhere defined.}
If $\tau$ is the discrete topology we just write $I(X)$ instead of $I(X,\tau)$
and call it the {\em symmetric inverse monoid} on $X$.
If $A$ is a subset of $X$ then the identity function defined on $A$ is denoted by $1_{A}$.

\subsection{Posets}

We need a little notation from the theory of posets.
Let $(X,\leq)$ be a poset.
If $A \subseteq X$ then $A^{\uparrow}$ is the set of all elements of $X$ above some element of $A$
and $A^{\downarrow}$ is the set of all elements of $X$ below some element of $A$.
If $A = \{a\}$ we write $a^{\uparrow}$ instead of $\{a\}^{\uparrow}$, and $a^{\downarrow}$ instead of $\{a\}^{\downarrow}$.
If $A = A^{\downarrow}$ we say that $A$ is an {\em order ideal}.

The complement of an element $e$ of a Boolean algebra is denoted by $\bar{e}$.

\subsection{Categories}

We regard a category as a generalized monoid.
Thus, the set of identities of the category $C$, denoted by $C_{o}$, is a subset of $C$ and there are two maps $\mathbf{d}, \mathbf{r} \colon C \rightarrow C_{o}$,
called, respectively, {\em domain} and {\em codomain}.
The elements of the category $C$ are called {\em arrows}, and are such that  $\mathbf{r}(a)\stackrel{a}{\longleftarrow} \mathbf{d}(a)$.
In this paper, the product $ab$ is defined precisely when $\mathbf{d}(a) = \mathbf{r}(b)$.
If $A$ and $B$ are subsets of a category $C$ then $AB$ is that subset of $C$ consisting of all products $ab$ where $a \in A$,
$b \in B$ and $ab$ is defined. 
It could, of course, be empty.
Singleton sets $\{a\}$ will be denoted, simply, by $a$.
A category $C$ is said to be {\em cancellative} if $ab = ac$ implies that $b = c$,
and $ba = ca$ implies that $b = c$.\footnote{Categorically, every element of the category is monic and epic but it is the cancellation properties
--- as in the monoid case --- which come to the fore.}
An arrow $x$ is {\em invertible} if there is an arrow $y$ such that $xy$ and $yx$ are identities.
Clearly, every identity is invertible.
A category in which the identities are the only invertible arrows is said to be {\em conical};
we have adopted this odd terminology from the theory of monoids.
Let $C$ be a category and let $a,b \in C$.
We say that $a$ and $b$ are {\em independent} if $aC \cap bC = \varnothing$;
otherwise, they are said to be {\em dependent}.
Thus $a$ and $b$ dependent means that there are elements $u,v \in C$ such that $au = bv$.\footnote{The terms `comparable' and `incomparable' 
were used in \cite{LV2019b}. Strictly speaking, we should say `dependent on the right'
and `independent on the right' but we only work `on the right' in this paper, anyway.}

\subsection{Distributive and Boolean inverse monoids}

An inverse monoid is said to be {\em distributive} if each compatible pair of elements 
has a join and multiplication distributes over such joins.
A {\em morphism} of distributive inverse monoids is a monoid morphism
which is required to preserve compatible joins.
If $X \subseteq S$, where $S$ is a distibutive inverse monoid, 
define $X^{\vee}$ to be the set of all joins of finite compatible subsets of $X$.
A distributive inverse monoid is said to be {\em Boolean} if its semilattice of idempotents
forms a Boolean algebra with respect to the natural partial order.

Let $I$ be a semigroup ideal of an inverse semigroup.
Let $a \in I$. Then $aa^{-1} \in I$ since $I$ is a semigroup ideal.
On the other hand if $a^{-1}a \in I$ then $a \in I$ since $a = a(a^{-1}a)$.
It follows that non-trivial ideals of inverse semigroups always contain idempotents.
If $a \in I$ then $ae \in I$ for all $e \in \mathsf{E}(S)$.
Thus, ideals are always order ideals.

A semigroup ideal $I$ of a Boolean inverse semigroup is said to be an {\em additive ideal}
if $a,b \in I$ and $a \sim b$ implies that $a \vee b \in I$.
If $I$ is a semigroup ideal of a Boolean inverse semigroup,
then $I^{\vee}$ is an additive ideal.

Let $S$ be a Boolean inverse monoid.
If $X \subseteq S$ define $\mathbf{d}(X) = \{\mathbf{d}(x) \colon x \in X\}$ and $\mathbf{r}(X) = \{\mathbf{r}(x) \colon x \in X\}$.
Let $e$ and $f$ be non-zero idempotents in $S$.
We write $e \preceq f$, and say there is a {\em pencil} from $e$ to $f$, if there is a finite set $X$ of $S$ such that $e = \bigvee \mathbf{d}(X)$
and $\bigvee \mathbf{r}(X) \leq f$.
The following was first proved in \cite{Lenz} and also in \cite[Lemma~4.1]{Lawson2016}.

\begin{lemma}\label{lem:cake} Let $S$ be a Boolean inverse monoid.
Let $e$ be an idempotent in the additive ideal $I$ and let $f$ be any idempotent.
Then $f \in I$ if and only if $f \preceq e$.
\end{lemma}

A Boolean inverse monoid is called {\em $0$-simplifying} if it has no non-trivial additive ideals.
A Boolean inverse monoid that is both fundamental and $0$-simplifying is called {\em simple}.\footnote{It is important to remember that this means `simple as a Boolean inverse monoid'. The term `simple' is
used in general semigroup theory with the meaning of no `non-trivial ideals'. Our terminology is natural since if $S$ is a simple Boolean inverse monoid
and there is a non-trivial surjective morphism to another Boolean inverse monoid then this morphism must be an isomorphism.}

For the following, see \cite{Lawson2022}.
Let $S$ be an inverse monoid.
A non-empty subset $A \subseteq S$ is a {\em filter} if $A = A^{\uparrow}$
and if $a,b \in A$ there exists $c \in A$ such that $c \leq a,b$.
It is {\em proper} if it does not contain $0$.
A proper filter $P$ of a distributive inverse monoid is said to be {\em prime} if $\bigvee_{i=1}^{m} s_{i} \in P$ implies that $s_{i} \in P$ for some $i$. 
A maximal proper filter is called an {\em ultrafilter}.
If $F$ is a filter in $E$ we write $e \wedge F \neq 0$ to mean that $e \wedge f \neq 0$ for each element $f \in F$.
The following was proved as \cite[Lemma 12.3]{Exel}.

\begin{lemma}\label{lem:uf-exel} Let $F$ be a filter in the meet-semilattice $E$.
Suppose that $e \wedge F \neq 0$ implies that $e \in F$.
Then $F$ is an ultrafilter.
\end{lemma}

The following result is proved using Zorn's Lemma.

\begin{lemma}\label{lem:proper-filters-uf}
In an inverse monoid, 
every proper filter is contained in an ultrafilter.
\end{lemma}

Let $A$ be a filter in the inverse monoid $S$.
Define $\mathbf{d}(A) = (A^{-1}A)^{\uparrow}$
and 
$\mathbf{r}(A) = (AA^{-1})^{\uparrow}$.
Then $\mathbf{d}(A)$ is a filter in $S$ which is also an inverse subsemigroup.
We call it an {\em identity filter}.
Furthermore, $A = (a \mathbf{d}(A))^{\uparrow}$ for any $a \in A$.
Clearly, $0 \in A$ if and only if $0 \in \mathbf{d}(A)$.
Also, $\mathsf{E}(\mathbf{d}(A))$ is a filter in $\mathsf{E}(S)$.
Similar remarks apply to $\mathbf{r}(A)$.
More generally, if $F$ is any identity filter and $\mathbf{d}(a) \in F$ then $(aF)^{\uparrow}$ is a filter
and if $\mathbf{r}(a) \in F$ then $(Fa)^{\uparrow}$ is a filter.
Observe that $A$ is an ultrafilter in $S$ if and only if  $\mathbf{d}(A)$ and $\mathbf{r}(A)$ are identity ultrafilters.

\section{Constructing groups from suitable categories}

In this section, we shall show how to construct a group from a category under certain assumptions on the category.
This part of the paper is related to  \cite{JS2014, JS2018}.
In Section~4, we shall specialize the results of this section to those categories that arise as higher rank graphs.

\subsection{Constructing a group from a suitable category}

Since categories generalize monoids, we can extend monoid-theoretic definitions to a category-theoretic setting.
Let $C$ be a category.
A subset $R \subseteq C$, possibly empty, is called a {\em right ideal} if
$RC \subseteq R$.
In category theory, such sets are called {\em sieves} but we prefer the terminology from algebra.
If $X$ is any subset of a category $C$ then $XC$ is a right ideal {\em generated by $X$}.
If the set $X$ is finite we say the right ideal is {\em finitely generated}.
If $a \in C$ we call $aC$ the  {\em principal right ideal} generated by $a$.

\begin{remark}
{\em Let $C$ be a category and let $R_{1}$ and $R_{2}$ be right ideals in $C$.
Then $R_{1} \cap R_{2}$ is always a right ideal as is $R_{1} \cup R_{2}$.
If $R_{1}$ and $R_{2}$ are finitely generated then so too is $R_{1} \cup R_{2}$.
However, if $R_{1}$ and $R_{2}$ are finitely generated there is no need for $R_{1} \cap R_{2}$ to be finitely generated.
}
\end{remark}

\begin{lemma}\label{lem:fgcat} Let $C$ be a category.
Then $C$ is finitely generated as a right ideal if and only if it has a finite number of identities.
\end{lemma}
\begin{proof} Suppose first that $C$ is finitely generated as a right ideal.
Then $C = XC$ where $X$ is a finite set.
Let $e \in C_{o}$ be an arbitrary identity of $C$.
Then $e \in XC$.
It follows that $e = xy$ for some $x \in X$ and $y \in C$.
Thus $e = \mathbf{r}(x)$.
We have proved that every identity of $C$ is the codomain of an element of $X$.
But $X$ is a finite set.
Thus the number of identities is finite.
Conversely, suppose that the number of identities is finite.
Then $C = C_{o}C$ and so $C$ is finitely generated as a right ideal.
\end{proof}

The following definition is due to \cite[Definition 3.1]{JS2014}.\\

\noindent
{\bf Definition. }We say that the category $C$ is {\em finitely aligned}
if $aC \cap bC$ is always finitely generated for any $a,b \in C$ (we include the possibility that it is empty).\\

\begin{remark}{\em The concept of finitely aligned was introduced in \cite{RSa}.
It seems to have been missed entirely within semigroup theory, since \cite{CP} makes no mention of it.
Perhaps, this is because good examples were lacking:
it is worth mentioning that free monoids are finitely aligned but here the interesection of two principal
right ideals is either empty or again a principal right ideal --- what is termed {\em singly aligned} in \cite{LV2019b}.
The semigroup version of finitely aligned has been introduced and explored in \cite{CG}.}
\end{remark}

\begin{lemma}\label{lem:fgright} Let $C$ be a category.
The intersection of any two finitely generated right ideals is finitely generated if and only if
$C$ is finitely aligned.
\end{lemma}
\begin{proof} It is clear that being finitely aligned is a necessary condition, we now show that it is sufficient.
Let $XC$ and $YC$ be two finitely generated right ideals.
Then $XC \cap YC = \bigcup_{x \in X, y \in Y} (xC \cap yC)$.
If $C$ is finitely aligned then each $xC \cap yC$ is finitely generated and, thus, so too is $XC \cap YC$.
\end{proof}

The proof of the following is easy. 

\begin{lemma}\label{lem:product-fa} 
Let $C$ and $D$ be finitely aligned categories.
Then $C \times D$ is finitely aligned.
\end{lemma}

A function $\theta \colon R_{1} \rightarrow R_{2}$ between two right ideals of a category $C$ satisfying
$\mathbf{d}(r) = \mathbf{d}(\theta (r))$ is called a {\em morphism}
if $\theta (rc) = \theta (r)c$ for all $r \in R_{1}$ and $c \in C$.
As usual, if $\alpha$ is a bijective morphism then $\alpha^{-1}$ is also a morphism.

\begin{remark}
{\em Morphisms as defined above arise naturally by left multiplication.
Let $C$ be a category and let $a \in C$.
Then we may define a function $\lambda_{a} \colon \mathbf{d}(a)C \rightarrow aC$ by $\lambda_{a}(x) = ax$. 
Observe that $\mathbf{d}(\lambda_{a}(x)) = \mathbf{d}(x)$.}
\end{remark}

\begin{lemma}\label{lem:right-ideal} Let $\theta \colon R_{1} \rightarrow R_{2}$ be a morphism between two right ideals.
Let $R \subseteq R_{1}$ be a right ideal.
Then $\theta (R)$ is a right ideal contained in $R_{2}$.
If $R$ is finitely generated then $\theta (R)$ is finitely generated.
\end{lemma}
\begin{proof} We prove first that $\theta (R)$ is a right ideal.
Let $c \in C$ be arbitrary and let $\theta (r) \in \theta (R)$.
Then $\theta (r)c = \theta (rc)$, since $\theta$ is a morphism.
But $rc \in R$.
It follows that $\theta (R)$ is a right ideal.
Suppose now that $R = XC$.
We prove that $\theta (XC) = \theta (X)C$.
We have that $\theta (XC) \subseteq \theta (X)C$ since $\theta (xc) = \theta (x)c$.
Conversely, $\theta (x)c = \theta (xc)$, since $\theta$ is a morphism.
In particular, it follows from this that if $X$ is finite then $\theta (X)$ is finite.
\end{proof}

\noindent
{\bf Definition. }Denote by $\mathsf{RI}(C)$ the set of all bijective morphisms between the right ideals of $C$
and by $\mathsf{R}(C)$ the set of all bijective morphisms between {\em finitely generated} right ideals of $C$.\\

\begin{proposition}\label{prop:inverse-monoid} Let $C$ be a category.
\begin{enumerate}
\item $\mathsf{RI}(C)$ is an inverse monoid.
\item If $C$ is finitely aligned and has a finite number of identities then  $\mathsf{R}(C)$ is an inverse submonoid of  $\mathsf{RI}(C)$.
\end{enumerate}
\end{proposition}
\begin{proof} (1) The whole category $C$ is a right ideal and so the identity function on $C$ is a bijective morphism,
and is an identity for $\mathsf{RI}(C)$.
The intersection of two right ideals is a right ideal.
It follows by Lemma~\ref{lem:right-ideal} that the composition of two bijective morphisms is a bijective morphism.
It is now clear that $\mathsf{RI}(C)$ is an inverse monoid.

(2) Suppose now that $C$ is finitely aligned and has a finite number of identities.
Then by Lemma~\ref{lem:fgcat}, the identity function on $C$ is the identity of $\mathsf{R}(C)$.
By Lemma~\ref{lem:fgright}, the intersection of any two finitely generated right ideals is a finitely generated right ideal.
By Lemma~\ref{lem:right-ideal}, it is now easy to see that $\mathsf{R}(C)$ is an inverse submonoid of $\mathsf{RI}(C)$.
\end{proof}

Let $C$ be a category.
We say that a non-empty right ideal $XC$ of $C$ is {\em essential} if it intersects every right ideal of $C$ in a non-empty set;
observe that is is enough to use principal right ideals.

\begin{lemma}\label{lem:one} Let $C$ be a category.
Then $1_{XC}$ is an essential idempotent in $\mathsf{RI}(C)$
if and only if 
$XC$ is an essential right ideal in $C$.
\end{lemma}
\begin{proof} Suppose that $1_{XC}$ is an essential idempotent in $\mathsf{RI}(C)$.
Let $a \in C$ be arbitrary.
Then $1_{aC}1_{XC} \neq \varnothing$.
But $1_{aC}1_{XC}$ is simply the identity function on the set $aC \cap XC$.
It follows that $aC \cap XC \neq \varnothing$.
The proof of the converse is similar.
\end{proof}

The following result tells us that if we want $\mathsf{R}(C)^{e}$
to be non-empty then we must assume that $C$ has a finite number of identities.

\begin{lemma}\label{lem:two} Let $C$ be a category that contains a finitely generated essential right ideal. 
Then $C$ contains only a finite number of identities.
\end{lemma}
\begin{proof} Let $XC$ be a finitely generated essential right ideal of $C$ and let $e \in C_{o}$ be an arbitrary identity.
Then $eC \cap XC$ is non-empty by assumption.
It follows that there exists $x \in X$ such that $eu = xv$ for some $u,v \in C$.
Thus $e = \mathbf{r}(x)$.
We have proved that each identity in $C$ is the codomain of an element of $X$.
But $X$ is a finite set.
It follows that the number of identities is finite.
\end{proof}

We can express whether a right ideal $XC$ is essential or not,
purely in terms of the properties of $X$.
A subset $X \subseteq C$ is said to be {\em large in $C$} if each $a \in C$ is dependent on an element of $X$.
A subset $X \subseteq aC$ is said to be {\em large in $aC$} if each $b \in aC$ is dependent on an element of $X$.

\begin{remark}
{\em What we call a `large' subset of $aC$ is called `exhaustive' in \cite{LS2010}}
\end{remark}

We now relate large subsets to essential right ideals.

\begin{lemma}\label{lem:zero} Let $C$ be a category with subset $X$.
Then $X$ is large if and only if $XC$ is essential.
\end{lemma}
\begin{proof} Suppose that $X$ is large.
We prove that $XC$ is essential.
Consider a principal right ideal $aC$.
Then since $X$ is large, we have that
$au = xv$ for some $x \in X$ and $u,v \in C$.
Thus $aC \cap xC \neq \varnothing$.
It follows that $XC$ is essential.
Conversely, suppose that $XC$ is essential.
Let $a \in C$ be arbitrary.
Then $aC \cap XC \neq \varnothing$.
Thus $au = xv$ for some $x \in X$ and $u,v \in C$.
It follows that $X$ is large.
\end{proof}

We can now define the group that we shall be interested in using the minimum group congruence $\sigma$. 
\vspace{5mm}
\begin{center}
\fbox{\begin{minipage}{15em}
{\bf Definition.} Let $C$ be a finitely aligned category with a finite number of identities.
Then
$$\mathscr{G}(C) = \mathsf{R}(C)^{e}/\sigma$$
is the  {\em group associated with $C$}.
\end{minipage}}
\end{center}
\vspace{5mm}

\subsection{The cancellative case}

We shall now revisit the construction of Section~3.1 under the additional assumption that $C$ is both cancellative and conical.

\begin{lemma}\label{lem:unique} Let $C$ be a category that is conical and cancellative.
Then $aC = bC$ if and only if $a = b$.
\end{lemma}
\begin{proof} Suppose that $aC = bC$.
Then $a = bx$ and $b = ay$ for some $x,y \in C$.
Thus $a = ayx$ and $b = bxy$.
By cancellation $xy$ and $yx$ are identities.
This implies $x$ and $y$ are invertible.
But $C$ is conical and so $x$ and $y$ are identities.
It follows that $a = b$.
The converse is immediate.
\end{proof}

The above result tells us that when the category is conical and cancellative,
we can identify principal right ideals by the unique elements that generate them.

We construct some special elements of the inverse monoid $\mathsf{R}(C)$. 
Let $a \in C$.
Define (as before) $\lambda_{a} \colon \mathbf{d}(a)C \rightarrow aC$ by $x \mapsto ax$.
In this case, it is easy to check that $\lambda_{a}$ is a bijective morphism,
the inverse of which we denote by $\lambda_{a}^{-1}$.
These maps are elements of the symmetric inverse monoid $I(C)$ and so generate
an inverse subsemigroup $\Sigma (C)$ called the {\em inverse hull} of $C$.
A product of the form $\lambda_{a}\lambda_{b}^{-1}$ is the empty function
unless $\mathbf{d}(a) = \mathbf{d}(b)$.
Clearly, $\lambda_{a}\lambda_{b}^{-1}$ is always a bijective morphism.

\begin{remark}
{\em The inverse hull $\Sigma (C)$ plays an important role in
this paper in answering structural questions about the inverse monoid $\mathsf{R}(C)$.}
\end{remark} 

\noindent
{\bf Definition. }We work in a conical, cancellative category $C$ where $a,b \in C$.
Let $\mathbf{d}(a) = \mathbf{d}(b)$.
Then $ab^{-1}$ is the bijective morphism from $bC$ to $aC$ given by $bx \mapsto ax$.
We call this, and the empty function, a {\em basic morphism}.
Observe that $\mathbf{d}(ab^{-1}) = bb^{-1}$ and $\mathbf{r}(ab^{-1}) = aa^{-1}$.\\

\begin{lemma}\label{lem:combinatorial} Let $C$ be a conical, cancellative category $C$.
Suppose that there is a bijective morphism $\theta$ from $bC$ to $aC$.
Then, in fact, $\theta = ab^{-1}$.
\end{lemma}
\begin{proof} Let $\theta \colon bC \rightarrow aC$ be a bijective morphism.
Then $\theta (b)C = aC$.
By Lemma~\ref{lem:unique}, $\theta (b) = a$.
Observe that $\theta (bc) = \theta (b)c = ac$.
Because $\theta$ is a morphism, we have that $\mathbf{d}(\theta (b)) = \mathbf{d}(b)$.
It follows that $\mathbf{d}(a) = \mathbf{d}(b)$.
We may therefore form the basic morphism $ab^{-1}$ and we have proved that $\theta = ab^{-1}$.
\end{proof}

In general, the inverse hull of a cancellative category is hard to describe but under the assumption that $C$ is finitely aligned,
the product of two basic morphisms can be explicitly computed. Note that, in the specific instance of higher rank graphs,
our basic morphisms are closely related to the building blocks of the inverse semigroup constructed in \cite{FMY}.
The following is our version of \cite[Lemma 3.3]{JS2014}.

\begin{lemma}\label{lem:apollo} Let $C$ be a finitely aligned conical cancellative category with a finite number of identities.
Suppose that $(\lambda_{a}\lambda_{b}^{-1})(\lambda_{c}\lambda_{d}^{-1})$ is non-empty
and that $bC \cap cC = \{x_{1}, \ldots, x_{m}\}C$.
Then
$$(\lambda_{a}\lambda_{b}^{-1})(\lambda_{c}\lambda_{d}^{-1})
=
\bigcup_{i=1}^{n} \lambda_{ap_{i}} \lambda_{dq_{i}}^{-1}$$
where $x_{i} = bp_{i} = cq_{i}$ where $1 \leq i \leq m$.
The elements $p_{i}$ and $q_{i}$ are uniquely determined.
\end{lemma}
\begin{proof} We calculate the product of $\lambda_{a}\lambda_{b}^{-1}$ with $\lambda_{c} \lambda_{d}^{-1}$ as partial bijections.
Let $bC \cap cC = \{x_{1}, \ldots, x_{m}\}C$
and put
$e = \mathbf{r}(b) = \mathbf{r}(c)$.
Then $\{x_{1}, \ldots, x_{m}\}C \subseteq eC$.
Let $x_{i} = bp_{i} = cq_{i}$ where $1 \leq i \leq m$.
The elements $p_{i}$ and $q_{i}$ are uniquely determined since we are working in a cancellative category.
Observe that $\mathbf{d}(a) = \mathbf{d}(b) = \mathbf{r}(p_{i})$ and so $ap_{i}$ is defined.
Similarly, $\mathbf{d}(d) = \mathbf{d}(c) = \mathbf{r}(q_{i})$ and so $dq_{i}$ is defined.
Observe that $\mathbf{d}(x_{i}) = \mathbf{d}(p_{i}) = \mathbf{d}(q_{i})$.
Thus the product $\lambda_{ap_{i}} \lambda_{dq_{i}}^{-1}$ is non-empty.
Observe that  $\lambda_{ap_{i}} \lambda_{dq_{i}}^{-1}$ and  $\lambda_{ap_{j}} \lambda_{dq_{j}}^{-1}$
are compatible;
this is easily checked by calculating
$(\lambda_{ap_{i}}\lambda_{dq_{i}}^{-1})^{-1}\lambda_{ap_{j}} \lambda_{dq_{j}}^{-1}$
and
$\lambda_{ap_{i}}\lambda_{dq_{i}}^{-1}(\lambda_{ap_{j}} \lambda_{dq_{j}}^{-1})^{-1}$
and showing that both are idempotents.
To prove that
$$(\lambda_{a}\lambda_{b}^{-1})(\lambda_{c}\lambda_{d}^{-1})
=
\bigcup_{i=1}^{n} \lambda_{ap_{i}} \lambda_{dq_{i}}^{-1}$$
it is enough, by symmetry, to check that both left and right hand sides have the same domains
and that the maps do the same thing, both of which are routine.
\end{proof}

In the case where $C$ is a finitely aligned cancellative category, the inverse hull
$\Sigma (C)$ is an inverse submonoid of $\mathsf{R}(C)$, but it is much easier to work with the latter than the former;
we describe the mathematical relationship between them below.
Lemma~\ref{lem:joinbasic}, Lemma~\ref{lem:oreo} and Lemma~\ref{lem:key-property}
show the important role played by the basic morphisms in the inverse monoid $\mathsf{R}(C)$.

\begin{lemma}\label{lem:joinbasic} Let $C$ be a finitely aligned conical cancellative category with a finite number of identities.
Let $\theta \colon XC \rightarrow YC$ be a bijective morphism between two finitely generated right ideals of $C$.
Then $\theta$ is a union of a finite number of basic morphisms.
\end{lemma}
\begin{proof} We have that $YC = \theta (X)C$ by Lemma~\ref{lem:right-ideal}.
Without loss of generality, we may put $Y = \theta (X)$ and so assume
that $\theta$ induces a bijection between $X$ and $Y$.
Let $x \in X$ and define $y_{x} = \theta (x) \in Y$.
Observe that $\mathbf{d}(x) = \mathbf{d}(y_{x})$.
We may therefore form the basic morphism $y_{x}x^{-1}$.
We claim that $\theta = \bigcup_{x \in X}y_{x}x^{-1}$.
Let $xc \in XC$.
Then $\theta (xc) = \theta (x)c = y_{x}c$.
But $(y_{x}x^{-1})(xc) = y_{x}c$.
\end{proof}

\begin{lemma}\label{lem:oreo} Let $C$ be a finitely aligned conical cancellative category with a finite number of identities.
We suppose that $\mathbf{d}(x) = \mathbf{d}(y)$ and that $\mathbf{d}(u) = \mathbf{d}(v)$.
\begin{enumerate}
\item $xy^{-1} \leq uv^{-1}$ if and only if $(x,y) = (us,vs)$ for some $s \in C$.
It follows that if $xy^{-1}$ is an idempotent so too is $uv^{-1}$.
\item  $xx^{-1} \perp yy^{-1}$ if and only if $x$ and $y$ are independent in $C$.
\item If $xC \cap yC = UC$, where $U$ is a finite set,
then
$xx^{-1} yy^{-1} = \bigcup_{u \in U} uu^{-1}$.
\end{enumerate}
\end{lemma}
\begin{proof} (1) By the definition of the order on partial functions, we have that $yC \subseteq vC$ and $xC \subseteq uC$.
In addition, $xy^{-1}$ and $uv^{-1}$ agree on elements of $yC$.
We have that $y = va$ and $x = ub$.
Now, $(xy^{-1})(y) = x$.
But $(uv^{-1})(y) = ua$.
It follows that $x = ua$ and so $ub = ua$ and so $a = b$.
The result now follows with $s = a = b$.
In order that $xy^{-1}$ be an idempotent, we must have that $x = y$.
It is therefore immediate that if $xy^{-1}$ is an idempotent then so too is $uv^{-1}$.

(2) The idempotents $xx^{-1}$ and  $yy^{-1}$ are orthogonal if and only if $xC \cap yC = \varnothing$.
But this is equivalent to saying that $x$ and $y$ are independent.

(3) The product of $xx^{-1}$ and $yy^{-1}$ is the identity function on $xC \cap yC$ which is the identity function on $UC$.
\end{proof}

The following is a key property since it shows how the basic morphisms sit inside the inverse monoid $\mathsf{R}(C)$.

\begin{lemma}\label{lem:key-property}  Let $C$ be a finitely aligned conical cancellative category with a finite number of identities.
Let $xy^{-1} \leq \bigcup_{j=1}^{n} u_{j}v_{j}^{-1}$ then
$xy^{-1} \leq u_{j}v_{j}^{-1}$ for some $j$.
\end{lemma}
\begin{proof} Observe that $yC \subseteq \{v_{1}, \ldots, v_{n}\}C$.
Thus $y = v_{j}p$ for some $j$ and $p \in C$.
Now $(xy^{-1})(y) = x$
whereas $(u_{j}v_{j}^{-1})(y) = u_{j}p$.
It follows that
$xy^{-1} \leq u_{j}v_{j}^{-1}$.
\end{proof}

In the light of the above results, it will be fruitful to abstract the properties we have distinguished so far in the relationship between 
the basic morphisms and the structure of $\mathsf{R}(C)$.
We shall revisit this abstraction in Section~5.1.
We say that $T$, a distributive inverse semigroup, is the {\em distributive completion} of $S$,
if there is a homomorphism $\iota \colon S \rightarrow T$ such that for any homomorphism $\alpha \colon S \rightarrow D$,
to a distributive inverse semigroup $D$, there is a unique morphism $\beta \colon T \rightarrow D$ such that $\alpha = \beta \iota$.

\begin{lemma}\label{lem:CUNY} Let $S$ be a distributive inverse monoid and let $\mathcal{B}$ be a subset of $S$, containing the zero, satisfying
the following four properties:
\begin{itemize}
\item Each element of $S$ is a finite join of elements of $\mathcal{B}$.
\item If $a \leq \bigvee_{i=1}^{m} a_{i}$ where $a,a_{i} \in \mathcal{B}$ then $a \leq a_{i}$ for some $i$.
\item If $a \leq b$ where $a,b \in \mathcal{B}$ and $a$ is a non-zero idempotent then $b$ is an idempotent.
\item The product of any two elements of $\mathcal{B}$ is either zero or the join of a finite number of elements from $\mathcal{B}$.
\end{itemize}
Denote by $\Sigma$ the inverse subsemigroup of $S$ generated by $\mathcal{B}$. 
We have the following:
\begin{enumerate}
\item $S$ is a $\wedge$-semigroup.
\item $S$ is the distributive completion of $\Sigma$.
\end{enumerate}
\end{lemma}
\begin{proof}
(1) Let $a = \bigvee_{i=1}^{m} a_{i}$, where $a_{i} \in \mathscr{B}$. 
Let $e = \bigvee_{j=1}^{n} e_{j}$, where $e_{i} \in \mathscr{B}$, be any idempotent such that $e \leq a$.
Observe that for each $j$ we have that $e_{j} \leq e$.
By assumption, for each $j$ there exists an $i$ such that $e_{j} \leq a_{i}$.
It follows that $a_{i}$ is an idempotent.
Now, we may split the join  $a = \bigvee_{i=1}^{m} a_{i}$ into two parts so that
$a = \left( \bigvee_{k=1}^{p} b_{k}   \right) \vee \left( \bigvee_{l=1}^{q} f_{l} \right)$
where $b_{k}, f_{l} \in \mathscr{B}$ and all the $f_{l}$ are idempotents and none of the $b_{k}$ is an idempotent.
It follows by our calculations above, that $\bigvee_{l=1}^{q} f_{l}$ is the largest idempotent less than or equal to $a$.
This proves that $S$ is a $\wedge$-semigroup by \cite{Leech}.

(2) We have an embedding $\Sigma \rightarrow S$.
We prove that $S$ is the distributive completion of $\Sigma$.
Let $\alpha \colon \Sigma \rightarrow D$ be any homomorphism to a distributive inverse monoid $D$.
Define $\beta \colon S \rightarrow D$ by $\beta (\bigvee_{j=1}^{n} a_{j}) = \bigvee_{j=1}^{n} \alpha (a_{j})$
where $a_{j} \in \mathcal{B}$.
We need to check that this is well-defined.
Suppose that
$$\bigvee_{j=1}^{n} a_{j}
=
\bigvee_{i=1}^{m} b_{i}$$
where $a_{j},b_{i} \in \mathcal{B}$.
Then, for each $j$, there exists an $i$ such that
$a_{j} \leq b_{i}$.
Thus $\alpha (a_{j}) \leq \alpha (b_{i})$.
From this result, and symmetry, the well-definedness of $\beta$ follows.
That $\beta$ is a morphism of distributive inverse semigroups is now routine to check.
\end{proof}

By proposition~\ref{prop:inverse-monoid}, we know that $\mathsf{R}(C)$ is an inverse monoid.
This inverse monoid is distributive because the union of two finitely generated right ideals is a finitely generated right ideal
and multiplication distributes over unions.
We may therefore apply Lemma~\ref{lem:CUNY}.

\begin{proposition}\label{prop:dis} Let $C$ be a finitely aligned conical cancellative category with a finite number of identities.
Then:
\begin{enumerate}
\item $\mathsf{R}(C)$ is a distributive inverse $\wedge$-monoid.
\item  $\mathsf{R}(C)$ is the distributive completion of $\Sigma (C)$.
\end{enumerate}
\end{proposition}

The following is important in the construction of the associated group.

\begin{proposition}\label{prop:three}
Let $C$ be a finitely aligned conical cancellative category with a finite number of identities.
Then the inverse monoid $\mathsf{R}(C)^{e}$ is $E$-unitary.
\end{proposition}
\begin{proof}
Let $\alpha \colon XC \rightarrow YC$ be a bijective morphism between
two finitely generated essential right ideals.
Suppose that $\alpha$ is the identity when restricted to the finitely generated essential right ideal $ZC$ where $ZC \subseteq XC$.
Let $x \in X$.
Then, since $ZC$ is an essential right ideal, we have that $xC \cap ZC \neq \varnothing$.
It follows that $xa = zb$ for some $a,b \in C$.
But, by assumption, $\alpha (zb) = zb$ and so $\alpha (xa) = xa$.
But $\alpha$ is a morphism and so $\alpha (xa) = \alpha (x)a$.
By cancellation, it follows that $\alpha (x) = x$.
It follows that $\alpha$ is the identity on $X$, and so $\alpha$ is also the identity on $XC$.
\end{proof}

Suppose that $C$ is a finitely aligned conical cancellative category with a finite number of identities.
Then by Proposition~\ref{prop:three}, Proposition~\ref{prop:four} and Lemma~\ref{lem:five},
we can say that two elements of $\mathsf{R}(C)^{e}$ are identified under $\sigma$ if they agree on the intersection of their domains of definition.
This process for constructing a group from an inverse semigroup
of partial bijections is identical to the one used in \cite{YC},
though our group is quite different from the one defined there.

\subsection{The group described in terms of maximal codes}

To say more about the structure of the groups we have constructed from a category,
we need to make further assumptions on that category.
First of all, we shall need to strengthen the notion of finite alignment.

A {\em code} is a finite subset $X \subseteq C$ any two distinct elements of which are independent.
Observe that `finiteness' is part of the definition of a code in this paper.
A {\em maximal code} is a large code.
Given an identify $e$ of $C$, a {\em code in $e$} is a finite subset $X \subseteq eC$  such that any two elements are independent.
A {\em maximal code in $e$} is a code in $e$ which is large in $eC$.
A finitely generated right ideal of $C$ is said to be {\em projective}\footnote{The authors would like to thank J. B. Fountain for this terminology.} if it generated by a code.
We now have the following refinement of the notion of a category's being finitely aligned.
We now strengthen Lemma~\ref{lem:unique}.

\begin{lemma}\label{lem:conical} Let $C$ be a conical cancellative category.
Let $U$ and $V$ be codes such that $UC = VC$.
Then $U = V$.
\end{lemma}
\begin{proof} Let $u \in U$.
Then $u = va$ for some $v \in V$ and $a \in C$.
Let $v = u'b$ for some $u' \in U$ and $b \in C$.
Then $u = va = u'ba$.
But $U$ is a code and so $u = u'$.
By cancellation, $ba$ is an identity.
Similarly, $ab$ is an identity.
But $C$ is conical and so $a$ and $b$ are identities.
We have therefore proved that $U \subseteq V$.
By symmetry, $V \subseteq U$ and so $U = V$ as claimed.
\end{proof}

\noindent
{\bf Definition.} We say that a (conical and cancellative) category $C$ is {\em strongly finitely aligned} if the set $xC \cap yC$, when non-empty, is a code.
If $xC \cap yC = \varnothing$, define $x \sqcup y = \varnothing$;
if $xC \cap yC \neq \varnothing$, define $x \sqcup y$ to be the code such that $xC \cap yC = (x \sqcup y)C$.\\

\begin{lemma}\label{lem:sfa} Let $C$ be a strongly finitely aligned conical cancellative category with a finite number of identities.
Let $XC$ and $YC$ be projective right ideals.
Then $XC \cap YC$ is either empty or a projective right ideal.
\end{lemma}
\begin{proof} We have that $XC \cap YC = \bigcup_{x \in X, y \in Y} (x \sqcup y)C$.
Thus $XC \cap YC$ is certainly finitely generated.
We prove that the set $Z = \bigcup_{x \in X, y \in Y} x \sqcup y$ is a code.
Let $a \in x_{i} \sqcup y_{j}$ and $b \in x_{k} \sqcup y_{l}$
where $x_{i}, x_{k} \in X$ and $y_{j}, y_{l} \in Y$.
We shall prove that $a$ and $b$ are independent.
Suppose, to the contrary, that $a$ and $b$ are dependent.
Then $z = au = bv$ for some $u,v \in C$.
But $a \in x_{i}C \cap y_{j}C$ and $b \in x_{k}C \cap y_{l}C$.
The sets  $x_{i}C \cap y_{j}C$ and $x_{k}C \cap y_{l}C$ are right ideals.
It follows that $z \in (x_{i}C \cap y_{j}C) \cap (x_{k}C \cap y_{l}C)$.
Thus $x = x_{i} = x_{k}$, since $x_{i}, x_{k} \in X$ and $X$ is a code
and $y = y_{j} = y_{l}$, since  $y_{j}, y_{l} \in Y$ and $Y$ is a code.
It follows that $a,b \in x \sqcup y$.
But $a,b \in x \sqcup y$ and $x \sqcup y$ is a code and so $a = b$.
\end{proof}

\noindent
{\bf Definition. } Let $C$ be a strongly finitely aligned cancellative category with a finite number of identities.
Define $\mathsf{P}(C)$ to be the set of all bijective morphisms between projective right ideals.\\

\begin{lemma}\label{prop:projective} Let $C$ be a strongly finitely aligned conical cancellative category with a finite number of identities.
Then $\mathsf{P}(C)$ is an inverse subsemigroup of $\mathsf{R}(C)$.
\end{lemma}
\begin{proof} By Lemma~\ref{lem:sfa}, the intersection of any two projective right ideals is a projective right ideal.
Let $\alpha \colon XC \rightarrow YC$ be a bijective morphism between two projective right ideals.
Let $ZC \subseteq XC$ be a projective right ideal.
It is enough to prove that $\alpha (Z)$ is a code.
Suppose that $\alpha (z)$ and $\alpha (z')$ are dependent for some $z,z' \in Z$.
Then $\alpha (z)u = \alpha (z')v$ for some $u,v \in C$.
Then $\alpha (zu) = \alpha (z'v)$ since $\alpha$ is a morphism.
Thus $zu = z'v$ because $\alpha$ is a bijection.
But $Z$ is a code, and so $z = z'$.
It follows that $\alpha (z) = \alpha (z')$.
We have therefore proved that $\alpha (Z)$ is a code.
It follows that $\alpha$ restricts to a bijective morphism $ZC \rightarrow \alpha (Z)C$.
The fact that $\mathsf{P}(S)$ is an inverse subsemigroup of $\mathsf{R}(S)$ is now immediate,
\end{proof}

The elements of $\mathsf{P}(C)^{e}$ are (using Lemma~\ref{lem:zero}) the bijective morphisms between the right ideals of $C$ generated by maximal codes.

\begin{lemma}\label{lem:projright}  Let $C$ be a strongly finitely aligned conical cancellative category with a finite number of identities.
Then $\mathsf{P}(C)^{e}$ is an inverse submonoid of $\mathsf{R}(C)^{e}$.
\end{lemma}
\begin{proof} We prove first that $\mathsf{P}(C)^{e}$ is contained in $\mathsf{R}(C)^{e}$.
Observe that an idempotent in $\mathsf{P}(C)^{e}$ is an identity function defined on a
finitely generated right ideal of $C$ generated by a code which intersects every finitely generated right ideal of $C$ generated
by a code. In particular, it intersects principal right ideals of $C$.
It is now clear that  $\mathsf{P}(C)^{e}$ is contained in $\mathsf{R}(C)^{e}$.
The composition in  $\mathsf{P}(C)^{e}$ is just the restriction of the composition in $\mathsf{R}(C)^{e}$.
\end{proof}

\noindent
{\bf Condition (MC)}. Let $C$ be a strongly finitely aligned conical cancellative category with a finite number of identities.
We assume that if $XC$ is any finitely generated essential right ideal then there is a $YC \subseteq XC$ where
$Y$ is a maximal code. In other words, every  finitely generated essential right ideal contains an essential projective right ideal.\\

\begin{lemma}\label{lem:restriction} Let $C$ be a strongly finitely aligned conical cancellative category with a finite number of identities
satisfying condition (MC).
Then each element of $\mathsf{R}(C)^{e}$ is above an element of  $\mathsf{P}(C)^{e}$.
\end{lemma}
\begin{proof} Let $\alpha \colon XC \rightarrow YC$ be a bijective morphism between two finitely generated essential right ideals.
Without loss of generality, we assume that $\alpha (X) = Y$.
Let $ZC \subseteq XC$ be a finitely generated right ideal generated by a maximal code.
We prove that $\alpha (Z)$ is also a maximal code.
It is a code by Proposition~\ref{prop:projective}.
We now show that it is a maximal code.
To do this, we need to show that every element of $C$ is dependent on an element of $\alpha (Z)$.
Let $a \in C$.
Since $YC$ is an essential right ideal we have that $au = yv$ for some $u,v \in C$.
But $y \in Y$ and so there is an $x \in X$ such that $\alpha (x) = y$.
It follows that $au = \alpha (x)v$.
Thus, $au = \alpha (xv)$ since $\alpha$ is a morphism.
Now, $Z$ is a maximal code.
It follows that $xvp = zs$ for some $p,s \in C$.
Thus, $\alpha (x)vp = \alpha (z)s$.
Hence, $aup = \alpha (x)vp =  \alpha (z)s$.
We have proved that $a \in C$ is dependent on an element of $\alpha (Z)$.
It follows that $\alpha (Z)$ is a maximal code.
\end{proof}

The following was proved as \cite[Lemma~7.10]{LV2019b}.

\begin{lemma}\label{lem:quotients} Let $S$ be an inverse subsemigroup of an inverse semigroup $T$.
Suppose that each element of $T$ lies above an element of $S$ in the natural partial order.
Then $S/\sigma \cong T/\sigma$.
\end{lemma}

On the strength of Lemma~\ref{lem:restriction} and Lemma~\ref{lem:quotients}, we have proved the following.

\begin{theorem}\label{them:groups}
Let $C$ be a strongly finitely aligned conical cancellative category with a finite number of identities
satisfying condition (MC).
Then
$$\mathsf{P}(C)^{e}/\sigma \cong \mathsf{R}(C)^{e}/\sigma.$$
\end{theorem}

A typical element of $\mathsf{P}(C)^{e}$ is a bijective morphism $\theta \colon XC \rightarrow YC$
where $X$ and $Y$ are maximal codes and $\theta (X) = Y$.
Then $\theta = \bigvee_{x \in X} \theta (x)x^{-1}$ by Lemma~\ref{lem:joinbasic}.
By Lemma~\ref{lem:oreo}, this is an orthogonal join.

\vspace{5mm}
\begin{center}
\fbox{\begin{minipage}{15em}
Let $C$ be a strongly finitely aligned conical cancellative category with a finite number of identities satisfying (MC).
Then
$$\mathscr{G}(C) \cong \mathsf{P}(C)^{e}/\sigma$$
is the  {\em group associated with $C$}.
\end{minipage}}
\end{center}
\vspace{5mm}

\section{The group associated with a higher rank graph}

The goal of this section is to show how to construct the group $\mathscr{G}(C)$,
described in Section~3, in the case where $C$ is a higher rank graph (under some suitable assumptions on $C$).
Higher-rank graphs were introduced in \cite{KP} as combinatorial models for the systems of matrices, and
associated $C^*$-algebras, studied in \cite{RS}. 

We begin with some motivation.
Let $A^{\ast}$ be the free monoid on the finite set $A$.
Then there is a homomorphism $A^{\ast} \rightarrow \mathbb{N}$ given by $x \mapsto |x|$, the length of $x$.
More generally, let $G$ be a directed graph.
Denote by $G^{\ast}$ the free category generated by $G$.
This consists of all finite allowable strings of elements of $G$ with
the identities being identified with the vertices of $G$.
Let $x \in G^{\ast}$.
Then $x = x_{1} \ldots x_{n}$ where $x_{1}, \ldots, x_{n}$ are edges of $G$ so that the edge $x_{i}$ begins where the edge $x_{i+1}$ ends.
Define $\mathbf{d}(x)$ to be the identity at the source of $x_{n}$
and define $\mathbf{r}(x)$ to be the identity at the target of $x_{1}$.
The free monoid is a special case of the free category when the directed graph is the bouquet of circles.
This time, we can define a functor $G^{\ast} \rightarrow \mathbb{N}$ which associates with a path $x$ its length $\lambda(x)$.
This functor has the following important property:
if $\lambda (x) = m + n$ then we can write $x = x_{1}x_{2}$ {\em uniquely} where $\lambda (x_{1}) = m$ and $\lambda (x_{2}) = n$.  
This example is generalized in the following definition, which comes from \cite{KP}.
We write $\mathbb{N}^{k}$ for $k$-tuples of elements of $\mathbb{N}$.
The elements of $\mathbb{N}^{k}$ are denoted by bold letters.
The additive identity is denoted by $\mathbf{0}$.
The set $\mathbb{N}^{k}$ is the positive cone of the lattice-ordered group $\mathbb{Z}^{k}$,
where the order is defined componentwise.\\

\noindent
{\bf Definition.} A countable category $C$ is said to be a {\em higher rank graph} or a {\em $k$-graph}
if there is a functor $d \colon C \rightarrow \mathbb{N}^{k}$,
called the {\em degree map},
satisfying the {\em unique factorization property} (UFP): if $d(a) = \mathbf{m} + \mathbf{n}$ then there are unique elements $a_{1}$ and $a_{2}$ in $C$
such that $a = a_{1}a_{2}$ where $d(a_{1}) = \mathbf{m}$ and $d(a_{2}) = \mathbf{n}$.
We call $d(x)$ the {\em\degree} of $x$.
A {\em morphism} of $k$-graphs is a degree-preserving functor.\\

The following was proved as \cite[Proposition 1.8]{KP}.

\begin{lemma}\label{lem:products-of-hrg}
The direct product of higher rank graphs is again a higher rank graph.
\end{lemma}

\begin{remark}{\em The  $1$-graphs are precisely the countable free categories.}
\end{remark}

It follows that whereas the finite direct product of free categories is not free,
it is always a higher rank graph.

\begin{notation}{\em 
Let $C$ be a high rank graph.
It will be useful to have some systematic notation that springs from the (UFP) rather than using {\em ad hoc} notation.
Let $a \in C$ and suppose that $\mathbf{0} \leq \mathbf{m} \leq \mathbf{n} \leq d(a)$. 
Then by repeated application of the (UFP) we obtain the factorization $a = a_{1}a_{2}a_{3}$ 
where 
$d(a_{1}) = \mathbf{m}$,
$d (a_{2}) = \mathbf{n} - \mathbf{m}$
and
$d(a_{3}) = d(a) - \mathbf{n}$.
Define
$$a[\mathbf{n},\mathbf{m}] = a_{2}.$$
Then if 
if $\mathbf{0} \le \mathbf{m}_1 \le \mathbf{m}_2 \le \cdots \le \mathbf{m}_l \le d(a)$, 
we have (using the definition above)
that
$$a = a[\mathbf{0}, \mathbf{m}_1] a[\mathbf{m}_1, \mathbf{m}_2] \cdots a[\mathbf{m}_{l-1}, \mathbf{m}_l] a[\mathbf{m}_l, d(a)].$$
It follows, for example, that if $a = bc$ and $\mathbf{m} \le d(b)$, then $a[0,\mathbf{m}] = b[0,\mathbf{m}]$.}
\end{notation}

By \cite[Remarks~1.2]{KP}, we have the following.
All are easy to prove directly.

\begin{lemma}\label{lem:important-hrg}
Let $C$ be a $k$-graph.
\begin{enumerate}
\item $C$ is cancellative.
\item $C$ is conical.
\item The elements of $C$ of degree $\mathbf{0}$ are precisely the identities.
\end{enumerate}
\end{lemma}

We shall now derive some properties of higher rank graphs that will be important later.
The following is proved by a simple application of the (UFP).
It generalizes \cite[Lemma 3.10]{LV2019b}.

\begin{lemma}\label{lem:levi} Let $d \colon C \rightarrow \mathbb{N}^{k}$ be a $k$-graph.
Suppose that $xy = uv$ where $d(x) \geq d(u)$.
Then there exists an element $t \in C$ such that $x = ut$ and $v = ty$.
In particular, if $d(x) = d(u)$ then $x = u$.
\end{lemma}

The above lemma proves the following.

\begin{lemma}\label{lem:indep} Let $d \colon C \rightarrow \mathbb{N}^{k}$ be a $k$-graph.
If $a \neq b$ and $d(a) = d(b)$ then $a$ and $b$ are independent.
\end{lemma}
\begin{proof} Suppose that $ax = by$ for some $x,y \in C$.
Then by Lemma~\ref{lem:levi}, we have that $a = b$ which contradicts our assumption.
\end{proof}

By the above lemma, any finite set of elements of $C$ in which all elements have the same degree is independent
and so forms a code.
Let $\mathbf{m} \in \mathbb{N}^{k}$.
Define $C_{\mathbf{m}}$ to be the subset of $C$ which consists of all elements of degree $\mathbf{m}$.\\

\noindent
{\bf Definition. }We say that a higher rank graph $C$ has {\em no sources} if for each identity $e$ of $C$ and element $\mathbf{m} \in \mathbb{N}^{k}$ there exists
an arrow $x \in C$ such that $\mathbf{r}(x) = e$ and $d(x) = \mathbf{m}$.\\

One can deduce the following from \cite[Proposition 1.8]{KP}.

\begin{lemma}\label{lem:products-no-sources} The direct product of two higher rank graphs each of which has no sources is itself
a higher rank graph that has no sources.
\end{lemma}

\begin{example}
{\em Let $G$ be a directed graph.
To say that $G^{\ast}$ has no sources means that the in-degree of each vertex is at least 1;
this means that ending at any vertex, I can construct a path of any length (by going backwards).
We refer the reader to \cite{JL}.}
\end{example}

\begin{lemma}\label{lem:mc} Let $d \colon C \rightarrow \mathbb{N}^{k}$ be a $k$-graph with no sources.
Then $C_{\mathbf{m}}$ is a maximal code.
\end{lemma}
\begin{proof} By Lemma~\ref{lem:indep}, $C_{\mathbf{m}}$ is a code.
We prove that it is maximal.
Let $a \in C$ be arbitrary.
Since $C$ is assumed to have no sources, there is an element $x$ such that $ax$ is defined and $d(x) = \mathbf{m}$.
It follows that $d(ax) \geq \mathbf{m}$.
By the (UFP), we can write $ax = by$ where $d(b) = \mathbf{m}$.
This proves the claim.
\end{proof}

\begin{lemma}\label{lem:maximal-code} Let $C$ be a $k$-graph with no sources, let $\mathbf{m} \in \mathbb{N}^{k}$, and let $e$ be an identity.
Then the set $eC_{\mathbf{m}}$ of all elements of $C$ with range $e$ and degree $\mathbf{m}$ is a maximal code in $eC$.
\end{lemma}
\begin{proof} Because we are assuming that there are no sources the set  $eC_{\mathbf{m}}$ is non-empty.
It consists of idependent elements by Lemma~\ref{lem:indep} and so is a code in $eC$.
It remains to prove that it is a maximal code in $eC$.
Let $c \in eC$ be arbitrary.
By Lemma~\ref{lem:mc}, there exists $a \in C_{\mathbf{m}}$ such that $c$ is dependent on $a$.
It follows that $cu = av$ for some $u,v \in C$.
Observe that the codomain of $a$ is the same as the codomain of $c$.
It follows that $a \in eC$.
\end{proof}

We can now show that condition (MC) holds for higher rank graphs with no sources.

\begin{lemma}\label{lem:containsmc} Let $d \colon C \rightarrow \mathbb{N}^{k}$ be a $k$-graph with no sources.
Let $XC$ be an essential finitely generated right ideal.
Put $\mathbf{m} = \bigvee_{x \in X}d(x)$.
Then $C_{\mathbf{m}}C \subseteq XC$.
Thus, every finitely generated essential right ideal contains a right ideal generated by a maximal code.
\end{lemma}
\begin{proof}
Let $y \in C_{\mathbf{m}}$.
We are assuming that $XC$ is an essential right ideal and so $X$ is a large subset by Lemma~\ref{lem:zero}.
It follows that $ya = xb$ for some $x \in X$ and $a,b \in C$.
But $d(y) \geq d(x)$.
Thus by Lemma~\ref{lem:levi}, there is $t \in C$ such that $y = xt$.
We have therefore proved that $C_{\mathbf{m}} \subseteq XC$.
It follows that
$C_{\mathbf{m}}C \subseteq XC$.
\end{proof}

The following result is well-known but we prove it for the sake of completeness.

\begin{lemma}\label{lem:fa} Let $d \colon C \rightarrow \mathbb{N}^{k}$ be a $k$-graph and suppose that $a,b,c \in C$
satisfy $cC \subseteq aC \cap bC$.
Then there is an element $e \in C$ such that $cC \subseteq eC \subseteq aC \cap bC$ where $d(e) = d(a) \vee d(b)$.
\end{lemma}
\begin{proof} By assumption, $c = au = bv$ for some $u,v \in C$.
Thus $d(c) = d(a) + d(u)$ and $d(c) = d(b) + d(v)$.
It follows that $d(a), d(b) \leq d(c)$.
Thus $d(a) \vee d(b) \leq d(c)$ (in the lattice-orderd group $\mathbb{Z}^{k}$).
Hence, in the notation we introduced above, we may define 
$$e = c[0, d(a) \vee d(b)]$$ 
and
$$f = c[d(a) \vee d(b), d(c)]$$ 
such that $c = ef$ and $d(e) = d(a) \vee d(b)$. 
Observe that $cC = efC \subseteq eC$.
Also, $au = c = ef = e[0, d(a)]e[d(a), d(e)]f$,
where we have factorized $e$.
By the (UFP), we have that $e[0, d(a)] = a$.
We therefore have that
$eC = ae[d(a), d(e)] C \subseteq aC$. 
A symmetric argument gives $eC \subseteq bC$.
\end{proof}

\noindent
{\bf Definition. }Let $a,b \in C$, a higher rank graph.
Define the notation $a \vee b$ as follows.
If $a$ and $b$ are independent, define $a \vee b = \varnothing$.
Otherwise, $a \vee b$ consists of all elements $e \in aC \cap bC$ such that $d(e) = d(a) \vee d(b)$.\\

\begin{lemma}\label{lem:meets} Let $d \colon C \rightarrow \mathbb{N}^{k}$ be a $k$-graph.
Then $aC \cap bC = \bigcup_{x \in a \vee b} xC$.
\end{lemma}
\begin{proof} Without loss of generality, we assume that $a \vee b$ is non-empty.
Let $y \in aC \cap bC$.
Then by Lemma~\ref{lem:fa}, we have that $y = ex$ for some $e \in a \vee b$.
It follows that the LHS is contained in the RHS.
The reverse inclusion is immediate.
\end{proof}

\noindent
{\bf Definition.}
We say that $C$ is {\em row finite} if for each identity $e$ of $C$, the number of elements of $eC$ of degree $\mathbf{m}$ is finite.\\

One can deduce the following from \cite[Proposition 1.8]{KP}.

\begin{lemma}\label{lem:products-row-finite} The direct product of two higher rank graphs each of which is row finite is itself
a higher rank graph that is row finite.
\end{lemma}

The following is immediate from the definitions and Lemma~\ref{lem:meets}.

\begin{lemma}\label{lem:finitely-aligned-rf}
A higher rank graph that is row finite is finitely aligned.
\end{lemma}

The following is immediate by Lemma~\ref{lem:indep} and shows that the categories underlying higher rank graphs are in
fact strongly finitely aligned.

\begin{lemma}\label{lem:code} Let $d \colon C \rightarrow \mathbb{N}^{k}$ be a $k$-graph.
Let $a,b \in C$. If $a \vee b$ is non-empty and finite it is a code.
\end{lemma}

We now show how to construct a group from a higher rank graph $C$ such that the following hold:
\begin{itemize}
\item $C$ has a finite number of identities.
\item $C$ has no sources.
\item $C$ is row-finite.
\end{itemize}
By Lemma~\ref{lem:important-hrg}, the category $C$ is conical and cancellative;
by Lemma~\ref{lem:mc} and Lemma~\ref{lem:containsmc}, condition (MC) holds;
by Lemma~\ref{lem:fa}, Lemma~\ref{lem:meets}, Lemma~\ref{lem:finitely-aligned-rf}, Lemma~\ref{lem:code},  it is strongly finitely aligned.
We may therefore define the group $\mathscr{G}(C)$ in the fashion of Section~3.3.

Using Lemma~\ref{lem:products-of-hrg},
Lemma~\ref{lem:products-no-sources},
and
Lemma~\ref{lem:products-row-finite},
examples of groups satisfying these conditions are easy to construct. 
Let $A_{1}, \ldots, A_{k}$ be finite alphabets.
Then there is a group $\mathscr{G}(A_{1}^{\ast} \times \ldots \times A_{k}^{\ast})$ constructed as above.
We refer the reader to \cite{LV2019b} for more on these groups in the case where each alphabet contains at least two elements.
More generally, let $G$ be a finite directed graph such that the in-degree of each vertex is at least 1.
Then we may form the group $\mathscr{G}(G)$, a class of which is studied in \cite{JL}.
If $G_{1}, \ldots, G_{k}$ is a set of finite directed graphs each of which satisfies the condition that the in-degree of each vertex is at least 1, 
then $G_{1}^{\ast} \times \ldots \times G_{k}^{\ast}$
is a higher rank graph satisfying the above conditions and so we may form the group    
 $\mathscr{G}(G_{1}^{\ast} \times \ldots \times G_{k}^{\ast})$.

\section{Our group is a group of units of a distributive inverse $\wedge$-monoid}

In this section, we shall prove that the group $\mathscr{G}(C)$ 
is a group of units of a distributive inverse $\wedge$-semigroup.
To show this, we shall reprove some results from \cite[Section~9]{LV2019b} in a slightly more general setting.
We recall first a definition due to Daniel Lenz \cite{Lenz}.

Let $S$ be an inverse semigroup with zero.
Define the relation $\equiv$ on $S$ as follows:
$s \equiv t$ if and only if for each $0 < x \leq s$ we have that $x^{\downarrow} \cap t^{\downarrow} \neq 0$
and for each $0 < y \leq t$ we have that $y^{\downarrow} \cap s^{\downarrow} \neq 0$.
Then $\equiv$ is a $0$-restricted congruence on $S$.
We denote the $\equiv$-class of $a$ by $[a]$.
We dub $\equiv$ the {\em Lenz congruence}.

It is easy to check that if $a,b \leq c$ then $a \sim b$, 
it follows by Lemma~\ref{lem:compatibility-meets} 
that if $a \sim b$ then $a \wedge b$ exists.
We shall use this below to guarantee that the meet of the two elements that we claim exists really does exist. 
Let $a \leq b$.
We say that $a$ is {\em essential in} $b$ if
$0 < x \leq b$ implies that $x \wedge a \neq 0$.
In this case, we write $a \leq_{e} b$.
Clearly,
$$a \leq_{e} b \Leftrightarrow a \equiv b \text{ and } a \leq b.$$
Our use of the word `essential' is explained in the next lemma.

\begin{lemma}\label{lem:essential-equiv} Let $T$ be an inverse monoid.
Then $e$ is an essential idempotent if and only if $e \equiv 1$.
\end{lemma}
\begin{proof} Suppose that $e$ is an essential idempotent.
Then, by definition, for any non-zero idempotent $f$ we have that $ef \neq 0$.
This proves that $e \equiv 1$.
The proof of the converse is immediate.
\end{proof}

Let $S$ be an inverse semigroup and
let $\{a_{1}, \ldots, a_{m}\} \subseteq a^{\downarrow}$.
We say that $\{a_{1}, \ldots, a_{m}\}$  is a {\em tight cover} of $a$ if for every $0 < z \leq a$
we have that $z \wedge a_{i} \neq 0$ for some $i$.

\begin{lemma}\label{lem:blackadder} Let $S$ be a distributive inverse semigroup.
\begin{enumerate} 

\item $\{a_{1}, \ldots, a_{m}\}$ is a tight cover of $a$ if and only if
$b \leq_{e} a$ where $b = \bigvee_{i=1}^{m} a_{i}$.

\item $\{\bigvee_{i=1}^{m} a_{i}, b_{1}, \ldots, b_{n}\}$ is a tight cover of $s$
if and only if 
$\{a_{1}, \ldots, a_{m}, b_{1}, \ldots, b_{n}\}$ is a tight cover of $s$.

\end{enumerate}
\end{lemma}
\begin{proof} 
(1) Use \cite[Lemma 2.5]{Lawson2016}.

(2) Suppose that $\{\bigvee_{i=1}^{m} a_{i}, b_{1}, \ldots, b_{n}\}$ is a tight cover of $s$.
We prove that $\{a_{1}, \ldots, a_{m}, b_{1}, \ldots, b_{n}\}$ is a tight cover of $s$.
Let $0 < z \leq s$.
The only case that need concern us is where $(\bigvee_{i=1}^{m} a_{i}) \wedge z \neq 0$.
We now use \cite[Lemma 2.5]{Lawson2016} to deduce that $a_{i} \wedge z \neq \varnothing$ for some $i$.
The proof of the converse is now straightforward.
\end{proof}

The proof of the following can be deduced using the proof of \cite[Lemma~9.7]{LV2019b}.

\begin{lemma}\label{lem:red-panda} Let $T$ be an inverse monoid in which $\equiv$ is idempotent-pure
and suppose that $T^{e}$ is $E$-unitary.
Then the group of units of $T/{\equiv}$ is isomorphic to $T^{e}/\sigma$.
\end{lemma}

The following was proved as \cite[Lemma~9.12]{LV2019b}.

\begin{lemma}\label{lem:idpt-pure-dist} Let $S$ be a distributive inverse semigroup.
If $\rho$ is idempotent-pure then $S/\rho$ is a distributive inverse semigroup and the natural map from $S$ to $S/\rho$
is a morphism of distributive inverse semigroups.
If, in addition, $S$ is a $\wedge$-semigroup then $S/\rho$ is a $\wedge$-semigroup
and the morphism preserves meets.
\end{lemma}

The goal now is to show that $\equiv$ is idempotent-pure in the case of interest to us.

\begin{lemma} In an inverse semigroup, suppose that $a \leq b$.
Then $a \leq_{e} b$ if and only if $\mathbf{d}(a) \leq_{e} \mathbf{d}(b)$.
\end{lemma}
\begin{proof} Suppose that $a \leq_{e} b$.
We prove that $\mathbf{d}(a) \leq_{e} \mathbf{d}(b)$.
Let $0 < f \leq  \mathbf{d}(a)$.
Then $0 < bf \leq b$.
Now $a,bf \leq b$ and so $a \sim bf$.
By assumption, $a \wedge bf \neq 0$.
Thus $\mathbf{d}(a) f \leq 0$.
We have therefore proved that $\mathbf{d}(a) \leq_{e} \mathbf{d}(b)$.
Conversely, suppose that $\mathbf{d}(a) \leq_{e} \mathbf{d}(b)$.
Let $0 < x \leq b$.
Then $0 < \mathbf{d}(x) \leq \mathbf{d}(b)$.
By assumption, $\mathbf{d}(x)\mathbf{d}(a) \neq 0$.
But $x,a \leq b$ implies that $x \sim a$.
It follows that $x \wedge a$ exists and is non-zero.
We have therefore proved that $a \leq_{e} b$.
 \end{proof}

Let $\theta \colon S \rightarrow T$ be a homomorphism.
We say that it is {\em essential} if $a \leq_{e} b$ implies that $\theta (a) = \theta (b)$.
We say that a congruence $\rho$ on a semigroup is {\em essential} if $a \leq_{e} b$ implies that $a \mathbin{\rho} b$.

\begin{lemma}\label{lem:important-tool} Let $S$ be an inverse semigroup.
If $\rho$ is any $0$-restricted, idempotent-pure essential congruence on $S$
then $a \, \rho \, b$ implies that $a \wedge b$ is defined and $a \wedge b \leq_{e} a,b$.
\end{lemma}
\begin{proof}
Suppose that $a \, \rho \, b$.
Then $a \sim b$ by Lemma~\ref{lem:idpt-pure-characterizations}
and so $a \wedge b$ exists by Lemma~\ref{lem:compatibility-meets}.
We prove that $a \wedge b \leq_{e} a$;
the fact that $a \wedge b \leq_{e} b$ follows by symmetry.
Since $a \sim b$ we have that $a \wedge b = ab^{-1}b$.
Let $0 < x \leq b$.
Then $x  \sim a \wedge b$ since $x, a \wedge b \leq b$.
Thus $(a \wedge b) \wedge x = (a \wedge b)x^{-1}x$.
But $(ab^{-1}b)x^{-1}x \, \rho \, x$.
We are given that $x \neq 0$ and so $(a \wedge b)\wedge x \neq 0$.
The claim now follows.
\end{proof}

The following result links $\leq_{e}$ and $\equiv$;
observe that $\equiv$ is $0$-restricted and essential.

\begin{lemma}\label{lem:new-characterization} Let $S$ be an inverse semigroup
in which $\equiv$ is idempotent-pure.
Then $a \equiv b$ if and only if there exists $c \leq_{e} a,b$.
\end{lemma}
\begin{proof} It is immediate that if there exists $c \leq_{e} a,b$ then $a \equiv b$.
The converse follows by  Lemma~\ref{lem:important-tool}.
\end{proof}

The next result characterizes the congruence $\equiv$ in the case where it is idempotent-pure.

\begin{proposition}\label{prop:uniqueness}
Let $S$ be an inverse semigroup on which $\equiv$ is idempotent-pure.
Then $\equiv$ is the unique $0$-restricted, idempotent-pure essential congruence on $S$.
\end{proposition}
\begin{proof} By definition $\equiv$ is $0$-restricted, it is idempotent-pure by assumption and it is an essential congruence
by virtue of its definition.
Let $\rho$ be any $0$-restricted, idempotent-pure essential congruence on $S$.
We shall prove that $\rho \, = \, \equiv$.
Let $a \, \rho \, b$.
By Lemma~\ref{lem:important-tool} there exists $x \leq_{e} a,b$, and so
Lemma~\ref{lem:new-characterization} gives $a \equiv b$.
We have therefore shown that $\mathbin{\rho} \subseteq \mathbin{\equiv}$.
We now prove the reverse inclusion.
Let $a \equiv b$.
Then Lemma~\ref{lem:important-tool} shows that $a \wedge b$ is defined and
$a \wedge b \leq_{e} a,b$ and hence $a \mathbin{\rho} b$ because $\rho$ is an essential
congruence.\end{proof}

The following is now immediate by Lemma~\ref{lem:new-characterization}.

\begin{lemma}\label{lem:cheese} Let $S$ be an inverse semigroup on which $\equiv$ is idempotent-pure.
Let $\rho$ any congruence on $S$ such that if $a \leq_{e} b$ then $\rho (a) = \rho (b)$.
Then $\equiv$ is contained in $\rho$.
\end{lemma}

The question is, therefore, for which inverse semigroups is $\equiv$ idempotent-pure.
We refer the reader back to Lemma~\ref{lem:oreo} to justify the following.

\begin{proposition}\label{prop:panda} Let $S$ be a distributive inverse semigroup.
Let $\mathcal{B}$ be a subset of $S$ having the following properties:
\begin{itemize}
\item Each element of $S$ is a finite join of elements from $\mathcal{B}$.
\item If $a \leq \bigvee_{i=1}^{m} a_{i}$ where $a,a_{i} \in \mathcal{B}$ then $a \leq a_{i}$ for some $i$.
\item If $a \leq b$, where $a,b \in \mathcal{B}$ and $a$ is a non-zero idempotent, then
$b$ is an idempotent.
\end{itemize}
Then $\equiv$ is idempotent-pure on $S$.
\end{proposition}
\begin{proof} Suppose that $a \, \equiv \, e$ where $e$ is a non-zero idempotent.
We need to prove that $a$ is an idempotent.
We use the fact that every element of $S$ is a join of elements in $\mathcal{B}$.
Then $\left(  \bigvee_{i=1}^{m} a_{i} \right) \equiv \left( \bigvee_{j=1}^{n} e_{j} \right)$
where $a_{i}, e_{j} \in \mathcal{B}$.
We assume all elements are non-zero.
For each $i$, we have that $a_{i} \leq  \bigvee_{i=1}^{m} a_{i}$.
Thus, from the definition of $\equiv$, there is a non-zero element $z$ in $S$
such that $z \leq a_{i}$ and $z \leq  \bigvee_{j=1}^{n} e_{j}$.
The element $z$ is an idempotent and belongs to $S$.
Thus $z$ is a join of non-zero idempotents from $\mathcal{B}$.
So, there is a non-zero idempotent $e \in \mathcal{B}$ such that $e \leq z$
In particular, $e \leq a_{i}$.
It follows from our assumptions that $a_{i}$ is itself an idempotent.
It follows that all of $a_{1}, \ldots , a_{m}$ are idempotents and so $a$ is an idempotent,
as required.
\end{proof}

The following refines Lemma~\ref{lem:cheese}.

\begin{lemma} Let $S$ be a distributive inverse semigroup.
Let $\mathcal{B}$ be a subset of $S$ having the following properties:
\begin{itemize}
\item Each element of $S$ is a finite join of elements from $\mathcal{B}$.
\item If $a \leq \bigvee_{i=1}^{m} a_{i}$ where $a,a_{i} \in \mathcal{B}$ then $a \leq a_{i}$ for some $i$.
\end{itemize}
Let $\theta \colon S \rightarrow T$ be a morphism of distributive inverse semigroups such that $b \leq_{e} a$, where $a \in \mathcal{B}$,
implies that $\theta (b) = \theta (a)$.
Then $\equiv$ is contained in the kernel of $\theta$.
\end{lemma}
\begin{proof} Suppose that $a \leq_{e} b$ where $b = \bigvee_{i} b_{i}$ such that $b_{i} \in \mathcal{B}$.
Then $a \wedge b_{i} \leq_{e} b_{i}$ for all $i$.
Observe that $a \wedge b_{i}$ is algebraically defined since $a \sim b_{i}$ by Lemma~\ref{lem:compatibility-meets}.
It follows that $\theta (b_{i}) \leq \theta (a)$ for all $i$ giving $\theta (b) \leq \theta (a)$.
But $\theta (a) \leq \theta (b)$ and so $\theta (a) = \theta (b)$.
\end{proof}

Let $C$ be a higher rank graph with a finite number of identities which has no sources and is row finite.
By Lemmas~\ref{lem:joinbasic}, \ref{lem:oreo}~and~Lemma~\ref{lem:key-property},
the inverse monoid $\mathsf{R}(C)$ satisfies the conditions of Proposition~\ref{prop:panda}
and so $\equiv$ is idempotent-pure.
By Proposition~\ref{prop:dis}, $\mathsf{R}(C)$ is a distributive inverse $\wedge$-monoid and so
by Lemma~\ref{lem:idpt-pure-dist}, it follows that $\mathsf{R}(C)/{\equiv}$ is a distributive
inverse $\wedge$-monoid.
By Proposition~\ref{prop:three}, $\mathsf{R}(C)^{e}$ is $E$-unitary.
Thus by Lemma~\ref{lem:red-panda}, the group of units of $\mathsf{R}(C)/{\equiv}$
is isomorphic to $\mathscr{G}(C)$.
We have therefore proved the following.

\begin{theorem}\label{them:israel}
Let $C$ be a higher rank graph with a finite number of identities which has no sources and is row finite.
Then the group $\mathscr{G}(C)$, defined at the end of Section~4, is isomorphic with the group of units of $\mathsf{R}(C)/{\equiv}$,
which is a distributive inverse $\wedge$-monoid.
\end{theorem}

\noindent
{\bf Definition. }Let $C$ be a higher rank graph with a finite number of identities which has no sources and is row finite.
Put $\mathsf{B}(C) = \mathsf{R}(C)/{\equiv}$.

\section{Our group is a group of units of a Boolean inverse $\wedge$-monoid}

In this section, we shall prove that the group $\mathscr{G}(C)$ is isomorphic to the group of units of a Boolean inverse $\wedge$-monoid.
This will enable us to show easily that $\mathscr{G}(C)$ is a topological full group.
Specifically, we shall prove that $\mathsf{B}(C)$ is Boolean.
To do this, we need the concept of a tight filter in an inverse semigroup.
The proper filter $A$ is {\em tight} if for every $a \in A$ and every tight cover
$\{a_{1}, \ldots, a_{m}\}$ of $a$, there exists $i \leq m$ such that $a_{i} \in A$.

\begin{lemma}\label{lem:beverly} Let $S$ be a distributive inverse monoid.
\begin{enumerate}
\item Every ultrafilter is a tight filter.
\item Every tight filter is a prime filter.
\end{enumerate}
\end{lemma}
\begin{proof} 
(1) This follows by \cite[Proposition~5.10]{LL}.

(2) Let $A$ be a tight filter and suppose that $\bigvee_{i=1}^{m} a_{i} \in A$.
Put $a = \bigvee_{i=1}^{m} a_{i}$ so that $a \in A$. 
We claim that $\{a_{1}, \ldots, a_{m}\}$ is a tight cover of $a$.
To see why, let $0 < z \leq a$.
Then $z \wedge a \neq 0$.
It follows, by \cite[Lemma 2.5]{Lawson2016} that $z \wedge a_{i} \neq 0$ for some $i$. 
Since $A$ is a tight filter it follows that $a_{i} \in A$ for some $i$.
\end{proof}

We can characterize the tight filters amongst the prime filters in the case interesting to us.

\begin{lemma}\label{lem:isla} 
Let $S$ be a distributive inverse semigroup in which $\equiv$ is idempotent-pure.
Let $X$ be a prime filter of $S$.
Then:
\begin{enumerate}
\item $X$ is a tight filter if and only if  $x \in X$ and $y \leq_{e} x$ implies that $y \in X$.
\item $X$ is a tight filter if and only if  $x \in X$ and $y \equiv x$ implies that $y \in X$.
\end{enumerate}
\end{lemma}
\begin{proof} 
(1) Suppose first that $X$ is a tight filter and that  $x \in X$ and $y \leq_{e} x$ .
Then $\{y\}$ is a tight cover of $x$.
It follows that $y \in X$.
Conversely, suppose that  $x \in X$ and $y \leq_{e} x$ implies that $y \in X$.
Using \cite[Lemma~2.5(4)]{Lawson2016}, if $\{a_{1}, \ldots, a_{m}\}$ is a tight cover of $a$
then $\bigvee_{i=1}^{m} a_{i} \leq_{e} a$.
It follows that  $\bigvee_{i=1}^{m} a_{i} \in X$.
But $X$ is a prime filter and so $a_{i} \in X$ for some $i$.
It follows that $X$ is a tight filter.

(2) Suppose that $X$ is a tight filter and that $x \in X$ and $y \equiv x$.
By  Lemma~\ref{lem:new-characterization}, there exists $z \leq_{e} x,y$.
By (1), we have that $z \in X$ and cince $z \leq y$ and $X$ is a filter it follows that $y \in X$.
Conversely, suppose that   $x \in X$ and $y \equiv x$ implies that $y \in X$.
The proof that $X$ is a tight filter follows from the same argument used in the proof of (1).
\end{proof}

To prove that a distributive inverse semigroup is Boolean,
we have to prove, by \cite[Lemma~3.20]{LL}, that every prime filter is an ultrafilter.
By Lemma~\ref{lem:idpt-pure-dist}, if $S$ is distributive and $\equiv$ is idempotent-pure
then $S/{\equiv}$ is distributive.
The following theorem is now relevant.

\begin{theorem}\label{them:seven} Let $S$ be a distributive inverse semigroup on which $\equiv$ is idem\-potent-pure.
Then $S/{\equiv}$ is Boolean if and only if every tight filter in $S$ is an ultrafilter.
\end{theorem}
\begin{proof} Put $T = S/{\equiv}$ and let $\theta \colon S \rightarrow T$ be the associated natural map.
Suppose first that $T$ is Boolean.
Then every prime filter of $T$ is an ultrafilter in $T$.
Let $A \subseteq S$ be a tight filter.
We shall prove that it is an ultrafilter in $S$.
Observe that $A' = \theta (A)^{\uparrow}$ is a proper filter in $T$ because $\theta$ is $0$-restricted.
We shall prove that it is a prime filter in $T$.
Suppose that $a' \vee b' \in A$.
Then there is $a \in A$ such that $\theta (a) \leq a' \vee b'$.
Put $e' = \mathbf{d} (\theta (a))$.
Then $\theta (a) = a'e' \vee b'e'$.
Put $a'' = a'e'$ and $b'' = b'e'$.
Then $\theta (a) = a'' \vee b''$ where $a'' \leq a'$ and $b'' \leq b'$.
Let $a_{1}$ and $b_{1}$ be any elements of $S$ such that $\theta (a_{1}) = a''$ and $\theta (b_{1}) = b''$.
Since the map $\theta$ is idempotent-pure we have that $a_{1} \sim b_{1}$.
It follows that $a_{1} \vee b_{1}$ is defined and  $\theta (a_{1} \vee b_{1}) = \theta (a)$.
But $a \in A$ so that by Lemma~\ref{lem:isla}, it follows that $a_{1} \vee b_{1} \in A$.
We proved in Lemma~\ref{lem:beverly} that every tight filter was a prime filter.
It follows that $a_{1} \in A$ or $b_{1} \in A$.
Without loss of generality, assume that $a_{1} \in A$.
Then $a'' = \theta (a_{1}) \in \theta (A)$
and so $a' \in A'$.
We have therefore proved that $A'$ is a prime filter.
By assumption, $A'$ is an ultrafilter.
We shall prove that $A$ is an ultrafilter.
Suppose that $A \subseteq B$ where $B$ is any proper filter.
Then $\theta (A)^{\uparrow} \subseteq \theta (B)^{\uparrow}$.
By assumption, $\theta (A)^{\uparrow} = \theta (B)^{\uparrow}$.
Now, let $b \in B$.
Then $\theta (b) \in  \theta (B)^{\uparrow}$.
Thus $\theta (b) \in  \theta (A)^{\uparrow}$.
It follows that there is $a \in A$ such that $\theta (a) \leq \theta (b)$.
Put $e' = \mathbf{d}(\theta (a))$ and let $e$ be any idempotent such that $\theta (e) = e'$.
Then $\theta (a) = \theta (be)$.
 But $a \in A$ and so $be \in A$ by  Lemma~\ref{lem:isla}.
But $A$ is a filter and so $b \in A$.
We have prove that $A = B$ and so $A$ is an ultrafilter.

Conversely, suppose that every tight filter in $S$ is an ultrafilter.
We prove that $T$ is Boolean.
Let $A'$ be a prime filter in $T$.
Put $A = \theta^{-1}(A')$.
We prove first that $A$ is a proper filter.
This set does not contain zero because $\theta$ is $0$-restricted,
Suppose that $a \leq b$ where $a \in A$.
Then $\theta (a) \leq \theta (b)$.
By assumption, $\theta (a) \in A'$.
But $A'$ is a filter and so $\theta (b) \in A'$.
It follows that $b \in A$.
Let $a,b \in A$.
Then $\theta (a), \theta (b) \in A'$.
By assumption, there exists $c' \in A'$ such that $c' \leq \theta (a), \theta (b)$.
It follows that $\theta (a) \mathbf{d}(c') = \theta (b)\mathbf{d}(c')$.
Let $e$ be any idempotent such that $\theta (e) = \mathbf{d}(c')$. 
Then $\theta (ae) = \theta (be)$.
But $\theta$ is idempotent pure and so $ae \sim be$.
By Lemma~\ref{lem:compatibility-meets}, $d = ae \wedge be$ exists and is defined purely algebraically.
Thus $\theta (d) = c$.
It follows that $d \in A$.
On the other hand, $d \leq a,b$.
We now prove that $A$ is a tight filter.
Suppose that $a \in A$ and $\{a_{1}, \ldots ,a_{m}\} \subseteq a^{\downarrow}$ is a tight cover.
Then $\bigvee_{i=1}^{m} a_{i} \leq_{e} a$.
We now use Lemma~\ref{lem:isla}, to deduce that
$\bigvee_{i=1}^{m} \theta (a_{i}) = \theta (a)$.
By assumption, $A'$ is a prime filter and so $\theta (a_{i}) \in A$ for some $i$.
This means that $a_{i} \in A$ for some $i$.
We have therefore proved that $A$ is a tight filter.
By assumption it is an ultrafilter.
Thus $\theta^{-1}(A')$ is an ultrafilter.
We now prove that $A'$ is an ultrafilter.
Suppose that $A' \subseteq B'$ where $B'$ is a proper filter.
Then $\theta^{-1}(A') \subseteq \theta^{-1}(B')$.
But, as above,  $\theta^{-1}(B')$ is a proper filter.
It follows that $\theta^{-1}(A') = \theta^{-1}(B')$.
Thus $A' = B'$ and so $A'$ is also an ultrafilter.
\end{proof}



The following result shows that to check whether every tight filter is an ultrafilter,
it is enough to restrict attention to the distributive lattice of idempotents.

\begin{proposition}\label{prop:idempotents} Let $S$ be a distributive inverse semigroup.
Then each tight filter in $S$ is an ultrafilter in $S$
if and only if
each tight filter in $\mathsf{E}(S)$ is an ultrafilter in $\mathsf{E}(S)$.
\end{proposition}
\begin{proof} We first prove first that $A$ is a prime filter in $S$ if and only if $\mathsf{E}(\mathbf{d}(A))$
is a prime filter in $S$.
Suppose that $A$ is a prime filter in $S$.
Observe that $\mathbf{d}(A)$ is a proper filter in $S$.
We prove that $\mathbf{d}(A)$ is a prime filter in $S$.
Let $x \vee y \in \mathbf{d}(A)$.
Then $a^{-1}a \leq x \vee y$ for some $a \in A$.
We have that $a^{-1}a = xa^{-1}a \vee ya^{-1}a$.
Whence $a = axa^{-1}a \vee aya^{-1}a$.
But $A$ is prime.
Without loss of generality, we may assume that $axa^{-1}a \in A$.
Thus $ax \in A$.
Now, $a^{-1}ax \in \mbox{d}(A)$.
It follows that $x \in \mathbf{d}(A)$.
We now prove the converse.
Suppose that $\mbox{d}(A)$ is prime.
Let $x \vee y \in A$.
Then $\mathbf{d}(x) \vee \mathbf{d}(y) \in \mathbf{d}(A)$.
But $\mathbf{d}(A)$ is prime.
Without loss of generality, we may assume that $\mathbf{d}(x) \in \mathbf{d}(A)$.
It follows that $(x \vee y)\mathbf{d}(x) \in A$ and so $x \in A$.
It is routine to check that $A$ is an ultrafilter if and only if $\mathbf{d}(A)$ is an ultrafilter.
The fact that $A$ is a tight filter if and only if $\mathbf{d}(A)$ is a tight filter follows by \cite[Lemma~5.9(1)]{LL}.
By \cite[Lemma 5.9(2)]{LL}, we have that $\mathbf{d}(A)$ is a tight filter (respectively, ultrafilter) if and only if $\mathsf{E}(\mathbf{d}(A))$ is a tight filter (respectively, ultrafilter)
in $\mathsf{E}(S)$.
We can now prove the proposition.
Suppose that each tight filter in $\mathsf{E}(S)$ is an ultrafilter in $\mathsf{E}(S)$.
Let $A$ be a tight filter in $S$.
Then $\mathbf{d}(A)$ is a tight filter in $S$.
Thus $\mathsf{E}(\mathbf{d}(A))$ is a tight filter in $\mathsf{E}(S)$.
By assumption, $\mathsf{E}(\mathbf{d}(A))$ is an ultrafilter in $\mathsf{E}(S)$.
Thus $\mathbf{d}(A)$ is an ultrafilter in $S$, and so $A$ is an ultrafilter in $S$.
The proof of the converse is now straightforward.
\end{proof}

Let $C$ be a higher rank graph with a finite number of identities which has no sources and is row finite.
Using Theorem~\ref{them:seven} and Proposition~\ref{prop:idempotents},
to prove that $\mathsf{B}(C)$ is Boolean, we need to prove that every tight filter in $\mathsf{E}(\mathsf{R}(C))$
is an ultrafilter in $\mathsf{E}(\mathsf{R}(C))$.
To do this, we shall relate filters in  $\mathsf{E}(\mathsf{R}(C))$ to appropriate subsets of $C$.
It is here that we make particular use of ideas to be found in \cite{JS2014}.

Let $C$ be a category.
Let $A \subseteq C$ be a non-empty subset.
We say that $A$ is a {\em prefilter} if $a,b \in A$ then there $aC \cap bC \cap A \neq \varnothing$.
Observe that in a prefilter each pair of elements is dependent but that something stronger is true:
namely, for all $a,b \in A$ there exist $u, v \in C$ such that $au = bv \in A$.
A prefilter $A$ is called a {\em filter} if it satisfies the additional condition that if  $a = xy$ and $a \in A$ then $x \in A$.
We call $x$ a {\em prefix} of $a$.
If $X$ is any non-empty subset of $C$ define
$\mathsf{Pref}(X)$ to be all elements $x \in C$ such that $a = xu$ for some $a \in X$ and $u \in C$;
thus $\mathsf{Pref}(X)$ is the set of all prefixes of elements of $X$.
Observe that $X \subseteq \mathsf{Pref}(X)$ .
We shall need the following \cite[Lemma 7.3]{JS2014} which we state using the language introduced here.

\begin{lemma}\label{lem:spielberg} 
Let $C$ be a countable, finitely aligned conical category.
Let $A$ be a prefilter in $C$ and let $x$ be any element which is dependent on every element of $A$.
Then there is a prefilter $B$ such that $A \subseteq B$ and $x \in B$.
\end{lemma}

The following is simply a version of \cite[Remark 7.2]{JS2014}.

\begin{lemma}\label{lem:lois} Let $C$ be a category.
If $A$ is a prefilter then $\mathsf{Pref}(A)$ is a filter.
\end{lemma}
\begin{proof} Let $x,y \in \mathsf{Pref}(A)$.
Then, by definition, there exist $a,b \in A$ such that $a = xu$ and $b = yv$ for some $u,v \in C$.
Since, by assimption, $A$ is a prefilter we can find $p,q \in C$ such that $ap = bq \in A$.
Thus $x(up) = y(vq) \in A \subseteq \mathsf{Pref}(A)$. 
We have therefore proved that $\mathsf{Pref}(A)$ is a prefilter.
It is a filter by virtue of its very construction.
\end{proof}

The proof of the following is by induction.

\begin{lemma}\label{lem:anne} Let $A$ be a filter in the category $C$.
Then for any finite non-empty subset $\{a_{1}, \ldots, a_{m}\} \subseteq A$
there  are elements $u_{1}, \ldots, u_{m}$ such that
$a_{1}u_{1} = \ldots = a_{m}u_{m} \in A$.
\end{lemma}

\begin{lemma}\label{lem:filter-property} Let $C$ be a finitely aligned category.
Let $A$ be a filter in $C$ and let $a,b \in A$.
Suppose that $\varnothing \neq aC \cap bC = \{c_{1}, \ldots, c_{n}\}C$.
Then $c_{i} \in A$ for some $i$.
\end{lemma}
\begin{proof} Let $A$ be a filter where $a,b \in A$.
Then $a$ and $b$ are dependent.
Thus, by definition, there exist $u,v \in C$ such that $au = bv$.
Put $z = au = bv$.
Then $z \in aC \cap bC$.
It follows that $z = c_{i}p$ for some $i$ and some $p \in C$.
But $z \in A$ and $A$ is a filter and so $c_{i} \in A$.
\end{proof}

The following lemma is used below.

\begin{lemma}\label{lem:nsw} Let $S$ be a distributive inverse semigroup.
Let $\mathcal{B} \subseteq S$ be such that every element of $S$ is a finite join of elements of $\mathcal{B}$
and if $a \leq \bigvee_{i=1}^{m} a_{i}$ where $a,a_{i} \in \mathcal{B}$ then $a \leq a_{i}$ for some $i$.
Let $A$ be a proper filter such that for each $a \in A$, where $a \in \mathcal{B}$, where
$\{a_{1}, \dots, a_{m}\}$ is a tight cover of  $a$,
and  where $a_{1}, \ldots, a_{m} \in \mathcal{B}$, implies that $a_{i} \in A$ for some $i$.
Then $A$ is a tight filter.
\end{lemma}
\begin{proof} Let $s \in A$ such that $\{b_{1}, \ldots, b_{n}\}$ is a tight cover of $s$.
We shall prove that $b_{i} \in A$ for some $i$.
By Lemma~\ref{lem:blackadder}, we can assume that $b_{1}, \ldots, b_{n} \in \mathcal{B}$.
Let $0 < t \leq s$.
By assumption, we can write $t = \bigvee_{j=1}^{n} a_{j}$ where $a_{j} \in \mathcal{B}$.
In particular, $0 < a_{1} \leq s$.
By assumption, there is some $b_{i}$ such that $a_{1} \wedge b_{i} \neq 0$.
Using \cite[Lemma 2.5]{Lawson2016} again, we deduce that $t \wedge b_{i} \neq 0$.
\end{proof}

We shall now invoke the notion of a large subset introduced in Section~3.1
Let $A \subseteq C$ be a filter.
We say that it is {\em tight} if the following condition holds:
if $a \in A$ and $\{a_{1}, \ldots, a_{m}\}$ is a large subset in $aC$, then $a_{i} \in A$ for some $i$.

\begin{lemma}\label{lem:correspondence} Let $C$ be a finitely aligned cancellative conical category with a finite number of identities.
Then $\{a_{1}a_{1}^{-1}, \ldots, a_{m}a_{m}^{-1}\}$ is a tight cover of $aa^{-1}$ 
if and only if $\{a_{1}, \ldots, a_{m}\}$ is a large subset in $aC$.
\end{lemma}
\begin{proof} There is no loss of generality in working with basic morphisms.
Suppose that  $\{a_{1}, \ldots, a_{m}\}$  is a large subset in $aC$.
Let $0 < bb^{-1} \leq aa^{-1}$.
Then $b = ap$ for some $p$.
Thus $b \in aC$.
It follows that $b$ is comparable with some element $a_{i}$.
Thus $bu = a_{i}v$ for some $u,v \in C$.
Put $z = bu = a_{i}v$.
Then $zz^{-1} \leq bb^{-1}, a_{i}a_{i}^{-1}$ and, of course, is non-zero.
We have proved that 
$\{a_{1}a_{1}^{-1}, \ldots, a_{m}a_{m}^{-1}\} \subseteq (aa^{-1})^{\downarrow}$ is a tight cover. 
Conversely, suppose that $\{a_{1}a_{1}^{-1}, \ldots, a_{m}a_{m}^{-1}\} \subseteq (aa^{-1})^{\downarrow}$ is a tight cover.
Let $b \in aC$.
Then $b = ap$.
Thus $bb^{-1} \leq aa^{-1}$.
It follows that there is $zz^{-1} \leq bb^{-1}, a_{i}a_{i}^{-1}$ for some $i$.
Thus $z = bu = a_{i}v$.
It follows that $\{a_{1}, \ldots, a_{m}\}$ is a large subset of $aC$.
\end{proof}

The following result is key to our proof that $\mathsf{B}(C)$ is Boolean.
we relate diffent kinds of filters in $\mathsf{E}(\mathsf{R}(C))$ to different kinds of filters in $C$.
By Proposition~\ref{prop:idempotents}, it is enough for us to regard filters inside the semilattice of idempotents.

\begin{proposition}\label{prop:mo} Let $C$ be a strongly finitely aligned cancellative conical category with a finite number of identities.
Given a filter $A \subseteq C$, define
$$
\mathcal{P}(A) = \{xx^{-1} \colon x \in A\}^{\uparrow} \cap \mathsf{E}(\mathsf{R}(C)).
$$
The inverse of $\mathcal{P}$ is given by
$$
\mathcal{F}(P) = \{x \in C \colon xx^{-1} \in P \}
$$
for each prime filter $P$ in $\mathsf{E}(\mathsf{R}(C))$.
Then $\mathcal{P}$ determines a bijective correspondence between:
\begin{itemize}
\item Filters in $C$ and prime filters in $\mathsf{E}(\mathsf{R}(C))$.
\item Tight filters in $C$ and tight filters in  $\mathsf{E}(\mathsf{R}(C))$.
\item Maximal filters in $C$ and ultrafilters in  $\mathsf{E}(\mathsf{R}(C))$.
\end{itemize}
\end{proposition}
\begin{proof} {\em With each filter in $C$, we associate a prime filter in $\mathsf{E}(\mathsf{R}(C))$.}
Let $A \subseteq C$ be a filter in $C$. 
We claim that  $\mathcal{P}(A)$ is a proper filter in $\mathsf{E}(\mathsf{R}(S))$.
Let $e, f \in \mathcal{P}(A)$.
Then there exist $x,y \in A$ such that 
$xx^{-1} \leq e$ and $yy^{-1} \leq f$.
However, 
because of Lemma~\ref{lem:key-property}, whenever we have a relation $xx^{-1} \leq \bigvee_{i=1}^{m} a_{i}a_{i}^{-1}$ then,
in fact, $xx^{-1} \leq a_{i}a_{i}^{-1}$ for some $i$.
We may therefore assume, without loss of generality, that $e = aa^{-1}$ and $f = bb^{-1}$ for some
$aa^{-1}, bb^{-1} \in \mathcal{P}(A)$.
It follows that $x = ap$ and $y = bq$ for some $p,q \in C$.
But $A$ is a filter and $a,b \in A$.
By assumption,  $aC \cap bC \cap A \neq \varnothing$.
Thus $au = bv \in A$ for some $u,v \in C$.
Put $z = au = bv$.
Then $zz^{-1} \in \mathcal{P}(C)$ and $zz^{-1} \leq aa^{-1}, bb^{-1}$
and so $zz^{-1} \leq e,f$.
It is now clear that  $\mathcal{P}(A)$ is a filter in $\mathsf{E}(\mathsf{R}(S))$,
and it is a proper filter by construction.
We prove that $\mathcal{P}(A)$ is, in fact, a prime filter.
Suppose that $\bigvee_{i=1}^{m} x_{i}x_{i}^{-1} \in  \mathcal{P}(A)$.
Then $xx^{-1} \leq \bigvee_{i=1}^{m} x_{i}x_{i}^{-1}$ for some $x \in A$.
Thus by Lemma~\ref{lem:key-property}, we have that $xx^{-1} \leq x_{i}x_{i}^{-1}$ for some $i$.
Hence $x = x_{i}p$ for some $p \in C$.
Since $x \in A$ it follows that $x_{i} \in A$ and so $x_{i}x_{i}^{-1} \in  \mathcal{P}(A)$.
{\em With each prime filter in $\mathsf{E}(\mathsf{R}(C))$, we associate a filter in $C$.}
Let $P$ be a prime filter in $\mathsf{E}(\mathsf{R}(C))$.
Define $\mathcal{F}(P) = \{ x \in C \colon xx^{-1} \in P\}$.
Observe that we need $P$ to be a prime filter, in order that $\mathcal{F}(P)$ be non-empty
since every idempotent of $\mathsf{P}(C)$ is of the form $\bigvee_{i=1}^{m} x_{i}x_{i}^{-1}$.
Let $x,y \in \mathcal{F}(P)$.
Then $xx^{-1}, yy^{-1} \in P$.
There is an element of $P$ below these two elements;
thus $\bigvee_{i=1}^{m} x_{i}x_{i}^{-1} \in P$ and  $\bigvee_{i=1}^{m} x_{i}x_{i}^{-1} \leq xx^{-1},yy^{-1}$.
Because $P$ is a prime filter, we must have $x_{i}x_{i}^{-1} \in P$ for some $i$.
But $x_{i}x_{i}^{-1} \in P$ and $x_{i}x_{i}^{-1} \leq xx^{-1},yy^{-1}$.
It follows that $x_{i} = xp$ and $x_{i} = ys$ for some $p,s \in C$.
It follows that $xp = ys \in \mathcal{F}(P)$.
we have therefore shown that $\mathcal{F}(P)$ is a prefilter.
We now prove that it is a filter.
Suppose that $x \in  \mathcal{F}(P)$ and $x = yu$.
Then $xx^{-1} \leq yy^{-1}$.
But $P$ is a filter and so $yy^{-1} \in P$ and so $y \in \in \mathcal{F}(P)$.
We have therefore proved that $\mathcal{F}(P)$ is a filter in $C$.
{\em The maps $A \mapsto \mathcal{P}(A)$ and $P \mapsto \mathcal{F}(P)$
are mutually inverse and order preserving.}
It is evident that these two constructions are order preserving.
We now prove that they are mutually inverse.
Clearly, $A \subseteq \mathcal{F}(\mathcal{P}(A))$.
Let $x \in \mathcal{F}(\mathcal{P}(A))$.
Then $yy^{-1} \leq xx^{-1}$ where $y \in A$.
Then $y = xp$.
Thus $x \in A$.
It follows that $A = \mathcal{F}(\mathcal{P}(A))$.
Clearly, $\mathcal{P}(\mathcal{F}(P)) \subseteq P$.
Let $\bigvee_{i=1}^{m} x_{i}x_{i}^{-1} \in P$.
Then $x_{i}x_{i}^{-1} \in P$ for some $i$.
Thus $x_{i} \in \mathcal{F}(P)$ and so $\bigvee_{i=1}^{m} x_{i}x_{i}^{-1} \in \mathcal{P}(\mathcal{F}(P))$.
It follows that $\mathcal{P}(\mathcal{F}(P)) = P$.

We have therefore established an order isomorphism between the set of
filters in $C$ and the set of prime filters in  $\mathsf{E}(\mathsf{R}(C))$.

Suppose, now, that $A$ is a maximal filter in $C$.
We claim that $\mathcal{P}(A)$ is an ultrafilter in $\mathsf{E}(\mathsf{R}(C))$.
There are two ways to prove this.
First, we use Zorn's lemma.
By Lemma~\ref{lem:proper-filters-uf}, we can suppose that  
$\mathcal{P}(A) \subseteq U$ where $U$ is an ultrafilter in $\mathsf{E}(\mathsf{R}(C))$.
But ultrafilters are prime by Lemma~\ref{lem:beverly}.
Thus we may apply $\mathcal{F}$ to $U$ and we get
$A \subseteq \mathcal{F}(U)$.
However, we are assuming that $A$ is a maximal filter and $\mathcal{F}(U)$ is a filter.
It follows that $A = \mathcal{F}(U)$.
Now apply $\mathcal{P}$ to both sides to deduce that  $\mathcal{P}(A) = U$.
Our second proof uses Lemma~\ref{lem:uf-exel};
we shall prove that $\mathcal{P}(A)$ is an ultrafilter using this result.
Suppose that $e = \left( \bigvee_{i=1}^{m} x_{i}x_{i}^{-1} \right) \wedge \mathcal{P}(A) \neq 0$.
Assume first that for each $x_{i}x_{i}^{-1}$ where $1 \leq i \leq m$ there is $a_{i} \in A$
such that $x_{i}x_{i}^{-1} \wedge a_{i}a_{i}^{-1} = 0$.
Then $x_{i}C \cap a_{i}C = \varnothing$.
By assumption, $a_{1}, \ldots, a_{m} \in A$.
But $A$ is a filter in $C$.
Thus, by Lemma~\ref{lem:anne}, there exist elements $u_{1},\ldots, u_{m} \in C$
such that $z = a_{1}u_{1} = \ldots = a_{m}u_{m} \in A$.
Thus $zz^{-1} \in \mathcal{P}(A)$.
It follows that $zC \cap x_{i}C \neq \varnothing$ for some $i$.
But then $a_{i}u_{i}C \cap x_{i}C \neq \varnothing$  which is a contradiction.
There is therefore an $x_{i}$ which is dependent on every element of $A$.
But $A$ is a maximal filter.
Thus $x_{i} \in A$ by Lemma~\ref{lem:spielberg}.
It follows that $x_{i}x_{i}^{-1} \in  \mathcal{P}(A)$ and so $e \in  \mathcal{P}(A)$.
We deduce that $\mathcal{P}(A)$ is an ultrafilter.
We now go in the opposite direction.
Suppose that $U$ is an ultrafilter in $\mathsf{E}(\mathsf{R}(C))$.
then $\mathcal{F}(U)$ is a filter in $C$.
We claim that it is maximal.
Suppose that $\mathcal{F}(U) \subseteq A$, where $A$ is a filter in $C$.
Then $U \subseteq \mathcal{P}(A)$.
But $U$ is an ultrafilter.
It follows that $U = \mathcal{P}(A)$.
Now apply $\mathcal{F}$ to both sides.
We get that $\mathcal{F}(U) = A$,
which proves that $\mathcal{F}(U)$ is a maximal filter.

We conclude by showing that under this correspondence tight filters correspond to tight filters.
We shall use Lemma~\ref{lem:nsw}.
Suppose now that $A$ is a tight filter in $C$.
We prove that $\mathcal{P}(A)$ is a tight filter in $\mathsf{E}(\mathsf{R}(C))$.
Let $\{a_{1}a_{1}^{-1}, \ldots, a_{m}a_{m}^{-1}\}$ be a tight cover of $xx^{-1}$ where $xx^{-1} \in \mathcal{P}(A)$.
Then $\{a_{1}, \ldots, a_{m}\}$ is large in $xC$.
By assumption, $a_{i} \in A$ for some $i$.
Thus $a_{i}a_{i}^{-1} \in \mathcal{P}(A)$ for some $i$.
It follows that  $\mathcal{P}(A)$ is a tight filter in $\mathsf{E}(\mathsf{R}(C))$.
We now prove the converse.
Let  $P$ be a tight filter in $\mathsf{E}(\mathsf{R}(C))$.
We prove that $\mathcal{F}(P)$ is a tight filter in $C$.
Let $x \in \mathcal{F}(P)$ and suppose that $\{a_{1}, \ldots, a_{m}\}$ is large in $xC$.
Then by Lemma~\ref{lem:correspondence}, we have that
$\{a_{1}a_{1}^{-1}, \ldots, a_{m}a_{m}^{-1}\}$ is a tight cover of $xx^{-1}$.
But $xx^{-1} \in P$ and $P$ is a tight filter and so $a_{i}a_{i}^{-1} \in P$ for some $i$.
It follows that $a_{i} \in \mathsf{F}(P)$ for some $i$, as required.
\end{proof}

Let $C$ be a higher rank graph.
By Theorem~\ref{them:seven},
to prove that $\mathsf{B}(C)$ is Boolean, we have to prove that every tight filter in $\mathsf{R}(C)$ is an ultrafilter.
By Proposition~\ref{prop:idempotents}, to prove this we need to prove that every tight filter in
$\mathsf{E}(\mathsf{R}(C))$ is an ultrafilter in $\mathsf{E}(\mathsf{R}(C))$.
To do this, by Proposition~\ref{prop:mo}, we need to prove that every tight filter in $C$ is a maximal filter in $C$.

To that end, we make the following definition.
A subset $A \subseteq C$ is said to be {\em good}\footnote{For want of a better word.} if it has the following two properties:
\begin{enumerate}
\item Any two elements of $A$ are dependent.
\item For each $\mathbf{m} \in \mathbb{N}^{k}$
there exists $a \in A$ such that $d(a) = \mathbf{m}$.
\end{enumerate}

\begin{remark}\label{rem:good-properties}
{\em If $A$ is a good subset of $C$ then there is an identity $e$ in $C$ such that
$A \subseteq eC$. This follows from the fact that any two elements in $A$ are dependent.
In fact, $e \in A$ since $e$ is the unique element of $A$ such that $d(e) = \mathbf{0}$.
In addition, we have the following.
If $a,b \in C$ and $d(a) = d(b)$ then, in fact, $a = b$ since,
being dependent, we have that $au = bv$ for some $u,v \in C$ and then we apply
Lemma~\ref{lem:levi}.}
\end{remark}

\begin{lemma}\label{lem:good-filter} Let $C$ be a finitely aligned $k$-graph which has a finite number of identities is row finite and has no sources.
Then good subsets are the same thing as tight filters.
\end{lemma}
\begin{proof} 
{\em Good subsets are tight filters. }We begin by proving that good subsets are filters.
Let $A$ be a good subset of $C$ and let $x,y \in A$.
Because $x$ and $y$ are dependent, we can find $u,v \in C$ such that  $xu = yv$.
Put $z = xu = yv$.
Let $z' \in A$ such that $d(z') = d(z)$.
Since $z',x \in A$ we have that $z's = xt$ for some $s,t \in C$.
Now, $d(z') = d(z) = d(x) + d(u)$ thus $d(z') \geq d(x)$.
It follows by  Lemma~\ref{lem:levi} that $z' = xp$ for some $p \in C$.
Likewise, $z' = yq$ for some $q \in C$.
Thus $z' = xu = yq$.
It follows that $z' \in xC \cap yC \cap A$.
Let $x = yz$ where $x \in A$.
Let $y' \in A$ be the unique element such that $d(y') = d(y)$.
Since $x,y' \in A$ we have that $xu = y'v$ for some $u,v \in C$.
But then $yzu = y'v$.
However, $d(y) = d(y')$, thus by Lemma~\ref{lem:levi}, it follows that $y = y'$ and so $y \in A$, as claimed.
We have therefore proved that every good subset is a filter.
We can now prove that good subsets are tight filters.
Let $A$ be a good subset.
Let $a \in A$ and suppose that $\{a_{1}, \ldots, a_{m}\}$ is large in $aC$.
Observe that $d(a_{i}) \geq d(a)$ for all $i$
since $a_{i} = ap_{i}$ for some $p_{i} \in C$.
Put $\mathbf{m} = \bigvee_{i=1}^{m} d(a_{i})$.
By assumption, there is a unique $z \in A$ such that $d(z) = \mathbf{m}$.
But $d(z) \geq d(a)$ and, since $a,z \in A$, they are comparable.
It follows by Lemma~\ref{lem:levi} that $z = as$ for some $s \in C$.
By assumption, $z$ is comparable with some element $a_{i}$.
But $d(z) \geq d(a_{i})$.
Thus $z = a_{i}t$ for some $t$ by Lemma~\ref{lem:levi}.
But $A$ is a filter, $z \in A$ and so $a_{i} \in A$, as required.
{\em Tight filters are good subsets. }Let $A$ be a tight filter.
We prove that it is a good subset.
We only need to prove one thing.
Let $\mathbf{m} \in \mathbb{N}^{k}$ be arbitrary.
Let $a \in A$ be arbitrary.
If $d(a) \geq \mathbf{m}$ then $a = ya'$ where $d(y) = \mathbf{m}$.
But $A$ is a filter, and so $y \in A$ and we are done.
In what follows we may therefore assume that $d(a) \ngeq \mathbf{m}$.
Put $\mathbf{n} = d(a) \vee \mathbf{m}$,
and so $\mathbf{n} > d(a)$.
Let $p_{1}, \ldots, p_{s} \in C$ be all the elements such that
$d(ap_{i}) = \mathbf{n}$.
(We assume that there are only finitely many
which is fine since they all have the same degree $\mathbf{n} - \mathbf{m}$.)
We shall prove that $\{ap_{1}, \ldots, ap_{s}\}$ is large in $aC$.
The elements $p_{1}, \ldots, p_{s}$ have the same range, $e$ say,
and they are all the elements in $eC$ with degree $\mathbf{n} - \mathbf{m}$.
Let $z \in aC$ be arbitrary.
Then $z = ap$ for some $p \in C$.
Observe that $p \in eC$.
Choose $u$ such that $d(pu) \geq \mathbf{n} - \mathbf{m}$.
Then $pu = p_{k}v$ for some $v \in C$ and some $k$.
Thus $apu = ap_{k}v$ and so $zu = ap_{k}v$.
We have proved that $z$ is dependent on $ap_{k}$. 
Thus $\{ap_{1}, \ldots, ap_{s}\}$ is large in $aC$.
It follows that $ap_{j} \in A$ for some $j$ since $A$ is a tight filter.
But $d(ap_{j}) = \mathbf{n} = \mathbf{m} + (\mathbf{n} - \mathbf{m})$.
Thus $A$ contains an element of degree $\mathbf{m}$ by the (UFP) and the properties of filters.
\end{proof}

\begin{lemma}\label{lem:good-filter-more} 
Let $C$ be a finitely aligned $k$-graph which has a finite number of identities is row finite and has no sources.
Then good subsets are the same thing as maximal filters.
\end{lemma}
\begin{proof} {\em Good subsets are maximal filters. }Let $A$ be a good subset and suppose that $A \subseteq B$ where $B$ is a filter.
Let $b \in B$.
Then, by assumption, there exists $a \in A$ such that $d(a) = d(b)$.
But $a,b \in B$ and $B$ is a filter and so $a$ and $b$ are comparable.
It follows by Lemma~\ref{lem:levi} that $a = b$.
Thus $b \in A$.
It follows that $A = B$ and so $A$ is maximal.
{\em Maximal filters are good subsets. }
We prove this result directly although it follows from Proposition~\ref{prop:mo}.
In fact, we prove that maximal filters are tight filters.
Let $A$ be a maximal filter.
Let $a \in A$ and suppose that $\{a_{1}, \ldots, a_{m}\}$ is large in $aC$.
For the sake of argument, suppose that none of $a_{1}, \ldots a_{m}$ belong to $A$.
We can actually assume that none of these elements is dependent on all the elements of $A$.
For, suppose that $a_{i}$ is dependent on every element of $A$.
Then, by Lemma~\ref{lem:spielberg}, there would be a filter $B$ that contains $A$ as a subset and $a_{i}$ as an element.
But $A$ is a maximal filter.
This implies that $a_{i} \in A$. which contradicts our assumption.
It follows that for each $a_{i} \notin A$, there is some $b_{i} \in A$
so that $a_{i}C \cap b_{i}C = \varnothing$.
Now, $a,b_{1}, \ldots, b_{m} \in A$.
Thus by Lemma~\ref{lem:anne},
we can find elements $u,u_{1}, \ldots, u_{m}$ such that
$au = b_{1}u_{1} = b_{2}u_{2} = \ldots = b_{m}u_{m} \in A$.
Put $z = au$.
We now use the fact that  $\{a_{1}, \ldots, a_{m}\}$ is large in $aC$.
It follows that $zs = a_{i}t$ for some $s,t$.
Thus $b_{i}u_{i}s = a_{i}t$.
But this says that $b_{i}$ and $a_{i}$ are dependent which is a contradiction.
It follows that some $a_{j} \in A$ and so $A$ is a tight filter.
\end{proof}

By
Theorem~\ref{them:seven},
Lemma~\ref{lem:good-filter},
Lemma~\ref{lem:good-filter-more},
and Proposition~\ref{prop:idempotents}
we deduce that $\mathsf{R}(C)/{\equiv}$ is a Boolean inverse monoid.
If we take into account Theorem~\ref{them:israel}, then we have proved the following.

\begin{theorem}\label{them:first-main} Let $C$ be a higher rank graph with a finite number of identities which has no sources and is row finite.
Then $\mathsf{B}(C)$ is a Boolean inverse monoid whose group of units is isomorphic to $\mathscr{G}(C)$.
\end{theorem}

\begin{examples} \mbox{}
{\em 
\begin{enumerate}

\item The monoid $\mathbb{N}$ under addition is isomorphic to the free monoid on one generator.
This monoid does not satisfy the conditions of \cite{LV2019b} but does satisfy the conditions of our paper.
The monoid $\mathsf{R}(\mathbb{N})$ is the bicyclic monoid with an adjoined zero.
It is a distributive inverse monoid since
the only way for a finite set of elements of the bicyclic monoid to be compatible is 
if they are $\leq$-related.
Thus, there will always be a largest element in any finite compatible subset
and this will be the join of the compatible subset.
The non-zero elements of our monoid are all essential and so the group is simply the 
maximum group homomorphic image of the bicyclic monoid which is $\mathbb{Z}$.
Observe that in this case, the Boolean inverse monoid is simply a group with an adjoined zero.
Thus the Boolean inverse monoids we obtain need not be particularly interesting.
For the bicyclic monoid and its properties, see \cite[Section 3.4, Section 5.4]{Lawson1998}. 

\item We now consider the free monoid on two generators.
Put $A = \{a,b\}$.
Then we want to describe the group associated with $A^{\ast}$.
This is the Thompson group $V$ as described in \cite{Lawson2007, Lawson2007b, Lawson2021, LS}.
It is the group of units of what we term the {\em Cuntz inverse monoid} $C_{2}$;
this is obtained by taking the distributive completion of the polycyclic monoid on two generators
and then factoring out by the congruence $\equiv$.

\item We can generalize the above two examples as follows.
Let $G$ be a finite directed graph where the in-degree if every vertex is at least 1.
Let $C$ be the free category generated by $G$.
The key property of free categories, like $C$, is that they are {\em rigid}:
this means that if $xC \cap yC \neq \varnothing$ then $xC \subseteq yC$ or $yC \subseteq xC$.
We shall use $d(x)$ to denote the length of the path $x$.
We can use a simple argument to prove that tight filters in $C$ are always maximal filters.
The maximal filters in $C$ are precisely the infinite filters (as can easily be proven).
Now, let $A$ be a tight filter.
Suppose that $x \in A$.
Let the identity at the domain of $x$ be $e$.
Let the edges into $e$ be $\{a_{1}, \ldots, a_{n}\}$, 
where, by assumption,  $n \geq 1$.
We claim that $\{xa_{1}, \ldots, xa_{n}\}$ is a large subset of $xC$.
Let $z \in xC$.
Without loss of generality,
we may assume that $z = xp$, where $p$ has length at least 1.
Clearly, $p$ must begin with one of the edges $a_{1}, \ldots, a_{n}$.
Without loss of generality, suppose that $p = a_{1}p'$ where $p' \in C$.
It follows that $z = xa_{1}p'$ and so is dependent on $xa_{1}$.
But, we have assumed that $A$ is a tight filter.
It follows that if $x \in A$ then $xa_{1} \in A$.
We deduce that every tight filter is infinite.
Thus every tight filter is a maximal filter.
\end{enumerate}

}
\end{examples}

\section{Properties of the Boolean inverse monoid $\mathsf{B}(C)$}

In this section, we shall determine when $\mathsf{B}(C)$ is simple as a Boolean inverse monoid.

\subsection{Aperiodicity}

Let $C$ be a $k$-graph.
We say that $C$ is {\em aperiodic}, following \cite{LS2010}, if for all $a,b \in C$ such that $a \neq b$ and $\mathbf{d}(a) = \mathbf{d}(b)$ there exists
an element $u \in C$ such that $au$ and $bu$ are independent.\\

We shall expore the algebraic meaning of aperiodicity in this section.
Let $S$ be a Boolean inverse monoid.
There is an action of $\mathsf{U}(S)$, the group of units of $S$, on $\mathsf{E}(S)$, the Boolean algebra of idempotents of $S$, given by $e \mapsto geg^{-1}$.
We call this the {\em natural action}.
The following was proved as \cite[Proposition~3.1]{Lawson2017}.

\begin{lemma}\label{lem:faithful} Let $S$ be a Boolean inverse monoid.
Then the natural action is faithful if and only if $S$ is fundamental.
\end{lemma}

\begin{remark}{\em The importance of Lemma~\ref{lem:faithful} is that it guarantees that
the group of units has a faithful representation in the group of automorphisms of the Boolean algebra.}
\end{remark}

The following is a souped up version of \cite[Lemma~3.2]{Lawson2017}.
In the proof of the lemma below, we use the result that in a Boolean algebra
$e \leq f$ if and only if $e\bar{f} = 0$.

\begin{lemma}\label{lem:plum}
Let $S$ be a Boolean inverse monoid.
Let $e$ be an idempotent and $a$ an element such that $ea \neq ae$.
Then there is a non-zero idempotent $f \leq e$ such that $f \perp afa^{-1}$.
\end{lemma}
\begin{proof} There are two cases.
{\bf Case~1}: Suppose that $e \overline{(aea^{-1})} \neq 0$.
Put $f = e \overline{(aea^{-1})} \leq e$.
Then
$f(afa^{-1}) =  e \overline{(aea^{-1})}a(e \overline{(aea^{-1})})a^{-1}$.
Thus
$$f(afa^{-1})
=  e \overline{(aea^{-1})}a(ea^{-1} a\overline{(aea^{-1})})a^{-1}
=  e \overline{(aea^{-1})}(aea^{-1})a \overline{(aea^{-1})} a^{-1}
= 0,$$
and the result is proved.
{\bf Case~2}: Now, suppose that $e \overline{(aea^{-1})} = 0$.
Then $e \leq aea^{-1}$ and so $ea \leq ae$.
For the sake of a contradiction, suppose that $aea^{-1}\overline{e} = 0$.
Then $aea^{-1} \leq e$ and so $ae \leq ea$.
It follows that $ea = ae$, which is a contradiction.
Thus, in fact,  $aea^{-1}\overline{e} \neq 0$.
Put $h = aea^{-1}\overline{e}$ and $f = a^{-1}ha$.
Suppose that $f = 0$.
Then $aha^{-1} = 0$.
But $aha^{-1} = h \neq 0$.
It follows that $f \neq 0$.
Also,
$f = a^{-1}ha = a^{-1}aea^{-1}\overline{e}a = ea^{-1}\overline{e}a \leq e$.
Finally, by direct computation, $afa^{-1}f = 0$.
\end{proof}

The following is now immediate by the above lemma when we put $b = af$.

\begin{corollary}\label{cor:plum} Suppose that $ae \neq ea$.
Then there is an infinitesimal $b$ such that $b \leq a$ and $b^{-1}b \leq e$.
\end{corollary}

\begin{proposition}\label{prop:infinitesimal} Let $S$ be a Boolean inverse monoid.
Then the following are equivalent:
\begin{enumerate}
\item $S$ is fundamental
\item Each non-idempotent element of $S$ is above an infinitesimal.
\end{enumerate}
\end{proposition}
\begin{proof} (1)$\Rightarrow$(2).
Let $a$ be a non-idempotent element.
Then, since $S$ is fundamental, there is an idempotent $e$ such that $ae \neq ea$.
Thus, by Corollary~\ref{cor:plum}, there is an infinitesimal $b \leq a$.

(2)$\Rightarrow$(1).
Let $a$ be a non-idempotent element.
Then $b \leq a$ where $b$ is an infinitesimal.
In particular, $b \neq 0$.
Suppose that $a$ commutes with $b^{-1}b$.
Then $ab^{-1}b = b^{-1}ba$.
But $ab^{-1}b = b$ because $b \le a$, and so $b = b^{-1}ba$.
It follows that $b^{2} = ba$ and so $ba = 0$.
But $b = ab^{-1}b = b^{-1}ba = 0$, which is a contradiction.
We have shown that $a$ cannot commute with $b^{-1}b$.
\end{proof}

\begin{lemma}\label{lem:needful} Let $C$ be a $k$-graph.
Then $C$ is aperiodic if and only if each non-idempotent basic morphism $ab^{-1}$ of $\mathsf{R}(C)$ lies above an infinitesimal basic morphism.
\end{lemma}
\begin{proof} Suppose that $C$ is aperiodic.
Let $ab^{-1}$ be a non-idempotent basic morphism.
Thus $a \neq b$ and $\mathbf{d}(a) = \mathbf{d}(b)$.
By assumption, there is an element $u \in C$ such that $au$ and $bu$ are independent.
Observe that $(au)(bu)^{-1} \leq ab^{-1}$ by Lemma~\ref{lem:oreo}.
To say that $au$ and $bu$ are independent means precisely that $(au)(au)^{-1} \perp (bu)(bu)^{-1}$ by Lemma~\ref{lem:oreo}.
Thus $(au)(bu)^{-1}$ is an infinitesimal and is below $ab^{-1}$.

Let $a \neq b$ and $\mathbf{d}(a) = \mathbf{d}(b)$.
Then $ab^{-1}$ is a non-idempotent basic morphism.
By assumption, there exists an infinitesimal $cd^{-1}$ such that $cd^{-1} \leq ab^{-1}$.
By Lemma~\ref{lem:oreo}, we have that $c = au$ and $d = bu$ for some $u \in C$.
By assumption, $(au)(au)^{-1} \perp (bu)(bu)^{-1}$ and so $au$ and $bu$ are incomparable.
We have therefore proved that $C$ is aperiodic.
\end{proof}

\begin{lemma}\label{lem:juice} Let $S$ be a Boolean inverse semigroup.
Then $\bigvee_{j=1}^{n} a_{j}$ an infinitesimal implies that each $a_{j}$ is an infinitesimal.
\end{lemma}
\begin{proof} By assumption, $\bigvee_{j=1}^{n} \bigvee_{i=1}^{n} a_{j}a_{i} = 0$.
In particular, $a_{j}^{2} = 0$ for each $1 \leq j \leq n$.
\end{proof}

The following theorem provides the algebraic meaning of being fundamental.

\begin{theorem}\label{them:fundamental} Let $C$ be a strongly finitely aligned higher rank graph with a finite number of identities which has no sources and is row finite.
Then the following are equivalent:
\begin{enumerate}
\item $\mathsf{B}(C)$ is fundamental.
\item $C$ is aperiodic.
\end{enumerate}
\end{theorem}
\begin{proof} (1)$\Rightarrow$(2).
We shall use Lemma~\ref{lem:needful} to prove that $C$ is aperiodic.
Let $ab^{-1}$ be a non-idempotent basic morphism of $\mathsf{R}(C)$.
Then $[ab^{-1}]$ is not an idempotent since $\equiv$ is idempotent-pure.
Thus, by Proposition~\ref{prop:infinitesimal}, $[\bigvee_{j=1}^{n}c_{j}d_{j}^{-1}] \leq [ab^{-1}]$
where $[\bigvee_{j=1}^{n}c_{j}d_{j}^{-1}]$ is an infinitesimal.
Thus by Lemma~\ref{lem:juice}, each  $[c_{j}d_{j}^{-1}]$ is an infinitesimal.
But then each $c_{j}d_{j}^{-1}$ is an infinitesimal, since $\equiv$ is $0$-restricted.
Observe that $(ab^{-1})(\bigvee_{j=1}^{n}d_{j}d_{j}^{-1})$ maps to $[\bigvee_{j=1}^{n}c_{j}d_{j}^{-1}]$.
Thus $(ab^{-1})(\bigvee_{j=1}^{n}d_{j}d_{j}^{-1})$ is an infinitesimal and lies below $ab^{-1}$.
By relabelling if necessary, we can assume that $(ab^{-1})d_{1}d_{1}^{-1}$ is non-zero and an infinitesimal in $\mathsf{R}(C)$ and lies below $ab^{-1}$.
But  $(ab^{-1})d_{1}d_{1}^{-1}$ is a join of basic morphisms and each of these basic morphisms must be either zero or an infinitesimal.
Pick a non-zero such basic morphism $cd^{-1}$.
Then $cd^{-1} \leq ab^{-1}$ and is an infinitesimal.
It follows by Lemma~\ref{lem:needful} that $C$ is aperiodic.

(2)$\Rightarrow$(1).
Let $[\bigvee_{i=1}^{m} a_{i}b_{i}^{-1}]$ be a non-idempotent element of $\mathsf{B}(C)$.
Then the element $\bigvee_{i=1}^{m} a_{i}b_{i}^{-1}$ is a non-idempotent in $\mathsf{R}(C)$ since $\equiv$ is idempotent-pure.
It follows that for some $i$ we have that $a_{i} \neq b_{i}$.
By Lemma~\ref{lem:needful}, there is an infinitesimal basic morphism $cd^{-1} \leq a_{i}b_{i}^{-1}$.
Thus $[cd^{-1}] \leq [\bigvee_{i=1}^{m} a_{i}b_{i}^{-1}]$.
Because $cd^{-1}$ is non-zero it follows that $[cd^{-1}]$ is non-zero.
We have therefore proved that each non-idempotent element of $\mathsf{B}(C)$ is above an infinitesimal.
It follows by  Proposition~\ref{prop:infinitesimal}, that $\mathsf{B}(C)$ is fundamental.
\end{proof}

\subsection{Cofinality}
Let $C$ be a conical cancellative category.
Define $C$ to be {\em cofinal} if for all $e,f \in C_{o}$ there exists a large subset $X$ of $fX$
such that $eC\mathbf{d}(x) \neq \varnothing$ for all $x \in X$.
this definition is from \cite{LS2010}.\\

We first of all convert this definition into one that is more useful for applying semigroup theory.

\begin{lemma}\label{lem:simple} Let $C$ be a finitely aligned $k$-graph.
The following are equivalent:
\begin{enumerate}

\item $C$ is cofinal.

\item For all idempotents in $\mathsf{R}(C)$ of the form $ee^{-1}$ and $ff^{-1}$, where $e,f \in C_{o}$,
there exists a set of basic morphisms $\{a_{1}x_{1}^{-1}, \ldots, a_{n}x_{n}^{-1}\}$ such that 
$\{x_{1}x_{1}^{-1}, \ldots, x_{n}x_{n}^{-1}\}$ is a tight cover of $ff^{-1}$ and
$\{a_{1}a_{1}^{-1}, \ldots, a_{n}a_{n}^{-1}\} \leq ee^{-1}$.
\end{enumerate}
\end{lemma}
\begin{proof} By Lemma~\ref{lem:correspondence}, $X$ is a large subset of $fX$
if and only if $\{xx^{-1} \colon x \in X\}$ is a tight cover of $ff^{-1}$.
Let $i$ and $j$ be identities of the category $C$.
To say that $a \in iCj$ is equivalent to saying that $aa^{-1} \leq ii^{-1}$
and $\mathbf{d}(x)\mathbf{d}(x)^{-1} = jj^{-1}$.
Thus cofinality transltaes into the following:
for all idempotents of the form $ee^{-1}$ and $ff^{-1}$, where $e,f \in C_{o}$,
there exists a tight cover $\{x_{1}x_{1}^{-1}, \ldots, x_{n}x_{n}^{-1}\}$ of $ff^{-1}$ 
and elements $a_{1},\ldots, a_{n}$ such that
$a_{i}a_{i}^{-1} \leq ee^{-1}$ where $\mathbf{d}(a_{i})\mathbf{d}(a_{i})^{-1} = \mathbf{d}(x_{i})\mathbf{d}(x_{i})^{-1}$
--- this latter condition is equivalent to $a_{i}x_{i}^{-1}$ being a well-defined basic morphism.
\end{proof}

\begin{lemma}\label{lem:orange} Let $C$ be a $k$-graph.
Then each non-trivial semigroup ideal of $\mathsf{B}(C)$ contains an idempotent of the form $[ee^{-1}]$ where $e$ is an identity of $C$.
\end{lemma}
\begin{proof}
Let $I$ be a non-trivial semigroup ideal in $\mathsf{B}(C)$.
Then $I$ is an order ideal.
Thus $[xx^{-1}] \in \mathsf{B}(C)$ for some $x \in C$.
Let $e = \mathbf{d}(x)$.
Then $ex^{-1}$ is a well-defined basic morphism and so belongs to $\mathsf{R}(C)$.
Observe that 
$\mathbf{d}(ex^{-1}) = xx^{-1}$
and 
$\mathbf{r}(ex^{-1}) = ee^{-1}$.
It follows that $[ex^{-1}] \in I$ and so $[ee^{-1}] \in I$.
\end{proof}

\begin{lemma}\label{lem:sasha} Let $S$ be a distributive inverse monoid in which $\equiv$ is idempotent pure.
Put $T = S/\equiv$ and denote the $\equiv$-congruence class containing $a$ by $[a]$.
Let $e$ and $f$ be any idempotents in $S$.
If $[f] \preceq [e]$ in $T$ then there is a set $X$ in $S$ such that
$\{\mathbf{d}(x) \colon x \in X\}$ is a tight cover of $f$
and $\mathbf{r}(x) \leq e$ for each $x \in X$.
\end{lemma}
\begin{proof} We are given that $[f] \preceq [e]$ in $T$.
Thus there is a pencil $\mathbf{X} = \{\mathbf{x}_{1},\ldots,\mathbf{x}_{m}\}$ in $T$ 
such that 
$[f] = \bigvee_{i = 1}^{m} \mathbf{d}(\mathbf{x}_{i})$
and
$\mathbf{r}[x_{i}] \leq [e]$ for $1 \leq i \leq m$.
Let $\mathbf{x}_{i} = [x_{i}]$.
Put $X = \{ex_{1}f,\ldots, ex_{m}f\}$.
Observe that $\mathbf{d}(ex_{i}f) \leq f$ and $\mathbf{r}(ex_{i}f) \leq e$.
We claim that $\mathbf{d}(X)$ is a tight cover of $f$.
Let $0 < h \leq f$.
Then $[0] < [h] \leq [f]$ since $\equiv$ is $0$-restricted.
Observe that $\bigvee_{i=1}^{m} [fx_{i}^{-1}ex_{i}] = [f]$.
Now use \cite[Lemma 2.5]{Lawson2016}, to deduce that
$[h] \wedge  [fx_{i}^{-1}ex_{i}] \neq [0]$ for some $i$.
But we are working with idempotents.
It follows that $[fx_{i}^{-1}ex_{i}h] \neq [0]$ and so,
since $\equiv$ is $0$-restricted,
we have that $fx_{i}^{-1}ex_{i}h \neq 0$ for some $i$.
\end{proof}

We can now prove our main result about cofinality.

\begin{theorem}\label{them:zero-simplifying}
Let $C$ be a strongly finitely aligned higher rank graph with a finite number of identities which has no sources and is row finite.
Then the following are equivalent:
\begin{enumerate}
\item $C$ is cofinal.
\item $\mathsf{B}(C)$ is $0$-simplifying.
\end{enumerate}
\end{theorem}
\begin{proof} (1)$\Rightarrow$(2).
Let $I$ be a non-trivial additive ideal of $\mathsf{B}(C)$.
By Lemma~\ref{lem:orange}, $I$ contains an idempotent of the form $[ee^{-1}]$
for some $e \in C_{o}$.
Let $f \in C_{o}$ be arbitrary.
Then by Lemma~\ref{lem:simple}, there exits a tight cover $\{x_{1}x_{1}^{-1}, \ldots, x_{n}x_{n}^{-1}\}$ of $ff^{-1}$ and
elements $a_{1}a_{1}^{-1}, \ldots, a_{n}a_{n}^{-1} \leq ee^{-1}$ such that
$a_{1}x_{1}^{-1}, \ldots, a_{n}x_{n}^{-1}$ are well-defined basic morphisms.
Observe that $\bigvee \mathbf{d}[a_{i}x_{i}^{-1}] = [ff^{-1}]$,
it is here that we use the fact that $\{x_{1}x_{1}^{-1}, \ldots, x_{n}x_{n}^{-1}\}$ is a tight cover of $ff^{-1}$
and Lemma~\ref{lem:new-characterization},
and that $\mathbf{r}[a_{i}x_{i}^{-1}] \leq [ee^{-1}]$.
Thus $\{[a_{1}x_{1}^{-1}], \ldots [a_{n}x_{n}^{-1}]\}$ is a pencil from $[ff^{-1}]$ to $[ee^{-1}]$.
But $I$ is an additive ideal and so $[ff^{-1}] \in I$ by Lemma~\ref{lem:cake}.
Thus $I$ contains every idempotent of the form $[ff^{-1}]$ where $f \in C_{o}$.
Now, let $[ab^{-1}]$ be an arbitrary basic morphism.
Then $\mathbf{d}(ab^{-1}) = bb^{-1} \leq \mathbf{r}(b)\mathbf{r}(b)^{-1}$ and
$\mathbf{r}(ab^{-1}) = aa^{-1} \leq \mathbf{r}(a)\mathbf{r}(a)^{-1}$.
But $I$ being an additive ideal is also an order ideal.
Thus $\mathbf{d}(ab^{-1}), \mathbf{r}(ab^{-1}) \in I$.
But in an inverse semigroup, an ideal that contains $x^{-1}x$ must contain $x$.
Thus $[ab^{-1}] \in I$.
But $I$ is additive and so $I$ contains all elements of $\mathsf{B}(C)$
(since every element is a join of basic morphisms).
We have therefore shown that $I = \mathsf{B}(C)$.

(2)$\Rightarrow$(1).
Let $e, f \in C_{o}$.
Observe that $I = (\mathsf{B}(C)[ee^{-1}]\mathsf{B}(C))^{\vee}$ is an additive ideal of $\mathsf{B}(C)$ containing $[ee^{-1}]$.
Then, by assumption, $I = \mathsf{B}(C)$.
Thus $[ff^{-1}] \in I$.
It follows by Lemma~\ref{lem:cake} that $[ff^{-1}] \preceq [ee^{-1}]$.
We now apply Lemma~\ref{lem:sasha},
to deduce that there is a finite set $\{a_{1}, \ldots, a_{m}\}$ or elements of $\mathsf{R}(C)$
such that  $\{\mathbf{d}(a_{1}), \ldots, \mathbf{d}(a_{m})\}$ is a tight cover of $ff^{-1}$ 
and  $\{\mathbf{r}(a_{1}), \ldots, \mathbf{r}(a_{m})\} \subseteq ee^{-1}$.
But each $a_{i}$ can be written as a join of basic morphisms.
We now apply Lemma~\ref{lem:blackadder} several times.
We may therefore assume that the $a_{i}$ are each basic morphisms.
By Lemma~\ref{lem:simple}, it follows that $C$ is cofinal.
\end{proof}

By theorem~\ref{them:fundamental} and Theorem~\ref{them:zero-simplifying}, we have proved the following.

\begin{theorem}\label{them:main} Let $C$ be a finitely aligned higher rank graph with a finite number of identities which has no sources and is row finite.
Then the following are equivalent:
\begin{enumerate}
\item $C$ is aperiodic and cofinal.
\item The Boolean inverse monoid $\mathsf{B}(C)$ is simple.
\end{enumerate}
\end{theorem}

It is clear that the Boolean inverse monoids $\mathsf{B}(C)$ are countable.
We now apply some results from \cite{Lawson2016}.
We call the countable atomless Boolean algebra the {\em Tarski algebra}
and use the term {\em Tarski monoid} to mean a countable, Boolean inverse $\wedge$-monoid whose semilattice of idempotents is the Tarski algebra.
The following is \cite[Proposition~4.4]{Lawson2016}.

\begin{proposition}\label{prop:old} let $S$ be a countable Boolean inverse $\wedge$-monoid.
If $S$ is $0$-simplifying then either $S$ is a Tarski monoid or the semilattice of idempotents of $S$ is finite.
\end{proposition}

By Lemma~\ref{lem:finite-fundamental},
a fundamental inverse semigroup in which the semilattice of idempotents is finite must itself be finite.
The finite simple Boolean inverse monoids are precisely the finite symmetric inverse monoids \cite[Theorem~4.18]{Lawson2012};
the groups of units of the finite symmetric inverse monoids are the finite symmetric inverse monoids.
With the help of \cite[Theorem~2.22]{Lawson2016}, we have therefore proved the following which is our main theorem.

\begin{theorem}\label{them:A} Let $C$ be a finitely aligned higher rank graph with a finite number of identities which has no sources and is row finite.
If $C$ is aperiodic and cofinal then there are two possibilities:
\begin{enumerate}
\item The Boolean inverse monoid $\mathsf{B}(C)$ is finite and isomorphic to a finite symmetric inverse monoid,
and its group of units is a finite symmetric group.
\item The Boolean inverse monoid $\mathsf{B}(C)$ is countably infinite,
and its group of units is isomorphic
to a full subgroup of the group of self-homeomorphisms of the Cantor space
which acts minimally and in which each element has clopen support.
\end{enumerate}
\end{theorem}

\begin{conjecture}{\em In the case where $\mathsf{B}(C)$ is countably infinite,
we are interested in the situation where its commutator subgroup is simple.
We conjecture that when the higher rank graph satisfies, in addition, the condition of \cite[Proposition~4.9]{KP}
then the commutator subgroup should be simple.}
\end{conjecture}

\section{The groupoid associated with $\mathsf{B}(C)$}

The goal of this section is to prove that the \'etale groupoid associated with the Boolean inverse monoid $\mathsf{B}(C)$
under non-commutative Stone duality is the usual groupoid $\mathcal{G}(C)$ associated with the higher rank graph $C$.
We refer the reader to \cite{Lawson2022} for a discussion of non-commutative Stone duality
which is needed during and after Lemma~\ref{lem:audrey}. 

Let $C$ be a $k$-graph.
Define $\Omega_{k}$ to be the category of all ordered pairs $(\mathbf{m}, \mathbf{n}) \in \mathbb{N}^k \times \mathbb{N}^k$, where $\mathbf{m} \leq \mathbf{n}$;
see \cite[Examples~1.7(ii)]{KP}.
A {\em $k$-tiling} in $C$ is a degree-preserving functor $w$ from $\Omega_{k}$ to $C$. These are called \emph{infinite paths} in $C$ elsewhere in the literature (for example in \cite{KP}).
Explicitly, $w$ satisfies the following three conditions:
\begin{enumerate}
\item $w(\mathbf{m},\mathbf{m})$ is an identity.
\item $w(\mathbf{m}, \mathbf{n})w(\mathbf{n},\mathbf{p}) = w(\mathbf{m},\mathbf{p})$
\item $d(w(\mathbf{m},\mathbf{n})) = \mathbf{n} - \mathbf{m}$.
\end{enumerate}
Denote the set of all $k$-tilings of $C$ by $C^{\infty}$.
If $w$ is a $k$-tiling, define $\mathbf{r}(w) = w(\mathbf{0},\mathbf{0})$.
Denote by $C^{\infty}(e)$ all $k$-tilings $w$ such that $e = \mathbf{r}(w)$.
For each $\mathbf{p} \in \mathbb{N}^{k}$ and $w \in C^{\infty}$ define $\sigma^{\mathbf{p}}(w) \in C^{\infty}$ by
$\sigma^{\mathbf{p}}(w)(\mathbf{m}, \mathbf{n}) = w(\mathbf{p} + \mathbf{m}, \mathbf{p} + \mathbf{n})$.
We shall now show that $k$-tilings of $C$ can be replaced by suitable subsets of $C$.

\begin{lemma}\label{lem:good-tilings} Let $C$ be a $k$-graph. 
For each $k$-tiling $w$, define
$$
\mathscr{C}_w = \{w(\mathbf{0}, \mathbf{m}) \colon \mathbf{m} \in \mathbb{N}^{k}\}.
$$
Then the map $w \mapsto \mathscr{C}_w$ is a bijection between $k$-tilings of $C$ and the set of good subsets of $C$.
\end{lemma}
\begin{proof}Let  $w \colon \Omega_{k} \rightarrow C$ be a $k$-tiling.
Observe that $\mathbf{r}(w(\mathbf{0}, \mathbf{m})) = \mathbf{r}(w(\mathbf{0}, \mathbf{0})) = e$, say.
Thus $\mathscr{C}_{w} \subseteq eC$.
The elements of  $\mathscr{C}_{w}$ are pairwise dependent and for each $\mathbf{n} \in \mathbb{N}^{k}$
there exists $x \in \mathscr{C}_{w}$ such that $d(x) = \mathbf{n}$.
So $\mathscr{C}_w$ is a good subset.
Conversely, let $A \subseteq C$ be a good subset. 
Then, \cite[Remarks~2.2]{KP} shows that there is a unique $k$-tiling $w_A$ such that $\mathscr{C}_{w_A} = A$.
\end{proof}

\noindent
{\bf Definition. }Let $C$ be a higher rank graph.
A subset $A \subseteq C$ is called {\em expanding} if each pair of elements of $A$ is dependent
and for each $\mathbf{m} \in \mathbb{N}^{k}$ there exists $a \in A$ such that $d(a) \geq \mathbf{m}$.\\

Every good subset is an expanding subset.
The following result shows how to get good subsets from expanding subsets.

\begin{lemma}\label{lem:boozy} Let $C$ be a higher rank graph.
Let $A$ be an expanding subset of $C$.
Then $\mathsf{Pref}(A)$ is a good subset.
\end{lemma}
\begin{proof} Let $x,y \in \mathsf{Pref}(A)$.
Then $a = xu$ and $b = yv$ for some $a,b \in A$ and $u,v \in C$.
The elements $a$ and $b$ are comparable.
Thus $xuc = yvd$ for some $c,d \in C$.
It follows that $x$ and $y$ are comparable.
Now let $\mathbf{m} \in \mathbb{N}^{k}$ be arbitrary.
Then there exists $a \in A$ such that $d(a) \geq \mathbf{m}$.
Let $d(a) = \mathbf{m} + \mathbf{n}$.
By the (UFP), there exists $x,y \in C$ such that $a = xy$, $d(x) = \mathbf{m}$ and $d(y) = \mathbf{n}$.
By definition, $x \in \mathsf{Pref}(A)$ and $d(x) = \mathbf{m}$.
We have therefore proved that $\mathsf{Pref}(A)$ is a good subset.
\end{proof}

\begin{lemma}\label{lem:hocus-pocus} Let $C$ be a $k$-graph.
Suppose that $A$ is an expanding subset such that $A \subseteq \mathbf{d}(x)C$.
Then $xA$ is an expanding subset.
\end{lemma}
\begin{proof} Let $xa, xb \in xA$.
Then $a,b \in A$ and so are comparable.
Thus $au = bv$ for some $u,v \in C$.
It follows that $(xa)u = (xb)v$ and so $xa$ and $xb$ are comparable.
Let $\mathbf{m} \in \mathbb{N}^{k}$.
Then there exists $a \in A$ such that $d(a) \geq \mathbf{m}$.
It follows that $d(xa) \geq \mathbf{m}$.
We have therefore proved that $xA$ is an expanding subset.
\end{proof}

Let $C$ be a higher rank graph,
let $x \in C$, 
and let $w$ a $k$-tiling such that $\mathbf{d}(x) = \mathbf{r}(w)$. 
We define the $k$-tiling $xw$.
By Lemma~\ref{lem:good-tilings},
we can construct from $w$ the good subset $\mathscr{C}_w$.
By Lemma~\ref{lem:hocus-pocus}, the set $x \mathscr{C}_w$ is an expanding subset.
Thus by Lemma~\ref{lem:boozy}, the set $\mathsf{Pref}(x \mathscr{C}_w)$ is a good subset.
It follows that $\mathsf{Pref}(x \mathscr{C}_w)$ corresponds, by Lemma~\ref{lem:good-tilings},
to a $k$-tiling, which we denote by $xw$. \\

\noindent
{\bf Definition. }Let $C$ be a higher rank graph.
Let $A$ be a good subset and let $x \in A$.
Define $x^{-1}A$ to be all elements $a$ such that $xa \in A$.\\

\begin{lemma}\label{lem:shazam}
Let $C$ be a higher rank graph,
let $A$ be a good subset in $C$, 
and let $x \in A$.
Then $x^{-1}A$ is an expanding set.
\end{lemma}
\begin{proof} The product $x\mathbf{d}(x)$ is defined.
Thus $\mathbf{d}(x) \in x^{-1}A$.
Let $u,v \in x^{-1}X$.
Then $xu,xv \in A$.
Thus $xub = xvc$ for some $b,c \in C$.
By cancellation, $ub = vc$.
Thus $u$ and $v$ are comparable.
Let $\mathbf{m} \in \mathbb{N}^{k}$.
Put $\mathbf{n} = \mathbf{m} + d(x)$.
Then there exists $a \in A$ such that $\mathbf{n} = d(a)$.
By the UFP, we can write $a = xy$ where $d(y) = \mathbf{m}$.
Thus $y \in x^{-1}A$.
\end{proof}

The proof of the following is immediate.

\begin{lemma}\label{lem:bien} Let $C$ be a $k$-graph and let $A$ be a good subset of $C$.
Suppose that the product $xy$ exists and that $xy \in A$.
Then $(xy)^{-1}A = y^{-1}(x^{-1}A)$.
\end{lemma}

Let $A$ be a good subset and let $\mathbf{m} \in \mathbb{N}^{k}$.
Then there is a unique  $x \in A$ such that $d(x) = \mathbf{m}$.
Define
$$\sigma^{\mathbf{m}}(A) = \mathsf{Pref}(x^{-1}A).$$

\begin{lemma}\label{lem:windfarm} Let $C$ be a higher rank graph.
Let $w$ be a $k$-tiling and $x = w(\mathbf{0},\mathbf{m})$.
Then $\mathsf{Pref}(x^{-1}\mathscr{C}_w) = \mathscr{C}_{\sigma^{\mathbf{m}}}(w)$.
\end{lemma}
\begin{proof} Let $y \in \mathsf{Pref}(x^{-1}\mathscr{C}_w)$.
Then $yv \in x^{-1}\mathscr{C}_{w}$ for some $v$.
It follows that $xyv \in \mathscr{C}_{w}$.
Thus $xyv = w(\mathbf{0},\mathbf{n})$.
We may write this as $xyv = w(\mathbf{0},\mathbf{m})w(\mathbf{m},\mathbf{n})$.
By uniqueness, $yv =  w(\mathbf{m},\mathbf{n})$ where $\mathbf{m} \leq \mathbf{n}$.
We may write this as
$yv = w(\mathbf{m},\mathbf{r})w(\mathbf{r},\mathbf{n})$ where $d(y) = \mathbf{r} - \mathbf{m}$.
It follows that $y =  w(\mathbf{m},\mathbf{r}) = (\sigma^{\mathbf{m}}w)(\mathbf{0},  \mathbf{r}-\mathbf{m})$.
Thus $y \in  \mathscr{C}_{\sigma^{\mathbf{m}}}(w)$.
We now prove the reverse inclusion.
Let $y \in \mathscr{C}_{\sigma^{\mathbf{m}}}(w)$.
Then $y = (\sigma^{\mathbf{m}}w)(\mathbf{0},\mathbf{n})$.
By definition $y = w(\mathbf{m},\mathbf{n} + \mathbf{m})$.
Thus $xy = w(\mathbf{0},\mathbf{n} + \mathbf{m}) \in \mathscr{C}_w$.
It follows that $y \in x^{-1} \mathscr{C}_w$.
\end{proof}

We can now define $\mathcal{G}(C)$, the usual groupoid associated with $C$, as follows \cite{KP, FMY}.
Its elements are triples $(w_{1}, \mathbf{n}, w_{2}) \in C^{\infty} \times \mathbb{Z}^{k} \times C^{\infty}$
where $\mathbf{l}, \mathbf{m} \in \mathbb{N}^{k}$ are such that
$\mathbf{n} = \mathbf{l} - \mathbf{m}$ and $\sigma^{\mathbf{l}}(w_{1}) = \sigma^{\mathbf{m}}(w_{2})$.
Define $\mathbf{d}(w_{1}, \mathbf{n}, w_{2}) = (w_{2}, \mathbf{0}, w_{2})$ and $\mathbf{r}(w_{1}, \mathbf{n}, w_{2}) = (w_{1}, \mathbf{0}, w_{1})$.
Multiplication is defined by $(w_{1}, \mathbf{m}, w_{2})(w_{2}, \mathbf{n}, w_{3}) = (w_{1}, \mathbf{m} + \mathbf{n}, w_{3})$
and the inverse by $(w_{2}, \mathbf{n},w_{1})^{-1} = (w_{1}, -\mathbf{n}, w_{2})$.
Let $x,y \in C$ such that $\mathbf{d}(x) = \mathbf{d}(y)$.
Then, a topology is defined on $\mathcal{G}(C)$ with base 
$$Z(x,y) = \{ (xw, d(x) - d(y), yw) \colon w \in C^{\infty}(\mathbf{d}(x))\}.$$

The remainder of this section is devoted to proving that the groupoid $\mathcal{G}(C)$
is the \'etale groupoid associated with the Boolean inverse monoid $\mathsf{B}(C)$ under non-commutative Stone duality when $C$ is a $k$-graph.
The first step in establishing this result is to describe  $\mathcal{G}(C)$ using good subsets.
We say that a triple $(A,\mathbf{n},B)$, 
where $A$ and $B$ are good subsets and $\mathbf{n} \in \mathbb{Z}^{k}$, is {\em allowable}
if there exists $x \in A$ and $y \in B$
such that $\mathbf{d}(x) = \mathbf{d}(y)$,
$\mathbf{n} = d(x) - d(y)$,
and $\mathsf{Pref}(x^{-1}A) = \mathsf{Pref}(y^{-1}B)$.
The set of such allowable triples becomes a groupoid when we define
$\mathbf{d}(A,\mathbf{n},B) = (B,\mathbf{0},B)$, 
$\mathbf{r}(A,\mathbf{n},B) = (A,\mathbf{0},A)$,
$(A,\mathbf{n},B)^{-1} = (B,-\mathbf{n},A)$,
and a partial product by 
$(A,\mathbf{m}, B)(B,\mathbf{n},C) = (A,\mathbf{m} + \mathbf{n},C)$.
The groupoid  $\mathcal{G}(C)$ is isomorphic to the groupoid of allowable triples
by the function 
$(w_{2},\mathbf{m},w_{1}) \mapsto (\mathscr{C}_{w_{2}}, \mathbf{m}, \mathscr{C}_{w_{1}})$
using Lemma~\ref{lem:good-tilings}.
The base for the topology on the set of allowable triples has the following form.
Let $x,y \in C$ such that $\mathbf{d}(x) = \mathbf{d}(y)$.
Then
$$Z'(x,y) = \{ (\mathsf{Pref}(xA), d(x) - d(y), \mathsf{Pref}(yA)) \colon A \subseteq \mathbf{d}(x)C\text{ is expanding}\}$$
is a set of allowable triples.

\begin{lemma}\label{lem:new-groupoid} Let $C$ be a $k$-graph.
The set of allowable triples forms a groupoid isomorphic to the groupoid $\mathcal{G}(C)$.
\end{lemma}

We shall now show that the groupoid $\mathcal{G}(C)$ is isomorphic to the groupoid constructed from the Boolean
inverse monoid $\mathsf{B}(C)$ as described in \cite{Lawson2022}.
We therefore need to relate ultrafilters in $\mathsf{B}(C)$ with allowable triples.

Lemma~\ref{lem:good-tilings}, Lemma~\ref{lem:good-filter} Proposition~\ref{prop:mo}
and the fact that in $C$ all tight filters are ultrafilters, we have proved the following.

\begin{proposition}\label{prop:eight} Let $C$ be a higher rank graph with a finite number of identities
which is row finite and has no sources. For each $k$-tiling $w$ of $C$, the set
$\mathsf{P}_w = \{w(0,\mathbf{m})w(0, \mathbf{m})^{-1} \colon \mathbf{m} \in \mathbb{N}^k\}^{\uparrow}$ is an ultrafilter
in $\mathsf{E}(\mathsf{R}(C))$, and the map $w \mapsto \mathsf{P}_w$ is a bijection between
$k$-tilings in $C$ and ultrafilters in $\mathsf{E}(\mathsf{R}(C))$.
\end{proposition}

Let $A$ be a good subset of $C$.
By Lemma~\ref{lem:good-filter-more}, this is (precisely) a maximal filter in $C$.
Define $\mathsf{F}(A) = \{[xx^{-1}] \colon x \in A \}^{\uparrow}$ in $\mathsf{B}(C)$.
By Theorem~\ref{them:seven} and Proposition~\ref{prop:mo}, this is an idempotent ultrafilter in $\mathsf{B}(C)$.

\begin{lemma}\label{lem:audrey} 
Let $C$ be a higher rank graph with a finite number of identities
which is row finite and has no sources.
Then there is a bijection between the set of allowable triples in $C$ and the set of ultrafilters in $\mathsf{B}(C)$.
\end{lemma}
\begin{proof} Let $\tau = (A,\mathbf{n},B)$ be an allowable triple.
Then $A$ and $B$ are good sets,
there is an $x \in A$, such that $d(x) = \mathbf{l}$, and there is $y \in B$, such that $d(y) = \mathbf{m}$,
where $\mathbf{d}(x) = \mathbf{d}(y)$, $\mathbf{n} = \mathbf{l} - \mathbf{m}$ and $\mathsf{Pref}(x^{-1}A) = \mathsf{Pref}(y^{-1}B)$.
Observe that $xy^{-1}$ is a well-defined morphism in $\mathsf{R}(C)$.
Define $\mathcal{F} = \mathsf{F}(B)$.
Then $\mathcal{F}$ is an identity filter in $\mathsf{B}(C)$.
Put
$$\mathcal{A} = \mathcal{A}_{\tau} = ([xy^{-1}]\mathcal{F})^{\uparrow};$$
this is the ultrafilter in $\mathsf{B}(C)$ associated with the allowable triple $\tau$.

We now show that this ultrafilter is independent of the choices we made above.
Let $x_{1} \in A$ be such that $d(x_{1}) = \mathbf{l}_{1}$ and let $y_{1} \in B$ be such that $d(y_{1}) = \mathbf{m}_{1}$
where $\mathbf{d}(x_{1}) = \mathbf{d}(y_{1})$,
$\mathbf{n} = \mathbf{l}_{1} - \mathbf{m}_{1}$
and $x_{1}^{-1}A = y_{1}^{-1}B$.
We show that
$\mathcal{B} = ([x_{1}y_{1}^{-1}]\mathcal{F})^{\uparrow}$
and
$\mathcal{A} = ([xy^{-1}]\mathcal{F})^{\uparrow}$
are equal.
Consider the product $(yx^{-1})(x_{1}y_{1}^{-1})$.
Let $xC \cap x_{1}C = \{u_{1}, \ldots, u_{m}\}C$.
Since $x,x_{1} \in A$ there exists an $i$ such that $u_{i} = xp_{i} = x_{1}q_{i} \in A$ by Lemma~\ref{lem:filter-property}.
Observe that $(yp_{i})(y_{1}q_{i})^{-1} \leq (yx^{-1})(x_{1}y_{1}^{-1})$.
We have that $yp_{i}, y_{1}q_{i} \in B$
--- this follows because $p_{i} \in x^{-1}A = y^{-1}B$ and so $yp_{i} \in B$,
and $q_{i} \in x_{1}^{-1}A = y_{1}^{-1}B$ and so $y_{1}q_{i} \in B$.
We shall prove that $d(yp_{i}) = d(y_{1}q_{i})$ from which it will follow that $yp_{i} = y_{1}q_{i}$,
since $B$ is a filter and so a good subset.
But this follows from the fact that $\mathbf{l}_{1} - \mathbf{m}_{1} = \mathbf{l} - \mathbf{m}$.
We have therefore found an idempotent in $\mathcal{F}$ below  $[yx^{-1}][x_{1}y_{1}^{-1}]$.
This proves that $\mathcal{A} = \mathcal{B}$.

We now go in the opposite direction.
Let $\mathcal{A}$ be an ultrafilter in $\mathsf{B}(C)$.
Then, using the fact that ultrafilters are prime,
we may write this in the form $([xy^{-1}]\mathcal{F})^{\uparrow}$ where $[yy^{-1}] \in \mathcal{F}$
and $\mathcal{F}$ is an idempotent ultrafilter in $\mathsf{B}(C)$.
The ultrafilter $\mathcal{F}$ is completely determined by the ultrafilter $\mathcal{F} \cap \mathsf{E}(\mathsf{B}(C))$ which is in
$\mathsf{E}(\mathsf{B}(C))$.
The ultrafilter $\mathcal{F} \cap \mathsf{E}(\mathsf{B}(C))$ arises from the maximal filter $B$ in $C$ via Proposition~\ref{prop:mo}.
Observe that $y \in B$.
Now $([xy^{-}] \mathcal{F}[yx^{-1}])^{\uparrow}$ is an idempotent ultrafilter in  $\mathsf{B}(C)$ that contains $[xx^{-1}]$.
This corresponds to a maximal filter $A$ in $C$ that contains $x$.
We have therefore constructed a triple $(A,d(x) - d(y),B)$.
It remains to show that it is allowable.
Thus we need to prove that $x^{-1}A = y^{-1}B$.
let $u \in y^{-1}A$.
Then $yu \in A$.
It follows that $[yu(yu)^{-1}] \in \mathcal{F} = \mathbf{d}(\mathcal{A})$.
But $[xy^{-1}] \in \mathcal{A}$.
It follows that $[xy^{-1}][yu(yu)^{-1}] \in \mathcal{A}$.
Thus $[xu(yu)^{-1}] \in \mathcal{A}$.
It folows that $xu \in B$, as required.
The converse is proved by symmetry.
\end{proof}

We now prove that the topological groupoid of allowable triples and the topological groupoid of ultrafilters in $\mathsf{B}(C)$ are isomorphic.
A composable pair of allowable triples has the following form:
$$(C, d(u)-d(v), B)(B, d(v) - d(w), A)$$
where $u \in C$, $v \in B$ and $w \in A$.
Their product is
$$(C, d(u) - d(w), A).$$
We now turn to products of ultrafilters.
Let $A$ and $B$ be ultrafilters such that $\mathbf{d}(A) = \mathbf{r}(B)$.
Then $A \cdot B = (AB)^{\uparrow}$.
Let $A = (aF)^{\uparrow}$ and $B = (bG)^{\uparrow}$ where $\mathbf{d}(A) = F$ and $\mathbf{d}(B) = G$.
Observe that $A \cdot B = (abG)^{\uparrow}$.
Now, $a\mathbf{r}(b) \in A\mathbf{d}(A) \subseteq A$.
Similarly, $\mathbf{d}(a)b \in B$.
It follows that we can assume $\mathbf{d}(a) = \mathbf{r}(b)$.
The ultrafilter associated with $(C, d(u)-d(v), B)$ is $([uv^{-1}]\mathcal{F})^{\uparrow}$
where $x \in B$ iff $[xx^{-1}] \in \mathcal{F}$.
The ultrafilter associated with $(B, d(v) - d(w), A)$ is
$([vw^{-1}]\mathcal{G})^{\uparrow}$.
The product of
$([uv^{-1}]\mathcal{F})^{\uparrow}$
and
$([vw^{-1}]\mathcal{G})^{\uparrow}$
is defined
and equals
$([uw^{-1}]\mathcal{G})^{\uparrow}$
which corresponds to the allowable triple
$(C, d(u) - d(w),A)$.
We now describe the topology.
Let $\mathbf{d}(x) = \mathbf{d}(y)$.
Then
$$Z'(x,y) = \{ (\mathsf{Pref}(xA), d(x) - d(y), \mathsf{Pref}(yA)) \colon A \subseteq \mathbf{d}(x)C\}$$
is a set of allowable triples.
This corresponds to the set of ultrafilters in $\mathsf{B}(C)$ that contain the element $[xy^{-1}]$.
It follows that the groupoids are isomorphic as topological groupoids.

\section{Invariants of the group $\mathscr{G}(C)$}

In this section, we shall assume that the Boolean inverse $\wedge$-monoid $\mathsf{B}(C)$ is simple and countably infinite.
The group $\mathscr{G}(C)$ coincides with the topological full group
of the groupoid $\mathcal{G}(C)$, which is a Hausdorff, \'etale, effective, minimal groupoid with
unit space homeomorphic to the Cantor space.
By \cite[Theorem~3.10]{Matui2015}, this implies that if
$\mathscr{G}(C) \cong \mathscr{G}(C')$, then $\mathcal{G}(C) \cong \mathcal{G}(C')$. Consequently
both the $K$-theory of the groupoid $C^*$-algebra $C^*(\mathcal{G}(C))$ and the homology, in the
sense of Matui, of the groupoid $\mathcal{G}(C)$ are isomorphism invariants of $\mathscr{G}(C)$.
When $C$ is a $k$-graph, this is particularly interesting because, as we saw in the preceding
section, the groupoid $\mathcal{G}(C)$ coincides with the one studied by Kumjian and Pask \cite{KP}. 
So known invariants of their groupoid are also invariants of our group $\mathscr{G}(C)$.
Kumjian and Pask prove in \cite[Corollary~3.5(i)]{KP}, that the $k$-graph $C^*$-algebra $C^*(C)$
coincides with the groupoid $C^*$-algebra $C^*(\mathcal{G}(C))$. 
Hence the $K$-theory of $C^*(C)$ provides an isomorphism invariant of $\mathscr{G}(C)$.
We therefore have the following.

\begin{corollary}
Let $C$ and $C'$ be row-finite aperiodic, cofinal higher rank graphs with finitely many vertices
and no sources. If $\mathscr{G}(C) \cong \mathscr{G}(C')$ as discrete groups, then $K_*(C^*(C))
\cong K_*(C^*(C'))$.
\end{corollary}

The difficulty here is that the $K$-theory of $k$-graph $C^*$-algebras has proven notoriously
difficult to compute. The most general results are those of \cite{Evans}, but these apply in
general only when $k \le 2$.
However another, closely related, invariant of the groupoid $\mathcal{G}(C)$ and hence of our
group $\mathscr{G}(C)$ is the homology of $\mathcal{G}(C)$.

\begin{corollary}
Let $C$ and $C'$ be row-finite aperiodic, cofinal higher rank graphs with finitely many vertices
and no sources. If $\mathscr{G}(C) \cong \mathscr{G}(C')$ as discrete groups, then
$H_*(\mathcal{G}(C)) \cong H_*(\mathcal{G}(C))$.
\end{corollary}

In this case, we have an explicit calculation of the invariant obtained from \cite[Proposition~7.6]{FKPS}
building on earlier work of Matui \cite{Matui2012}. Specifically, the homology of the groupoid
$\mathcal{G}(C)$ is precisely the homology of a chain complex developed by Evans \cite{Evans}.
To describe it, we proceed as follows.

Let $\varepsilon_1, \dots, \varepsilon_k$ denote the generators of $\mathbb{Z}^k$ (so
$\varepsilon_i = (0, \dots, 0, 1, 0, \dots, 0)$, with the $1$ appearing in the $i$th coordinate).
Also for each $i$, let $M_i$ be the $i$th coordinate matrix of the $k$-graph $C$. That is,
recalling that $C_o$ is the (finite) space of identity morphisms (or vertices) of $C$,
the matrix $M_i$ is the $C_o \times C_o$ matrix with entries
\[
M_i(e, f) = |\{a \in C : d(a) = \varepsilon_i, \mathbf{r(a)} = e\text{ and }\mathbf{d}(a) = f\}|.
\]
We regard each $M_i$ as an endomorphism of the free abelian group $\mathbb{Z} C_o$.
Recall that for $p \ge 1$, we write $\bigwedge^p \mathbb{Z}^k$ for the $p$th exterior power of
$\mathbb{Z}^k$, which is generated by the elements $\varepsilon_{i_1} \wedge \cdots \wedge
\varepsilon_{i_p}$ where $1 \le i_j \le k$ for all $j$.
Define $D^C_0 = \mathbb{Z} C_o$,
and for $p \ge 1$, define $D^C_p =
\big(\bigwedge^p \mathbb{Z}^k\big) \otimes \mathbb{Z}C_{\mathbf{0}}$; observe that this forces
$D^C_p = \{0\}$ for $p \ge k$. For $p  \ge 2$, define $\partial_p  \colon D^C_p \to D^C_{p-1}$ (using
the hat symbol to indicate deletion of a term) by
\[
\partial_p((\varepsilon_{i_1} \wedge \cdots \wedge \varepsilon_{i_p}) \otimes \varepsilon_e)
    = \sum^p_{j=0} (-1)^{j+1} (\varepsilon_{i_1} \wedge \cdots \wedge \widehat{\varepsilon}_{i_j} \wedge \cdots \wedge \varepsilon_{i_p}) \otimes (I - M_{i_j}^t)\varepsilon_e;
\]
and, finally, define $\partial_1 \colon  D^C_1 \to D^C_0$ by
\[
\partial_1(\varepsilon_i \otimes \varepsilon_e) = (1 - M^t_i)\varepsilon_e.
\]
Then $(D^C_*, \partial_*)$ is a chain complex, and \cite[Proposition~7.6]{FKPS} shows that
$H_*(D^C_*, \partial_*)$ is isomorphic to the groupoid homology $H_*(\mathcal{G}_C)$. So we obtain the following corollary.

\begin{corollary}\label{cor:H* invariants}
Let $C$ and $C'$ be row-finite aperiodic, cofinal higher rank graphs with finitely many vertices
and no sources. If $\mathscr{G}(C) \cong \mathscr{G}(C')$ as discrete groups, then $H_*(D^C_*,
\partial_*) \cong H_*(D^{C'}_*, \partial_*)$. In particular,
\[
 H_k(D^C_*, \partial_*)
 =
\bigcap^k_{i=1} \ker(I - M(C)^t_i) \cong \bigcap^k_{i=1} \ker(I - M(C')^t_i),
\]
and
\[
H_0(D^C_*, \partial_*)
=
\mathbb{Z}C_o / \Big(\sum^k_{i=1} \operatorname{im}(I - M(C)^t_i)\Big)
     \cong \mathbb{Z}C'_o / \Big(\sum^k_{i=1} \operatorname{im}(I - M(C')^t_i)\Big).
\]
\end{corollary}
\begin{proof}
The first statement follows immediately from \cite[Theorem~3.10]{Matui2015} and
\cite[Proposition~7.6]{FKPS}.
The second is by direct computation: the intersection of the kernels
of the $I - M(C)^t_i$ is isomorphic to $H_k(D^C_*, \partial_*)$ and the quotient of
$\mathbb{Z}C_o$ by the sum of their images is isomorphic to $H_0(D^C_*, \partial_*)$
(and similarly for $C'$).
\end{proof}

\section{Two series of concrete examples}

In this section we provide two constructions of infinite families of $k$-graphs $C$ with mutually
non-isomorphic groups $\mathscr{G}(C)$. In the first family, the higher-rank graphs $C$ can all
be chosen to be of the same rank $k \ge 2$, and the associated groups $\mathscr{G}(C)$
are distinguished by the finite $0$th homology groups of the $k$-graphs. In the second family of
examples, we show that for each $k \ge 1$ and each $R \ge 1$, there is an aperiodic cofinal
$k$-graph $C_{k, R}$ that is row-finite and has finitely many vertices and whose $k$th homology
group is of rank $R$.

\subsection{Examples distinguished by their $0$th homology}

We shall use the results of the previous section to construct, for each $k \ge 2$ and each
$k$-tuple of integers $(m_1,m_2,\ldots,m_k)$, two series of pairwise non-isomorphic $k$-graphs
with two vertices and different groups
\[
\mathbb{Z}C_o/\Big(\sum^k_{i=1} \operatorname{im}(I - M(C)^t_i)\Big).
\]
Our construction consists of two steps:
first, we construct a family of cube complexes with two vertices, covered by products of $k$ trees,
and second, we explain how to get a $k$-graph from each complex.

For background on cube complexes covered by products of $k$ trees see \cite{RSV} and references in the paper.

\noindent{\bf Step 1.} Let $X_1,...,X_k$ be distinct alphabets, such that
$\left| X_i \right|=m_i$ and
$$X_i=\{x_1^i,x_2^i,...,x_{m_i}^i\}.$$
Let $F_i$ be the free group generated by $X_i$.
The direct product
\[
G = F_1\times F_2 \times \ldots \times F_k
\]
has presentation
\[
    G=\langle X_1,X_2,...,X_k | [x^i_s,x^j_l]=1, i \neq j=1,...,k; s=1,...,m_i;l=1,...,m_j \rangle,
\]
where $[x,y]$ denotes the commutator $xyx^{-1}y^{-1}$.

The group $G$ acts simply and transitively on a Cartesian product $\Delta$ of $k$ trees $T_1,T_2,..,T_k$ of valencies $2m_1,2m_2,...,2m_k$ respectively:
each $T_i$ is identified with the Cayley tree of $F_i$, and the action of $G$ is the coordinatewise action of the component groups $F_i$.

The quotient of this action is a cube complex $P$ with one vertex $v$ such that the universal cover of $P$ is $\Delta$. The edges of the cube complex $P$
are naturally labelled by elements of $X=X_1\cup X_2...\cup X_k$, and are naturally oriented by the usual algebraic ordering on $F_i$. The $1$-skeleton of $P$ is a wedge of $\sum_{i=1}^k m_i$ circles.

We construct a family of double covers of $P$ in the following way. Consider a labelling $\ell : X \to \mathbb{Z}_2$ of the elements of $X$ (equivalently the edges of the $1$-skeleton of $P$).
We obtain a cover $P^2_\ell$ of $P$ whose vertex set is $\{v\} \times \mathbb{Z}_2$ and whose set of $1$-cubes is $X \times \mathbb{Z}_2$, with range and domain maps given by
$r(x, i) = (r(x), i)$ and $s(x,i) = (s(x), i + \ell(x))$. Specifically, $P^2_\ell$ is the quotient of $\Delta$ by the action of the kernel of the homomorphism $G \to \mathbb{Z}_2$
induced by $\ell$. Observe that in $P^2_\ell$, for a given $x \in X$ either $(x, 0)$ and $(x,1)$ are loops based at $(v,0)$ and $(v,1)$ (if $\ell(x) = 0$), or $(x,0)$ is an edge
from $(v,1)$ to $(v,0)$ and $(x,1)$ is an edge from $(v,0)$ to $(v,1)$.

So Figure~\ref{fig:1} illustrates the $2$-cubes and part of the $1$-skeleton of $P$ corresponding to symbols $a \in X_i$ and $b \in X_j$ with $i \not= j$
in which $\ell(a) = 1$ and $\ell(b) = 0$; Figure~\ref{fig:2} illustrates the corresponding $2$-cubes and part of the $1$-skeleton if $\ell(a) = \ell(b) = 1$.

\begin{figure}[!htbp]
\[
\begin{tikzpicture}
    \node (00) at (0,0) {$(v,0)$};
    \node (01) at (0,2) {$(v,1)$};
    \node (10) at (2,0) {$(v,0)$};
    \node (11) at (2,2) {$(v,1)$};
    \draw[->-] (00) to node[pos=0.5, right] {\small$(a,1)$} (01);
    \draw[->-] (00) to node[pos=0.5, below] {\small$(b,0)$} (10);
    \draw[->-] (10) to node[pos=0.5, left] {\small$(a,1)$} (11);
    \draw[->-] (01) to node[pos=0.5, above] {\small$(b,1)$} (11);
    \begin{scope}[xshift=3.8cm]
    \node (00) at (0,0) {$(v,1)$};
    \node (01) at (0,2) {$(v,0)$};
    \node (10) at (2,0) {$(v,1)$};
    \node (11) at (2,2) {$(v,0)$};
    \draw[->-] (00) to node[pos=0.5, right] {\small$(a,0)$} (01);
    \draw[->-] (00) to node[pos=0.5, below] {\small$(b,0)$} (10);
    \draw[->-] (10) to node[pos=0.5, left] {\small$(a,0)$} (11);
    \draw[->-] (01) to node[pos=0.5, above] {\small$(b,1)$} (11);
    \end{scope}
    \begin{scope}[xshift=8cm]
    \node (v0) at (0,1) {$(v,0)$};
    \node (v1) at (2,1) {$(v,1)$};
    \draw[->-, out=30, in=150] (v0) to node[pos=0.5, above] {\small$(a,1)$} (v1);
    \draw[->-, out=210, in=330] (v1) to node[pos=0.5, below] {\small$(a,0)$} (v0);
    \draw[->-] (v0) .. controls +(-1, 1) and +(-1, -1) .. (v0) node[pos=0.25, above] {\small$(b,0)$};
    \draw[->-] (v1) .. controls +(1, 1) and +(1, -1) .. (v1) node[pos=0.75, below] {\small$(b,1)$};
    \end{scope}
\end{tikzpicture}
\]
\caption{$\ell(a) = 1$ and $\ell(b) = 0$}\label{fig:1}
\end{figure}
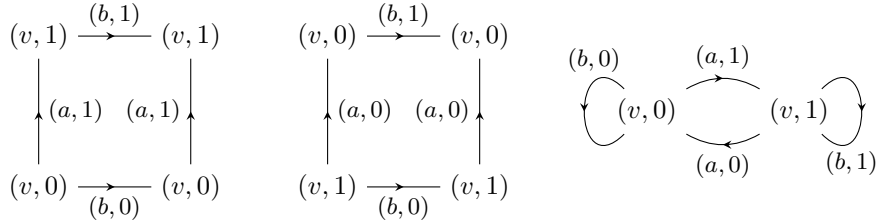

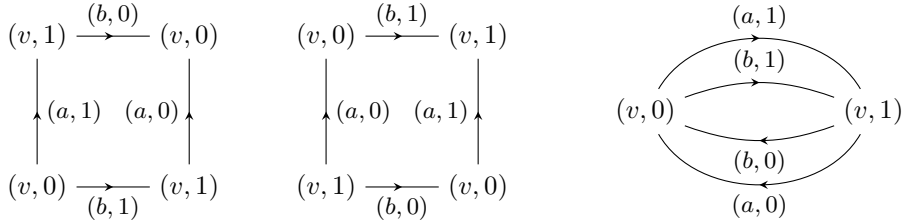
\begin{figure}[!htbp]
\[
\begin{tikzpicture}
    \node (00) at (0,0) {$(v,0)$};
    \node (01) at (0,2) {$(v,1)$};
    \node (10) at (2,0) {$(v,1)$};
    \node (11) at (2,2) {$(v,0)$};
    \draw[->-] (00) to node[pos=0.5, right] {\small$(a,1)$} (01);
    \draw[->-] (00) to node[pos=0.5, below] {\small$(b,1)$} (10);
    \draw[->-] (10) to node[pos=0.5, left] {\small$(a,0)$} (11);
    \draw[->-] (01) to node[pos=0.5, above] {\small$(b,0)$} (11);
    \begin{scope}[xshift=3.8cm]
    \node (00) at (0,0) {$(v,1)$};
    \node (01) at (0,2) {$(v,0)$};
    \node (10) at (2,0) {$(v,0)$};
    \node (11) at (2,2) {$(v,1)$};
    \draw[->-] (00) to node[pos=0.5, right] {\small$(a,0)$} (01);
    \draw[->-] (00) to node[pos=0.5, below] {\small$(b,0)$} (10);
    \draw[->-] (10) to node[pos=0.5, left] {\small$(a,1)$} (11);
    \draw[->-] (01) to node[pos=0.5, above] {\small$(b,1)$} (11);
    \end{scope}
    \begin{scope}[xshift=8cm]
    \node (v0) at (0,1) {$(v,0)$};
    \node (v1) at (3,1) {$(v,1)$};
    \draw[->-, out=20, in=160] (v0) to node[pos=0.5, above] {\small$(b,1)$} (v1);
    \draw[->-, out=200, in=340] (v1) to node[pos=0.5, below] {\small$(b,0)$} (v0);
    \draw[->-, out=60, in=120] (v0) to node[pos=0.5, above] {\small$(a,1)$} (v1);
    \draw[->-, out=240, in=300] (v1) to node[pos=0.5, below] {\small$(a,0)$} (v0);
    \end{scope}
\end{tikzpicture}
\]
\caption{$\ell(a) = \ell(b) = 1$}\label{fig:2}
\end{figure}

We will consider two specific labellings $\ell_u$ and $\ell_m$ (the $u$ and $m$ stand for
``uniform" and ``mixed"). The uniform labelling $\ell_u$ is given by $\ell_u(x) = 1$ for all $x$,
whereas the mixed labelling $\ell_m$ satisfies $\ell_m|_{X_1} \equiv 0$ and $\ell_m|_{X \setminus
X_1} \equiv 1$. So under $\ell_u$ all $2$-squares are as in Figure~\ref{fig:2}; but under $\ell_m$
squares in which $b$ belongs to $X_1$ are as in Figure~\ref{fig:1}, and the remaining squares are
as in Figure~\ref{fig:2}. Figure~\ref{fig:4} illustrates a $3$-cube in the cube complex for
$\ell_u$ (left) and a $3$-cube in the cube complex for $\ell_m$ in which $a$ belongs to $X_1$ and
the edges $b, c$ belong to $X \setminus X_1$ (right).

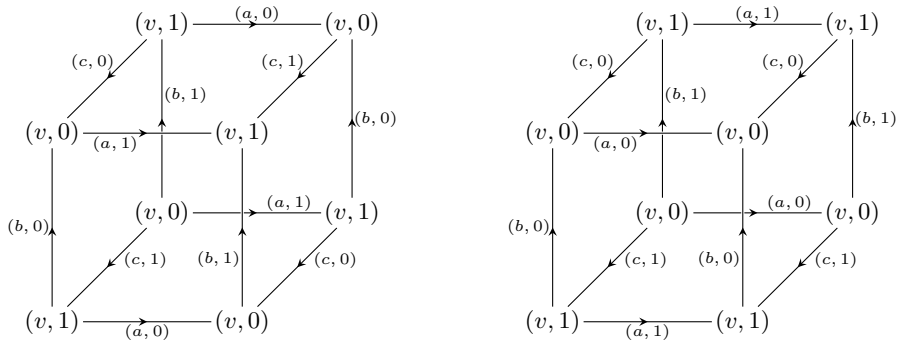
\begin{figure}[!htbp]
\[
\begin{tikzpicture}[scale=1.25]
    \node[inner sep=1pt] (000) at (0,0,0) {\small$(v,0)$};
    \node[inner sep=1pt] (100) at (2,0,0) {\small$(v,1)$};
    \node[inner sep=1pt] (010) at (0,2,0) {\small$(v,1)$};
    \node[inner sep=1pt] (110) at (2,2,0) {\small$(v,0)$};
    \node[inner sep=1pt] (001) at (0,0,3) {\small$(v,1)$};
    \node[inner sep=1pt] (101) at (2,0,3) {\small$(v,0)$};
    \node[inner sep=1pt] (011) at (0,2,3) {\small$(v,0)$};
    \node[inner sep=1pt] (111) at (2,2,3) {\small$(v,1)$};
    \draw[->-] (000) to node[pos=0.75, above, inner sep=0.5pt] {\tiny$(a,1)$} (100);
    \draw[->-] (000) to node[pos=0.65, right, inner sep=0.5pt] {\tiny$(b,1)$} (010);
    \draw[->-] (000) to node[pos=0.3, anchor=north west, inner sep=0pt] {\tiny$(c,1)$} (001);
    \draw[->-] (100) to node[pos=0.5, right, inner sep=0.5pt] {\tiny$(b,0)$} (110);
    \draw[->-] (100) to node[pos=0.3, anchor=north west, inner sep=0pt] {\tiny$(c,0)$} (101);
    \draw[->-] (001) to node[pos=0.5, below, inner sep=0.5pt] {\tiny$(a,0)$} (101);
    \draw[->-] (001) to node[pos=0.5, left, inner sep=0.5pt] {\tiny$(b,0)$} (011);
    \draw[->-] (010) to node[pos=0.5, above, inner sep=0.5pt] {\tiny$(a,0)$} (110);
    \draw[->-] (010) to node[pos=0.4, anchor=south east, inner sep=0pt] {\tiny$(c,0)$} (011);
    \draw[->-] (110) to node[pos=0.4, anchor=south east, inner sep=0pt] {\tiny$(c,1)$} (111);
    \draw[very thick, white] (101) to (111);
    \draw[->-] (101) to node[pos=0.3, left, inner sep=0.5pt] {\tiny$(b,1)$} (111);
    \draw[very thick, white] (011) to (111);
    \draw[->-] (011) to node[pos=0.25, below, inner sep=1pt] {\tiny$(a,1)$} (111);
\end{tikzpicture}
\qquad\qquad
\begin{tikzpicture}[scale=1.25]
    \node[inner sep=1pt] (000) at (0,0,0) {\small$(v,0)$};
    \node[inner sep=1pt] (100) at (2,0,0) {\small$(v,0)$};
    \node[inner sep=1pt] (010) at (0,2,0) {\small$(v,1)$};
    \node[inner sep=1pt] (110) at (2,2,0) {\small$(v,1)$};
    \node[inner sep=1pt] (001) at (0,0,3) {\small$(v,1)$};
    \node[inner sep=1pt] (101) at (2,0,3) {\small$(v,1)$};
    \node[inner sep=1pt] (011) at (0,2,3) {\small$(v,0)$};
    \node[inner sep=1pt] (111) at (2,2,3) {\small$(v,0)$};
    \draw[->-] (000) to node[pos=0.75, above, inner sep=0.5pt] {\tiny$(a,0)$} (100);
    \draw[->-] (000) to node[pos=0.65, right, inner sep=0.5pt] {\tiny$(b,1)$} (010);
    \draw[->-] (000) to node[pos=0.3, anchor=north west, inner sep=0pt] {\tiny$(c,1)$} (001);
    \draw[->-] (100) to node[pos=0.5, right, inner sep=0.5pt] {\tiny$(b,1)$} (110);
    \draw[->-] (100) to node[pos=0.3, anchor=north west, inner sep=0pt] {\tiny$(c,1)$} (101);
    \draw[->-] (001) to node[pos=0.5, below, inner sep=0.5pt] {\tiny$(a,1)$} (101);
    \draw[->-] (001) to node[pos=0.5, left, inner sep=0.5pt] {\tiny$(b,0)$} (011);
    \draw[->-] (010) to node[pos=0.5, above, inner sep=0.5pt] {\tiny$(a,1)$} (110);
    \draw[->-] (010) to node[pos=0.4, anchor=south east, inner sep=0pt] {\tiny$(c,0)$} (011);
    \draw[->-] (110) to node[pos=0.4, anchor=south east, inner sep=0pt] {\tiny$(c,0)$} (111);
    \draw[very thick, white] (101) to (111);
    \draw[->-] (101) to node[pos=0.3, left, inner sep=0.5pt] {\tiny$(b,0)$} (111);
    \draw[very thick, white] (011) to (111);
    \draw[->-] (011) to node[pos=0.25, below, inner sep=1pt] {\tiny$(a,0)$} (111);
\end{tikzpicture}
\]
\caption{A $3$-cube for $\ell_u$ (left) and for $\ell_m$ (right)}\label{fig:4}
\end{figure}

\noindent{\bf Step 2.} For each of $\star = u$ and $\star = m$, we explain how to construct a $k$-graph $C_\star$ from $P^2_{\ell_\star}$,
by specifying their skeletons and factorization rules as in \cite{HRSW}. For either value of $\star$, we define
$(C_\star)_o = \{(v,0), (v,1)\}$, and for each $i \le k$, we define $(C_\star)_{e_i}$ to be the set
\[
    X_i \times \mathbb{Z}_2 \sqcup \{(\overline{x}, j) : x \in X_i, j \in \mathbb{Z}_2\},
\]
a disjoint union of two copies $X \times \mathbb{Z}_2$. The range and source maps on
$(C_\star)_{e_i}$ restrict to those in $P^2_{\ell_\star}$ on $X_i \times \mathbb{Z}_2$, and we
define $r(\overline{x}, j) = s(x, j)$ and $s(\overline{x},j) = r(x,j)$. The factorization rules
are as follows: for each $2$-cube
\[
\begin{tikzpicture}
    \node (00) at (0,0) {$(v,0)$};
    \node (01) at (0,2) {$(v,1)$};
    \node (10) at (2,0) {$(v,0)$};
    \node (11) at (2,2) {$(v,1)$};
    \draw[->-] (00) to node[pos=0.5, left] {\small$(a,i_a)$} (01);
    \draw[->-] (00) to node[pos=0.5, below] {\small$(b,i_b)$} (10);
    \draw[->-] (10) to node[pos=0.5, right] {\small$(a,j_a)$} (11);
    \draw[->-] (01) to node[pos=0.5, above] {\small$(b,j_b)$} (11);
\end{tikzpicture}
\]
in $P^2_{\ell_\star}$, we have four factorization rules:
\begin{align*}
(b,i_b)(a,i_a) &= (a,j_a)(b,j_b),&\qquad
(a,i_a)(\overline{b},j_b) &= (\overline{b},i_b)(a,j_a),\\
(\overline{a},j_a)(b,i_b) &= (b,j_b)(\overline{a},i_a),&\qquad
(\overline{b},j_b)(\overline{a},j_a) &= (\overline{a},i_a)(\overline{b},i_b)
\end{align*}

That these factorization rules satisfy the associativity condition of \cite{HRSW} follows from a routine calculation
using that the $P^2_{\ell_\star}$ is a quotient of of a direct product of trees.

To proceed, we define matrices
\[
D_i := \begin{pmatrix} 2m_i&0\\ 0&2m_i \end{pmatrix}
    \quad\text{ and }\quad
T_i := \begin{pmatrix} 0&2m_i\\ 2m_i&0 \end{pmatrix}.
\]
Observe that in $C_u$ the adjacency matrices $M(C)_i$ are equal to $T_i$, while in $C_m$, we have
$M(C)_1 = D_1$ and $M(C)_i = T_i$ for $2 \le i \le k$. Routine calculations using this show that
\begin{align*}
\mathbb{Z} (C_u)_o / \Big(\sum^k_{i=1} \operatorname{im}(I - M(C_u)^t_i)\Big)
    &\cong \mathbb{Z}/\gcd(4m_1^2-1, \dots, 4m_k^2-1)\mathbb{Z},\quad\text{ and}\\
\mathbb{Z} (C_m)_o / \Big(\sum^k_{i=1} \operatorname{im}(I - M(C_m)^t_i)\Big)
    &\cong \mathbb{Z}/\gcd(2m_1 - 1, 4m_2^2-1, \dots, 4m_k^2 - 1)\mathbb{Z}.
\end{align*}

\begin{corollary}
There are at least two non-isomorphic groups $\mathscr{G}(C)$ for each $k \ge 2$ and infinitely many choices of alphabets $Y_1,\ldots,Y_k$ of sizes
$2m,\ldots,2m$.
\end{corollary}
\begin{proof}
Take $m_1 = \cdots = m_k$ in the examples above. We obtain $H_0(D^{C_u}_*, \partial_*) \cong
\mathbb{Z}/(4m^2-1)\mathbb{Z}$, and since $2m-1$ divides $4m^2-1$ we have $H_0(D^{C_m}_*,
\partial_*) \cong \mathbb{Z}/(2m-1)\mathbb{Z}$. The result then follows from Corollary~\ref{cor:H*
invariants}.
\end{proof}

Note that the known groups $nV$ can be presented in our language using an $n$-graph $C$ with $1
\times 1$ adjacency matrices with single entry equal $2$. It is relatively easy to check that each
$\ker(\partial_j)$ in Evans' complex is a subgroup of a direct sum of kernels of the maps $I -
M(C)^t_i$, which in this instance are all equal to the $1 \times 1$ matrix $(-1)$. So $H_j(C) = 0$
for $1 \le j \le n$. Also, $H_0(C) = \mathbb{Z}/\big(\sum^n_{i=1} I - M(C)^t_i\mathbb{Z}) =
\mathbb{Z}/(-1)\mathbb{Z} = 0$. Hence the homology groups of all of these $n$-graphs are trivial.
In particular none of the groups $\mathscr{G}(C)$ discussed above are isomorphic to the groups
$nV$.

\subsection{Examples with nontrivial $k$th homology}

Fix integers $k, R \ge 1$. We will construct a $k$-graph $C_{k, R}$ by specifying its
skeleton and factorization rules as in \cite{HRSW}.

The vertex set $V$ of the skeleton $E$ has $R+1$ elements (for example, we could take $V = \{1, \dots, R+1\}$,
but to lighten notation, we will avoid choosing a particular enumeration).

For each $i \le k$, the set of edges of $E$ of colour $i$ is
\[
\{e^i_{v, m, w} : v \not=w \in V \text{ and } 1 \le m \le 2\}
    \cup \{e^i_{v, m, v} : v \in V \text{ and } 1 \le m \le 3\}
\]
and the range and domain maps are given by $\mathbf{r}(e^i_{v, m, w}) = v$ and $\mathbf{d}(e^i_{v,
m, w}) = w$.

So for any two distinct vertices $v,w$, there are $2$ edges of each colour pointing from $v$ to $w$
(and also from $w$ to $v$), and there are $3$ loops of each colour at each vertex.

For example, in the skeleton of $C_{k, 2}$, each of the singly-coloured subgraphs is as in
Figure~\ref{fig:skeleton}.
\begin{figure}[!htbp]
\[
\begin{tikzpicture}
    \node[inner sep=0.5pt] (v) at (0,0) {$v$};
    \node[inner sep=0.5pt] (w) at (4,0) {$w$};
    \draw[->-, out=20, in=160] (v) to node[pos=0.5, below,inner sep=0pt] {\small$e^i_{w,1,v}$} (w);
    \draw[->-, out=45, in=135] (v) to node[pos=0.5, above, inner sep=0pt] {\small$e^i_{w,2,v}$} (w);
    \draw[->-, out=200, in=340] (w) to node[pos=0.5, above,inner sep=0pt] {\small$e^i_{v,1,w}$} (v);
    \draw[->-, out=225, in=315] (w) to node[pos=0.5, below, inner sep=0pt] {\small$e^i_{v,2,w}$} (v);
    \draw[->-] (v) .. controls +(-1, 1) and +(1, 1) .. (v) node[pos=0.5, above,] {\small$e^i_{v,1,v}$};
    \draw[->-] (v) .. controls +(-1, 1) and +(-1, -1) .. (v) node[pos=0.5, left] {\small$e^i_{v,2,v}$};
    \draw[->-] (v) .. controls +(-1, -1) and +(1, -1) .. (v) node[pos=0.5, below] {\small$e^i_{v,3,v}$};
    \draw[->-] (w) .. controls +(-1, 1) and +(1, 1) .. (w) node[pos=0.5, above] {\small$e^i_{v,1,v}$};
    \draw[->-] (w) .. controls +(1, 1) and +(1, -1) .. (w) node[pos=0.5, right] {\small$e^i_{v,2,v}$};
    \draw[->-] (w) .. controls +(-1, -1) and +(1, -1) .. (w) node[pos=0.5, below] {\small$e^i_{v,3,v}$};
\end{tikzpicture}
\]
\caption{One of the singly-coloured subgraphs in the skeleton of $C_{k,2}$}\label{fig:skeleton}
\end{figure}
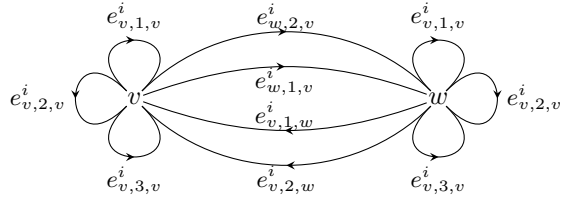

We must now specify factorization rules. For $j \not= l \le k$, the $jl$-coloured paths are those
of the form $e^j_{u,m,v}e^l_{v,n,w}$, and we must specify factorization rules that determine range
and source preserving bijections between $jl$-coloured paths and $lj$-coloured paths and that
satisfy the associativity condition \cite[Equation~(3.2)]{HRSW}.

We define them in two cases:
\begin{itemize}
\item[(F1)] if $u,v \in V$ are distinct, then we define $e^i_{u, m, u}e^j_{u, n, v} = e^j_{u, n, v}e^i_{v, m, v}$.
\item[(F2)] if either $u = v = w$ or $u \not= v$ and $v \not= w$, we define $e^i_{u, m, v}e^j_{v, n, w} = e^j_{u, n, v}e^i_{v, m, w}$.
\end{itemize}
It is routine to check that these factorization rules determine a complete collection of squares
as in \cite[p.578]{HRSW}. Routine but tedious calculations also verify the associativity condition
\cite[Equation~(3.2)]{HRSW}.

By \cite[Theorem~4.4]{HRSW} there is a $k$-graph $C_{k,R}$ whose skeleton is the coloured graph
$E$, and whose factorization rules are given by (F1)~and~(F2). Since for all $v,w \in V$ we have
$v (C_{k,R})_{e_1} w \not= \emptyset$, it is immediate that $C_{k,R}$ is cofinal. To see that it
is aperiodic, let $B_{R+1}$ be the $1$-graph whose skeleton is a bouquet of $R+1$ loops at a
single vertex. Fix a vertex $v \in V$ and observe that the sub-$k$-graph generated by the edges
$\{e^i_{(v,m,v)} : i \le k\text{ and }m \le R+1\}$ is isomorphic to the cartesian product
$\prod^k_{i-1} B_{R+1}$ of $k$ copies of $B_{R+1}$. Since the $1$-graph $B_{R+1}$ is aperiodic, so
is the $k$-graph $\prod^k_{i-1} B_{R+1}$, and so it contains an aperiodic infinite path. This is
then an aperiodic infinite path in $C_{k, R}$.

We now observe that for each $i \le k$, the matrix $(1 - M(C_{k, R})^t_i)$ is the $(R+1) \times (R+1)$
integer matrix
\[
A_R := \left(\begin{matrix}
    -2 & -2 & \cdots & -2 \\
    -2 & -2 &  \cdots & -2 \\
    \vdots & \vdots & &\vdots \\
    -2 & -2 & \cdots & -2
\end{matrix}\right).
\]
Consequently $H_k(C_{k, R}) = \ker(A_R) \cong \mathbb{Z}^R$.

We have now proved the following.

\begin{corollary}
For each positive integer $k$, there exists a family $\{C_{k, R} : R \ge 2\}$ of $k$-graphs, such
that $\mathscr{G}(C_{k, R}) \cong \mathscr{G}(C_{k', R'})$ only if $k = k'$ and $R = R'$.
Moreover, for any $k, R$, the group $\mathscr{G}(C_{k,R})$ is not isomorphic to $\mathscr{G(C')}$
for any $l$-graph $C'$ with $l < k$.
\end{corollary}


\end{document}